\documentclass[reqno]{amsart}
\usepackage{hyperref}

\AtBeginDocument{{\noindent\small
\emph{}}
\vspace{8mm}}
\begin{document}
\title[\hfil fractional Schr\"{o}dinger-Poisson system]
{Concentrating phenomenon for fractional nonlinear Schr\"{o}dinger-Poisson system with critical nonlinearity}

\author[K. M. Teng ]
{Kaimin Teng}  
\address{Kaimin Teng (Corresponding Author)\newline
Department of Mathematics, Taiyuan
University of Technology, Taiyuan, Shanxi 030024, P. R. China}
\email{tengkaimin2013@163.com}


\begin{abstract}
In this paper, we study the following fractional Schr\"{o}dinger-Poisson system
\begin{equation*}
\left\{
  \begin{array}{ll}
    \varepsilon^{2s}(-\Delta)^su+V(x)u+\phi u=g(u) & \hbox{in $\mathbb{R}^3$,} \\
     \varepsilon^{2t}(-\Delta)^t\phi=u^2,\,\, u>0& \hbox{in $\mathbb{R}^3$,}
  \end{array}
\right.
\end{equation*}
where $s,t\in(0,1)$, $\varepsilon>0$ is a small parameter. Under some suitable assumptions on potential function $V(x)$ and critical nonlinearity term $g(u)$, we construct a family of positive solutions $u_{\varepsilon}\in H^s(\mathbb{R}^3)$ which concentrates around the global minima of $V$ as $\varepsilon\rightarrow0$.
\end{abstract}

\maketitle
\numberwithin{equation}{section}
\newtheorem{theorem}{Theorem}[section]
\newtheorem{lemma}[theorem]{Lemma}
\newtheorem{definition}[theorem]{Definition}
\newtheorem{remark}[theorem]{Remark}
\newtheorem{proposition}[theorem]{Proposition}
\newtheorem{corollary}[theorem]{Corollary}
\allowdisplaybreaks

\section{Introduction}
In this paper, we study the following fractional Schr\"{o}dinger-Poisson system
\begin{equation}\label{main}
\left\{
  \begin{array}{ll}
    \varepsilon^{2s}(-\Delta)^su+V(x)u+\phi u=g(u) & \hbox{in $\mathbb{R}^3$,} \\
     \varepsilon^{2t}(-\Delta)^t\phi=u^2,\,\, u>0& \hbox{in $\mathbb{R}^3$,}
  \end{array}
\right.
\end{equation}
where $s,t\in(0,1)$ and $\varepsilon>0$ is a small parameter. The potential $V:\mathbb{R}^3\rightarrow\mathbb{R}$ is a bounded continuous function satisfying\\
$(V_0)$ $\inf\limits_{x\in\mathbb{R}^3}V(x)=V_0>0$;\\
$(V_1)$ There is a bounded domain $\Lambda\subset\mathbb{R}^3$ such that
\begin{equation*}
V_0<\min_{\partial\Lambda}V(x),\quad \mathcal{M}=\{x\in\Lambda\,\,|\,\, V(x)=V_0\}\neq{\O}.
\end{equation*}
Without of loss of generality, we may assume that $0\in\mathcal{M}$. The nonlinearity $g:\mathbb{R}\rightarrow\mathbb{R}$ is of $C^1$-class function. The non-local operator $(-\Delta)^s$ ($s\in(0,1)$), which is called fractional Laplace operator, can be defined by
\begin{equation*}
(-\Delta)^su(x)=C_s\,P.V.\int_{\mathbb{R}^3}\frac{u(x)-u(y)}{|x-y|^{3+2s}}\,{\rm d}y=C_s\lim_{\varepsilon\rightarrow0}\int_{\mathbb{R}^3\backslash B_{\varepsilon}(x)}\frac{u(x)-u(y)}{|x-y|^{3+2s}}\,{\rm d}y
\end{equation*}
for $u\in\mathcal{S}(\mathbb{R}^3)$, where $\mathcal{S}(\mathbb{R}^3)$ is the Schwartz space of rapidly decaying $C^{\infty}$ function, $B_{\varepsilon}(x)$ denote an open ball of radius $r$ centered at $x$ and the normalization constant $C_s=\Big(\int_{\mathbb{R}^3}\frac{1-\cos(\zeta_1)}{|\zeta|^{3+2s}}\,{\rm d}\zeta\Big)^{-1}$. Fractional Laplacian appears in lots of real world, such as: fractional quantum mechanics \cite{La,La1}, anomalous diffusion \cite{MK}, financial \cite{CT}, obstacle problems \cite{S}, conformal geometry and minimal surfaces \cite{CM}. In the very recent years, the progress of nonlinear equations involving fractional Lapalcian can be found in \cite{AM,BCPS,BV,CS,CSi,CT,CW,DMV,DPDV,DPW,FL,HZ1,NPV,PP,S,SZ,Teng-1,Teng-2} and so on.

For $u\in\mathcal{S}(\mathbb{R}^3)$, the fractional Laplace operator $(-\Delta)^s$ can be expressed as an inverse Fourier transform
\begin{equation*}
(-\Delta)^su=\mathcal{F}^{-1}\Big((2\pi|\xi|)^{2s}\mathcal{F}u(\xi)\Big),
\end{equation*}
where $\mathcal{F}$ and $\mathcal{F}^{-1}$ denote the Fourier transform and inverse transform, respectively. If $u$ is sufficiently smooth, it is known that (see \cite{NPV}) it is equivalent to
\begin{equation*}
(-\Delta)^su(x)=-\frac{1}{2}C_s\int_{\mathbb{R}^3}\frac{u(x+y)+u(x-y)-2u(y)}{|x-y|^{3+2s}}\,{\rm d}y.
\end{equation*}
By a classical solution of \eqref{main}, we mean two continuous functions that $(-\Delta)^su$ is well defined for all $x\in\mathbb{R}^3$ and satisfies \eqref{main} in pointwise sense.

Since we are looking for positive solutions, we may assume that $g(s)=0$ for $s<0$. Furthermore, we need the following conditions:\\
$(g_0)$ $\lim\limits_{\tau\rightarrow0^{+}}\frac{g(\tau)}{\tau}=0$;\\
$(g_1)$ $\lim\limits_{\tau\rightarrow+\infty}\frac{g(\tau)}{\tau^{2_s^{\ast}-1}}=\kappa>0$;\\
$(g_2)$ there exists $\lambda>0$ such that $g(\tau)\geq\lambda \tau^{q-1}+\tau^{2_s^{\ast}-1}$ for some $\frac{4s+2t}{s+t}<q<2_s^{\ast}$ and all $\tau\geq0$.

The hypotheses $(g_0)$-$(g_2)$ are so-called the critical Berestycki-Lions type conditions, which was introduced in \cite{ZCZ}. For simplicity, we may assume that $\kappa=1$ and  $g(\tau)=f(\tau)+|\tau|^{2_s^{\ast}-2}\tau$, for $\tau>0$. Then system \eqref{main} is equivalent to the following one
\begin{equation}\label{main-0}
\left\{
  \begin{array}{ll}
    \varepsilon^{2s}(-\Delta)^su+V(x)u+\phi u=f(u)+u^{2_s^{\ast}-1} & \hbox{in $\mathbb{R}^3$,} \\
     \varepsilon^{2t}(-\Delta)^t\phi=u^2,\,\, u>0& \hbox{in $\mathbb{R}^3$}
  \end{array}
\right.
\end{equation}
where $f$ satisfies \\
$(f_0)$ $\lim\limits_{\tau\rightarrow0^{+}}\frac{f(\tau)}{\tau}=0$;\\
$(f_1)$ $\liminf\limits_{\tau\rightarrow+\infty}\frac{f'(\tau)}{\tau^{2_s^{\ast}-2}}=0$;\\
$(f_2)$ there exists $\lambda>0$ such that $f(\tau)\geq\lambda \tau^{q-1}$, for $\tau>0$ and some $q\in(\frac{4s+2t}{s+t},2_s^{\ast})$.

In the very recent years, the study of the existence, concentration and multiplicity of positive solutions for fractional Schr\"{o}dinger-Poisson system \eqref{main} is just starting. When $\varepsilon=1$, by using the Nehari-Pohozaev manifold combing monotone trick with global compactness Lemma, Teng \cite{Teng2} studied the existence of positive ground state solution for the system
\begin{equation}\label{main1}
\left\{
  \begin{array}{ll}
    (-\Delta)^su+V(x)u+\phi u=|u|^{p-1}u+|u|^{2_s^{\ast}-2}u & \hbox{in $\mathbb{R}^3$,} \\
    (-\Delta)^t\phi=u^2& \hbox{in $\mathbb{R}^3$.}
  \end{array}
\right.
\end{equation}
Using the similar methods, in \cite{Teng3}, positive ground state solutions for problem \eqref{main1} with $|u|^{p-1}u+|u|^{2_s^{\ast}-2}u$ replaced by $|u|^{p-1}u$ with $p\in(2,2_s^{\ast}-1)$, were established when $s=t$. In \cite{ZDS}, the authors studied the existence of radial solutions for system \eqref{main1} with  $|u|^{p-1}u+|u|^{2_s^{\ast}-2}u$ replaced by $f(u)$, where the nonlinearity $f(u)$ verifies the subcritical or critical assumptions of Berestycki-Lions type. When $0<\varepsilon<1$ small, in \cite{MS}, the authors studied the semiclassical state of the following system
\begin{equation*}
\left\{
  \begin{array}{ll}
    \varepsilon^{2s}(-\Delta)^su+V(x)u+\phi u=f(u) & \hbox{in $\mathbb{R}^N$,} \\
    \varepsilon^{\theta}(-\Delta)^{\frac{\alpha}{2}}\phi=\gamma_{\alpha}u^2& \hbox{in $\mathbb{R}^N$,}
  \end{array}
\right.
\end{equation*}
where $s\in(0,1)$, $\alpha\in(0,N)$, $\theta\in(0,\alpha)$, $N\in(2s,2s+\alpha)$, $\gamma_{\alpha}$ is a positive constant, $f(u)$ satisfies the following subcritical growth assumptions: $0<KF(t)\leq f(t)t$ with some $K>4$ for all $t\geq0$ and $\frac{f(t)}{t^3}$ is strictly increasing on $(0,+\infty)$. In \cite{LZ}, by using the methods mentioned before, Liu and Zhang proved the existence and concentration of positive ground state solution for problem \eqref{main-0}. When the system \eqref{main-0} verifying that $s=t$ and $f(u)+u^{2_s^{\ast}-1}$ replaced by $K(x)|u|^{p-2}u$ which $V$ has positive global minimum and $K(x)$ has global maximum, in \cite{YZZ}, the authors prove the existence of a positive ground state for $\varepsilon>0$ sufficiently small and concentration behavior of these ground state solutions as $\varepsilon\rightarrow0$. In \cite{Teng2}, we studied the system \eqref{main} with competing potential, i.e., $g(u)=K(x)f(u)+Q(x)|u|^{2_s^{\ast}-2}u$, where $f$ is a function of $C^1$ class, superlinear and subcritical nonlinearity, $V(x)$, $K(x)$ and $Q(x)$ are positive continuous functions. Under some suitable assumptions on $V$, $K$ and $Q$, we prove that there is a family of positive ground state solutions which concentrate on the set of minimal points of $V(x)$ and the sets of maximal points of $K(x)$ and $Q(x)$. For the local assumption on the potential $V(x)$, Teng \cite{Teng3} firstly applied the penalization methods developed by \cite{DF} to study the concentration phenomenon of system \eqref{main-1} under the hypotheses made on $V(x)$ and $f$:
\begin{itemize}
  \item $\inf\limits_{x\in\mathbb{R}^3}V(x)=\alpha_0>0$ and there is a bounded domain $\Lambda\subset\mathbb{R}^3$ such that
\begin{equation*}
V_0=\inf_{\Lambda}V(x)<\min_{\partial\Lambda}V(x);
\end{equation*}
  \item $\lim\limits_{\tau\rightarrow0^{+}}\frac{f(\tau)}{\tau^3}=0$, there exist $\lambda>0$ and $C>0$ such that $f(\tau)\geq\lambda \tau^{q-1}$ for some $4\leq q<2_s^{\ast}$ and $|f'(\tau)|\leq C(1+|\tau|^{p-2})$, where $4<p<2_s^{\ast}$;
  \item $\frac{f(\tau)}{\tau^{3}}$ is non-decreasing in $\tau\in (0,+\infty)$.
\end{itemize}
The penalization methods were applied to fractional Schr\"{o}dinger equations, please see \cite{AM,A,HZ1}. For extending our result in \cite{Teng3}, through modifying the penalization methods developed by Byeon and Wang \cite{BW}, Teng \cite{Teng4} studied the concentration behavior of system \eqref{main} with $V(x)$ satisfying $(V_0)$-$(V_1)$ and $g$ verifying
\begin{itemize}
\item $\lim\limits_{\tau\rightarrow0^{+}}\frac{g(\tau)}{\tau}=0$, $\lim\limits_{\tau\rightarrow+\infty}\frac{g'(\tau)}{\tau^{2_s^{\ast}-2}}=0$;
\item  there exists $\lambda>0$ such that $g(\tau)\geq\lambda \tau^{q-1}$ for some $\frac{4s+2t}{s+t}<q<2_s^{\ast}$ and all $\tau\geq0$;
\item $\frac{g(\tau)}{\tau^{q-1}}$ is non-decreasing in $\tau\in (0,+\infty)$.
\end{itemize}

When $s=1$, system \eqref{main} reduces to the following Schr\"{o}dinger-Poisson system
\begin{equation}\label{main-1}
\left\{
  \begin{array}{ll}
    -\varepsilon^2\Delta u+V(x)u+\phi u=g(x,u) & \hbox{in $\mathbb{R}^3$,} \\
    -\varepsilon^2\Delta \phi=u^2& \hbox{in $\mathbb{R}^3$.}
  \end{array}
\right.
\end{equation}
In recent years, there has been increasing attention to \eqref{main-1} on the existence of positive solutions, ground state solutions, multiple solutions and semiclassical states, see for example \cite{AR,BF,DW,HZ,Ruiz,Ruiz1,ZLZ} and the references therein. It is well known that system \eqref{main-1} appearing in quantum mechanics models (see e.g. \cite{LS}) and in semiconductor theory \cite{MRS}. Especially, systems like \eqref{main-1} have been introduced in \cite{BF} as a model to describe solitary waves. Regarding the concentration phenomenon of solutions for Schr\"{o}dinger-Poisson systems like \eqref{main-1}, there has been the object of interest for many scholars. In \cite{H}, the author studied the system \eqref{main-1} with $g(x,v)=f(v)$ satisfying
\begin{itemize}
  \item $\frac{f(t)}{t^3}$ increasing in $(0,\infty)$, $\exists\theta>4$ such that $0<\theta F(t)=\int_0^tf(s)\,{\rm d}s\leq tf(t)$ for all $t>0$,
  \item $f'(t)t^2-3f(t)t\geq Ct^{\sigma}$, $\sigma\in(4,6)$, and $f(t)=o(t^3)$ as $t\rightarrow0$.
\end{itemize}
By using Ljusternik-Schnirelmann theory and minimax method, the author obtained the multiplicity of positive solutions for $\varepsilon>0$ small which concentrate on the minima of $V(x)$. In \cite{WTXZ}, Wang et al. studied the existence and concentration of positive ground state solutions for system \eqref{main-1} with $g(x,v)=b(x)f(v)$ satisfying
\begin{itemize}
  \item $\frac{f(t)}{t^3}$ increasing in $(0,\infty)$, $\frac{F(t)}{t^4}\rightarrow+\infty$ as $t\rightarrow\infty$,
  \item $|f(t)|\leq c(1+|t|^{p-1})$, $p\in(4,6)$, and $f(t)=o(t^3)$ as $t\rightarrow0$.
\end{itemize}
In the critical case, He and Zou \cite{HZ} considered system \eqref{main-1} with $g(x,v)=v^5+ f(v)$, where $f$ satisfies the similar hypotheses as \cite{H}, proved that system \eqref{main-1} has a ground state solution concentrating around a global minimum of $V(x)$ as $\varepsilon\rightarrow0$. In \cite{WTXZ1}, the authors studied the system \eqref{main-1} with $g(x,v)=b(x)f(v)+|v|^4v$, where $f$ satisfies
\begin{itemize}
  \item $\frac{f(t)}{t^3}$ increasing in $(0,\infty)$, $f(t)=o(t^3)$ as $t\rightarrow0$,
  \item $f(t)\geq ct^{q-1}$, $|f(t)|\leq c(1+|t|^{p-1})$, $4<q\leq p<6$.
\end{itemize}
Under some suitable assumptions on $V(x)$ and $b(x)$, Wang et al. \cite{WTXZ1} proved the existence of least energy solution $(u_{\varepsilon},\phi_{\varepsilon})$ and then showed that $u_{\varepsilon}$ converges to the least energy solution of the associated limit problem and concentrates to some set in $\mathbb{R}^3$ depending on the potentials $V$ and $b$.
The above assumptions made on the potential $V(x)$ is global, for the local assumption, there are few results obtained in the literature. As far as we know, only in \cite{HL} studied the Schr\"{o}dinger-Poisson system \eqref{main-1} with $V(x)$ satisfying the local condition $\inf\limits_{\Lambda}V(x)<\inf\limits_{\partial\Lambda}V(x)$ and $g(x,v)=\lambda|v|^{p-2}v+|v|^4v$ for $3< p\leq4$, where $\Lambda$ is an open set of $\mathbb{R}^3$ and $\lambda>0$, the authors constructed a family of positive solutions which concentrates around a local minimum of $V$ as $\varepsilon\rightarrow0$.

The semiclassical state of the following Schr\"{o}dinger-Poisson system
\begin{equation}\label{main-1-1}
\left\{
  \begin{array}{ll}
    -\varepsilon^2\Delta u+V(x)u+K(x)\phi u=u^p & \hbox{in $\mathbb{R}^3$,} \\
    -\Delta \phi=u^2& \hbox{in $\mathbb{R}^3$}
  \end{array}
\right.
\end{equation}
has attracted many scholars' attention. When $p\in(1,5)$, Ruiz and Vaira \cite{RV} proved the existence of multi-bump solutions of system and these bumps concentrate around a local minimum of the potential $V$. Ianni and Vaira \cite{IV} obtained the existence of positive bound state solutions which concentrate on a non-degenerate local minimum or maximum of $V$ by using a Lyapunov-Schmitt reduction method. Ianni and Vaira \cite{IV1} also showed the existence of radially symmetric solutions, which concentrate on the spheres. For the critical case, for system \eqref{main-1-1} with $u^p$ replaced by $f(u)+u^5$, in \cite{LGF}, the authors proved the multiplicity of positive solutions and the number of positive
solutions depends on the profile of the potential and that each solution concentrates around its corresponding global minimum point of the potential in the semiclassical limit.
For the local assumptions on the potential $V(x)$, Seok \cite{S} studied the system \eqref{main-1-1} with $u^p$ replaced by $f(u)$ satisfying
\begin{itemize}
  \item $f(t)=o(t)$ as $t\rightarrow0$, $\lim\limits_{t\rightarrow\infty}\frac{f(t)}{t^p}<\infty$ for some $p\in(1,5)$,
  \item $\exists T>0$ such that $\frac{1}{2}mT^2<F(T)$, $F(t)=\int_0^tf(s)\,{\rm d}s$
\end{itemize}
and proved the existence of the spike solutions through following a variational approach developed by Byeon-Jeanjean \cite{BJ,BJ1}. Using the similar ideas as Byeon-Jeanjean \cite{BJ}, Zhang \cite{Z} considered the system \eqref{main-1-1} with $u^p$ replaced by a general nonlinearity $f(u)$ satisfying the critical growth assumptions
\begin{itemize}
  \item $f(t)=o(t)$ as $t\rightarrow0$, $\lim\limits_{t\rightarrow\infty}\frac{f(t)}{t^5}=\kappa>0$,
  \item $\exists C>0$ and $p<6$ such that $f(t)\geq\kappa t^5+Ct^{p-1}$ for $t\geq0$
\end{itemize}
and constructed a solution $(u_{\varepsilon},\phi_{\varepsilon})$, which concentrates at an isolated component of positive locally minimum points of $V$ as $\varepsilon\rightarrow0$.

From the above known results, we see that the monotonic hypothesis $\frac{f(t)}{t^3}$ is necessary to study the concentration behavior of system \eqref{main-1} whatever critical case or subcritical case. The purpose of this paper is to weak this monotonic hypothesis to the following one:\\
$(f_3)$ $\frac{f(\tau)}{\tau^{q-1}}$ is non-decreasing in $\tau\in (0,+\infty)$, where $q\in(\frac{4s+2t}{s+t},2_s^{\ast})$.

To the best of our knowledge, except \cite{HL}, there are few papers to study the concentration phenomenon of Schr\"{o}dinger-Poisson system \eqref{main-1} with local assumption on the potential $V(x)$, not mention to the fractional Schr\"{o}dinger-Poisson system \eqref{main}. Motivated by the above cited papers, the goal of this paper is to study the existence and concentration of positive bound state solutions for system \eqref{main-0} under $(V_0)$-$(V_1)$ and $(f_0)$-$(f_3)$.

Our main results is as follows.
\begin{theorem}\label{thm1-1}
Let $2s+2t>3$, $s,t\in(0,1)$. Suppose that $V$ satisfies $(V_0)$, $(V_1)$ and $g\in C(\mathbb{R}^{+},\mathbb{R})$ satisfies $(g_0)$--$(g_3)$. Then there exists an $\varepsilon_0>0$ such that system \eqref{main} possesses a positive solution $(u_{\varepsilon},\phi_{\varepsilon})\in H_{\varepsilon}\times\mathcal{D}^{t,2}(\mathbb{R}^3)$ for all $\varepsilon\in(0,\varepsilon_0)$. Moreover, there exists a maximum point $x_{\varepsilon}$ of $u_{\varepsilon}$ such that $\lim\limits_{\varepsilon\rightarrow0}{\rm dist}(x_{\varepsilon},\mathcal{M})=0$ and
\begin{equation*}
u_{\varepsilon}(x)\leq\frac{C\varepsilon^{3+2s}}{C_0\varepsilon^{3+2s}+|x-x_{\varepsilon}|^{3+2s}}\quad x\in\mathbb{R}^3,\,\,\text{and}\,\, \varepsilon\in(0,\varepsilon_0)
\end{equation*}
for some constants $C>0$ and $C_0\in\mathbb{R}$.
\end{theorem}
We will give some comments on our main result.
\begin{remark}
The hypothesis $(V_1)$ is a special case of the local assumption
\begin{equation*}
\inf_{\Lambda}V(x)<\inf_{\partial\Lambda}V(x)
\end{equation*}
which first introduced by M. del Pino and P. L. Felmer \cite{DF}, because there have not some local priori estimates like Theorem 8.17 in \cite{GT}.
\end{remark}

\begin{remark}
Comparing with the results in \cite{H,HZ,WTXZ,WTXZ1}, the monotone hypothesis $(f_3)$ is weaker even in the case $s=t=1$ ($q-1>\frac{4s+2t}{s+t}-1=2$).
\end{remark}

\begin{remark}
The condition $(V_2)$ is local and the $(AR)$-condition for fractional Schr\"{o}dinger-Poisson system is not satisfied, we need to modify the penalization methods developed by J. Byeon, Z. Q. Wang \cite{BJ,BW} and combine the penalization methods introduced by M. del Pino, P. L. Felmer \cite{DF}, for overcoming the obstacle caused by the non-compactness due to the unboundedness of the domain and the lack of $(AR)$ condition.
\end{remark}

The paper is organized as follows, in Section 2, we give some preliminary results. In Section 3, we prove the existence of positive ground state solutions for "limit problem". In Section 4, we prove the main result Theorem \ref{thm1-1}.

\section{Variational Setting}

In this section, we outline the variational framework for studying problem \eqref{main-0} and list some preliminary Lemma which used later. In the sequel, we denote by $\|\cdot\|_{p}$ the usual norm of the space $L^p(\mathbb{R}^3)$, the letter $c_i$ ($i=1,2,\ldots$) or $C$ denote by some positive constants.

\subsection{Work space stuff}
We define the homogeneous fractional Sobolev space $\mathcal{D}^{\alpha,2}(\mathbb{R}^3)$ as follows
\begin{equation*}
\mathcal{D}^{\alpha,2}(\mathbb{R}^3)=\Big\{u\in L^{2_{\alpha}^{\ast}}(\mathbb{R}^3)\,\,\Big|\,\,|\xi|^{\alpha}(\mathcal{F}u)(\xi)\in L^2(\mathbb{R}^3)\Big\}
\end{equation*}
which is the completion of $C_0^{\infty}(\mathbb{R}^3)$ under the norm
\begin{equation*}
\|u\|_{\mathcal{D}^{\alpha,2}}=\Big(\int_{\mathbb{R}^3}|(-\Delta)^{\frac{\alpha}{2}}u|^2\,{\rm d}x\Big)^{\frac{1}{2}}=\Big(\int_{\mathbb{R}^3}|\xi|^{2\alpha}|(\mathcal{F}u)(\xi)|^2\,{\rm d}\xi\Big)^{\frac{1}{2}}
\end{equation*}
The fractional Sobolev space $H^{\alpha}(\mathbb{R}^3)$ can be described by means of the Fourier transform, i.e.
\begin{equation*}
H^{\alpha}(\mathbb{R}^3)=\Big\{u\in L^2(\mathbb{R}^3)\,\,\Big|\,\,\int_{\mathbb{R}^3}(|\xi|^{2\alpha}|(\mathcal{F}u)(\xi)|^2+|(\mathcal{F}u)(\xi)|^2)\,{\rm d}\xi<+\infty\Big\}.
\end{equation*}
In this case, the inner product and the norm are defined as
\begin{equation*}
(u,v)=\int_{\mathbb{R}^3}(|\xi|^{2\alpha}(\mathcal{F}u)(\xi)\overline{(\mathcal{F}v)(\xi)}+(\mathcal{F}u)(\xi)\overline{(\mathcal{F}v)(\xi)})\,{\rm d}\xi
\end{equation*}
and
\begin{equation*}
\|u\|_{H^{\alpha}}=\bigg(\int_{\mathbb{R}^3}(|\xi|^{2\alpha}|(\mathcal{F}u)(\xi)|^2+|(\mathcal{F}u)(\xi)|^2)\,{\rm d}\xi\bigg)^{\frac{1}{2}}.
\end{equation*}
From Plancherel's theorem we have $\|u\|_2=\|\mathcal{F}u\|_2$ and $\||\xi|^{\alpha}\mathcal{F}u\|_2=\|(-\Delta)^{\frac{\alpha}{2}}u\|_2$. Hence
\begin{equation*}
\|u\|_{H^{\alpha}}=\bigg(\int_{\mathbb{R}^3}(|(-\Delta)^{\frac{\alpha}{2}}u(x)|^2+|u(x)|^2)\,{\rm d}x\bigg)^{\frac{1}{2}},\quad \forall u\in H^{\alpha}(\mathbb{R}^3).
\end{equation*}
We denote $\|\cdot\|$ by $\|\cdot\|_{H^{\alpha}}$ in the sequel for convenience.

In terms of finite differences, the fractional Sobolev space $H^{\alpha}(\mathbb{R}^3)$ also can be defined as follows
\begin{equation*}
H^{\alpha}(\mathbb{R}^3)=\Big\{u\in L^2(\mathbb{R}^3)\,\,\Big|\,\,D_{\alpha}u\in L^2(\mathbb{R}^3)\Big\},\quad |D_{\alpha}u|^2=\int_{\mathbb{R}^3}\frac{|u(x)-u(y)|^2}{|x-y|^{3+2\alpha}}\,{\rm d}y
\end{equation*}
endowed with the natural norm
\begin{equation*}
\|u\|_{H^{\alpha}}=\bigg(\int_{\mathbb{R}^3}|u|^2\,{\rm d}x+\int_{\mathbb{R}^3}|D_{\alpha}u|^2\,{\rm d}x\bigg)^{\frac{1}{2}}.
\end{equation*}

Also, in view of Proposition 3.4 and Proposition 3.6 in \cite{NPV}, we have
\begin{equation}\label{equ2-1}
\|(-\Delta)^{\frac{\alpha}{2}}u\|_2^2=\int_{\mathbb{R}^3}|\xi|^{2\alpha}|(\mathcal{F}u)(\xi)|^2\,{\rm d}\xi=\frac{1}{C_{\alpha}}\int_{\mathbb{R}^3}|D_{\alpha}u|^2\,{\rm d}x.
\end{equation}

We define the Sobolev space $H_{\varepsilon}=\{u\in H^s(\mathbb{R}^3)\,\,|\,\, \int_{\mathbb{R}^3}V(\varepsilon x)u^2\,{\rm d}x<\infty\}$ endowed with the norm
\begin{equation*}
\|u\|_{H_{\varepsilon}}=\Big(\int_{\mathbb{R}^3}(|D_su|^2+V(\varepsilon x)u^2)\,{\rm d}x\Big)^{\frac{1}{2}}.
\end{equation*}

It is well known that (see \cite{LM}) $H^s(\mathbb{R}^3)$ is continuously embedded into $L^r(\mathbb{R}^3)$ for $2\leq r\leq 2_{s}^{\ast}$ ($2_{s}^{\ast}=\frac{6}{3-2s}$).
Obviously, the conclusion also holds for $H_{\varepsilon}$.

\subsection{Formulation of Problem \eqref{main-0}}
It is easily seen that, just performing the change of variables $u(x)\rightarrow u(x/\varepsilon)$ and $\phi(x)\rightarrow \phi(x/\varepsilon)$, and taking $z=x/\varepsilon$, problem \eqref{main-0} can be rewritten as the following equivalent form
\begin{equation}\label{main-2-1}
\left\{
  \begin{array}{ll}
   (-\Delta)^su+V(\varepsilon z)u+\phi u=f(u)+u^{2_s^{\ast}-1} & \hbox{in $\mathbb{R}^3$,} \\
    (-\Delta)^t\phi=u^2,\,\, u>0& \hbox{in $\mathbb{R}^3$}
  \end{array}
\right.
\end{equation}
which will be referred from now on. Observe that if $4s+2t\geq3$, there holds $2\leq\frac{12}{3+2t}\leq\frac{6}{3-2s}$ and thus $H_{\varepsilon}\hookrightarrow L^{\frac{12}{3+2t}}(\mathbb{R}^3)$. Considering $u\in H_{\varepsilon}$, the linear functional $\widetilde{\mathcal{L}}_u:\mathcal{D}^{t,2}(\mathbb{R}^3)\rightarrow\mathbb{R}$ is defined by $\widetilde{\mathcal{L}}_u(v)=\int_{\mathbb{R}^3}u^2v\,{\rm d}x$. Using the Lax-Milgram theorem, there exists a unique $\phi_u^t\in\mathcal{D}^{t,2}(\mathbb{R}^3)$ such that
\begin{align*}
C_s\int_{\mathbb{R}^3\times\mathbb{R}^3}\frac{(\phi_u^t(z)-\phi_u^t(y))(v(z)-v(y))}{|z-y|^{3+2s}}\,{\rm d}y\,{\rm d}z&=\int_{\mathbb{R}^3}(-\Delta)^{\frac{t}{2}}\phi_u^t(-\Delta)^{\frac{t}{2}}v\,{\rm d}z\\
&=\int_{\mathbb{R}^3}u^2v\,{\rm d}z,\quad \forall v\in\mathcal{D}^{t,2}(\mathbb{R}^3),
\end{align*}
that is $\phi_u^t$ is a weak solution of $(-\Delta)^t\phi_u^t=u^2$ and so the representation formula holds
\begin{equation*}
\phi_u^t(x)=c_t\int_{\mathbb{R}^3}\frac{u^2(y)}{|x-y|^{3-2t}}\,{\rm d}y,\quad x\in\mathbb{R}^3,\quad c_t=\pi^{-\frac{3}{2}}2^{-2t}\frac{\Gamma(\frac{3-2t}{2})}{\Gamma(t)}.
\end{equation*}
Substituting $\phi_u^t$ in \eqref{main-2-1}, it reduces to a single fractional Schr\"{o}dinger equation
\begin{equation}\label{R-1}
(-\Delta)^su+V(\varepsilon z)u+\phi_u^tu=f(u)+(u^{+})^{2_s^{\ast}-1}\quad z\in\mathbb{R}^3.
\end{equation}
The solvation of \eqref{R-1} can be looking for the critical points of the associated energy functional $J_{\varepsilon}: H_{\varepsilon}\rightarrow\mathbb{R}$ defined by
\begin{align*}
J_{\varepsilon}(u)&=\frac{1}{2}\int_{\mathbb{R}^3}|D_su|^2\,{\rm d}z+\frac{1}{2}\int_{\mathbb{R}^3}V(\varepsilon z)u^2\,{\rm d}z+\frac{1}{4}\int_{\mathbb{R}^3}\phi_u^tu^2\,{\rm d}z-\int_{\mathbb{R}^3}F(u)\,{\rm d}z\\
&-\frac{1}{2_s^{\ast}}\int_{\mathbb{R}^3}(u^{+})^{2_s^{\ast}}\,{\rm d}z.
\end{align*}

Let us summarize some properties of the function $\phi_u^t$. By using simple computation, it is easy to check the following conclusions.
\begin{lemma}\label{lem2-1}
For every $u\in H_{\varepsilon}$ with $4s+2t\geq3$, define $\Phi(u)=\phi_u^t\in \mathcal{D}^{t,2}(\mathbb{R}^3)$, where $\phi_u^t$ is the unique solution of equation $(-\Delta)^t\phi=u^2$. Then there hold:\\
$(i)$ If $u_n\rightharpoonup u$ in $H_{\varepsilon}$, then $\Phi(u_n)\rightharpoonup\Phi(u)$ in $\mathcal{D}^{t,2}(\mathbb{R}^3)$;\\
$(ii)$ $\Phi(tu)=t^2\Phi(u)$ for any $t\in\mathbb{R}$;\\
$(iii)$ For $u\in H_{\varepsilon}$, one has
\begin{equation*}
\|\Phi(u)\|_{\mathcal{D}^{t,2}}\leq C\|u\|_{\frac{12}{3+2t}}^2\leq C\|u\|_{H_\varepsilon}^2,\quad \int_{\mathbb{R}^3}\Phi(u)u^2\,{\rm d}x\leq C\|u\|_{\frac{12}{3+2t}}^4\leq C\|u\|_{H_\varepsilon}^4,
\end{equation*}
where constant $C$ is independent of $u$;\\
$(iv)$ Let $2s+2t>3$, if $u_n\rightharpoonup u$ in $H_{\varepsilon}$ and $u_n\rightarrow u$ a.e. in $\mathbb{R}^3$, then for any $v\in H_{\varepsilon}$,
\begin{equation*}
\int_{\mathbb{R}^3}\phi_{u_n}^tu_nv\,{\rm d}z\rightarrow\int_{\mathbb{R}^3}\phi_{u}^tuv\,{\rm d}z\quad\text{and}\quad\int_{\mathbb{R}^3}f(u_n)v\,{\rm d}z\rightarrow\int_{\mathbb{R}^3}f(u)v\,{\rm d}z
\end{equation*}
and
\begin{equation*}
\int_{\mathbb{R}^3}(u_n^{+})^{2_s^{\ast}-1}v\,{\rm d}z\rightarrow\int_{\mathbb{R}^3}(u^{+})^{2_s^{\ast}-1}v\,{\rm d}z,
\end{equation*}
and thus $u$ is a solution for problem \eqref{R-1}.
\end{lemma}

In the following, we collect some useful Lemma. We define
\begin{equation*}
\mu_{\infty}=\lim_{R\rightarrow\infty}\limsup_{n\rightarrow\infty}\int_{\{|x|>R\}}|D_su_n|^2\,{\rm d}z\quad \nu_{\infty}=\lim_{R\rightarrow\infty}\limsup_{n\rightarrow\infty}\int_{\{|x|>R\}}|u_n|^{2_s^{\ast}}\,{\rm d}z.
\end{equation*}

\begin{lemma}\label{lem2-2}(\cite{PP,ZZX})
Let $\{u_n\}\subset H^s(\mathbb{R}^3)$ be such that $u_n\rightharpoonup u$ in $\mathcal{D}^{s,2}(\mathbb{R}^3)$, $|D_su_n|^2\rightharpoonup\mu$ and $|u_n|^{2_s^{\ast}}\rightharpoonup\nu$ weakly$-\ast$ in $\mathcal{M}(\mathbb{R}^3)$ as $n\rightarrow\infty$. Here $\mathcal{M}(\mathbb{R}^3)$ is the space of finite nonnegative Borel measures on $\mathbb{R}^3$. Then\\
$(i)$ there exist a (at most countable) set of distinct points $\{x_j\}_{j\in J}\subset\mathbb{R}^3$, $\mu_j\geq0$, $\nu_j\geq0$ with $\mu_j+\nu_j>0$ ($j\in J$) such that
\begin{equation*}
\mu\geq|D_su|^2+\sum_{j\in J}\mu_j\delta_{x_j}\quad \nu=|u|^{2_s^{\ast}}+\sum_{j\in J}\nu_j\delta_{x_j},\quad \mu_j=\mu(\{x_j\}),\,\, \nu_j=\nu(\{x_j\});
\end{equation*}
$(ii)$
Then $\mu_{\infty}$ and $\nu_{\infty}$ are well defined satisfy
\begin{equation*}
\limsup_{n\rightarrow\infty}\int_{\mathbb{R}^3}|D_su_n|^2\,{\rm d}z=\int_{\mathbb{R}^3}\,{\rm d}\mu+\mu_{\infty}\quad \limsup_{n\rightarrow\infty}\int_{\mathbb{R}^3}|u_n|^{2_s^{\ast}}\,{\rm d}z=\int_{\mathbb{R}^3}\,{\rm d}\nu+\nu_{\infty};
\end{equation*}
$(iii)$
\begin{equation*}
\nu_j\leq(\mathcal{S}_s^{-1}\mu_j)^{\frac{2_s^{\ast}}{2}}\,\, \text{for any}\,\, j\in J\,\, \text{and}\,\,\nu_{\infty}\leq(\mathcal{S}_s^{-1}\mu_{\infty})^{\frac{2_s^{\ast}}{2}}.
\end{equation*}
\end{lemma}

\begin{proposition}\label{pro2-1} (\cite{Teng})
Let $\{u_n\}$ be a bounded sequence in $H^s(\mathbb{R}^3)$. If
\begin{equation*}
\lim_{n\rightarrow\infty}\sup_{y\in\mathbb{R}^3}\int_{B_R(y)}|u_n|^{2_s^{\ast}}\,{\rm d}x=0,
\end{equation*}
where $R$ is a positive number, then $u_n\rightarrow0$ in $L^{2_s^{\ast}}(\mathbb{R}^3)$ as $n\rightarrow\infty$.
\end{proposition}

\begin{proposition}\label{pro2-3}
Let $\{u_k\}\subset \mathcal{D}^{s,2}(\mathbb{R}^3)$ be a bounded sequence such that $u_k\rightharpoonup0$ in $\mathcal{D}^{s,2}(\mathbb{R}^3)$. Suppose that there exists a bounded open set $Q\subset\mathbb{R}^3$ and a positive number $\gamma>0$ such that
\begin{equation}\label{equ2-6}
\int_{Q}|u_k|^{2_s^{\ast}}\,{\rm d}x\geq\gamma>0.
\end{equation}
Moreover, suppose that
\begin{equation}\label{equ2-7}
(-\Delta)^su_k=|u_k|^{2_s^{\ast}-2}u_k-\chi_k\quad x\in\mathbb{R}^3,
\end{equation}
where $\chi_k\in (H^s(\mathbb{R}^3))'$, and $|\langle\chi_k,\varphi\rangle|\leq\varepsilon_k\|\varphi\|$ for any $\varphi\in C_0^{\infty}(V)$, where $V$ is an open neighborhood of $Q$ and $\varepsilon_k\rightarrow0$ as $k\rightarrow\infty$. Then there exist a sequence of points $\{z_k\}\in\mathbb{R}^3$ and a sequence of positive numbers $\{\sigma_k\}$ such that $v_k(x)=\sigma_k^{\frac{3-2s}{2}}u_k(\sigma_k x+z_k)$ converges weakly in $\mathcal{D}^{s,2}(\mathbb{R}^3)$ to a nontrivial solution $v$ of
\begin{equation}\label{equ2-8}
(-\Delta)^sv=|v|^{2_s^{\ast}-2}v\quad x\in\mathbb{R}^3.
\end{equation}
Moreover, $z_k\rightarrow z\in \overline{Q}$ and $\sigma_k\rightarrow0$ as $k\rightarrow\infty$.
\end{proposition}

\begin{proof}
Since $\{u_k\}$ is bounded in $\mathcal{D}^{s,2}(\mathbb{R}^3)$ and $u_k\rightharpoonup0$ in $\mathcal{D}^{s,2}(\mathbb{R}^3)$, by Phrokorov¡¯s theorem (Theorem 8.6.2 in \cite{B}), there exist $\mu,\nu\in\mathcal{M}(\mathbb{R}^3)$ such that
\begin{equation*}
|D_su_k|^2\rightharpoonup\mu\,\,\text{and}\,\, |u_k|^{2_s^{\ast}}\rightharpoonup\nu\,\, \text{weakly-}\ast\,\, \text{in}\,\, \mathcal{M}(\mathbb{R}^3)\,\,\text{as}\,\, k\rightarrow\infty.
\end{equation*}
By Lemma \ref{lem2-2}, there exist an at most countable index set $J$, sequence $\{x_j\}_{j\in J}\subset\mathbb{R}^3$ and $\{\nu_j\}\subset(0,\infty)$ such that
\begin{equation*}
\nu=\sum_{j\in J}\nu_j\delta_{x_j}.
\end{equation*}
We claim that there is at least one $j_0\in J$ such that $x_{j_0}\in \overline{Q}$ with $\nu_{j_0}>0$. If not, for all $j\in J$, $x_j\not\in \overline{Q}$ with $\nu_{j}>0$, then
\begin{equation*}
\int_{\mathbb{R}^3}|u_k|^{2_s^{\ast}}\varphi(x)\,{\rm d}x\rightarrow\int_{\mathbb{R}^3}\sum_{j\in J}\nu_j\delta_{x_j}\varphi\,{\rm d}x=\sum_{j\in J}\nu_j\varphi(x_j), \quad \forall \varphi\in C_0(\mathbb{R}^3).
\end{equation*}
Taking ${\rm supp}(\varphi)= \overline{Q}$, we see that $\int_{Q}|u_k|^{2_s^{\ast}}\,{\rm d}x\rightarrow0$, contradicts with \eqref{equ2-6}. Thus, the claim is true.

We define the L\'{e}vy concentration function
\begin{equation*}
Q_k(r)=\sup_{x\in\overline{Q}}\int_{B_r(x)}|u_k(z)|^{2_s^{\ast}}\,{\rm d}z,
\end{equation*}
then $Q_k$ is a non-decreasing and bounded function. Fixing a small $\tau\in(0,\mathcal{S}_s^{\frac{3}{2s}})$, we can find $\sigma_k:=\sigma_k(\tau)\in\mathbb{R}_{+}$, $z_k\in \overline{Q}$ such that
\begin{equation*}
\int_{B_{\sigma_k}(z_k)}|u_k|^{2_s^{\ast}}\,{\rm d}z=Q_k(\sigma_k)=\tau.
\end{equation*}
Set $v_k(x)=\sigma_k^{\frac{3-2s}{2}}u_k(\sigma_kx+z_k)$, we have that
\begin{equation*}
\widetilde{Q}_k(r)=:\sup_{x\in\bar{Q}_k}\int_{B_r(x)}|v_k|^{2_s^{\ast}}\,{\rm d}z=\sup_{x\in\bar{Q}}\int_{B_{\sigma_k r}(x)}|u_k(z)|^{2_s^{\ast}}\,{\rm d}z=Q_k(\sigma_k r),
\end{equation*}
where $\bar{Q}_k=\{x\in\mathbb{R}^3\,\,|\,\, \sigma_kx+z_k\in\overline{Q}\}$. Hence, we obtain that
\begin{equation*}
\widetilde{Q}_k(1)=\tau=Q_k(\sigma_k)=\int_{B_{\sigma_k}(z_k)}|u_k(z)|^{2_s^{\ast}}\,{\rm d}z=\int_{B_1(0)}|v_k(z)|^{2_s^{\ast}}\,{\rm d}z.
\end{equation*}
Now, we prove that there is a small $\tau_0\in(0,\mathcal{S}_s^{\frac{3}{2s}})$ such that $\sigma_k(\tau_0)\rightarrow0$ as $k\rightarrow\infty$. Otherwise, for any $\varepsilon>0$, there exists $r_{\varepsilon}>0$ such that $\sigma_k(\varepsilon)>r_{\varepsilon}$. Hence, for any $x\in\bar{\Omega}$, there holds
\begin{equation*}
\int_{B_{r_{\varepsilon}}(x)}|u_k(z)|^{2_s^{\ast}}\,{\rm d}z\leq\sup_{x\in\overline{Q}}\int_{B_{\sigma_k(\varepsilon)}(x)}|u_k(z)|^{2_s^{\ast}}\,{\rm d}z=Q_k(\sigma_k(\varepsilon))=\varepsilon.
\end{equation*}
Furthermore,
\begin{equation*}
\nu_{j_0}\leq\int_{B_{r_{\varepsilon}}(x_{j_0})}|u_k(z)|^{2_s^{\ast}}\,{\rm d}z+o_k(1)\leq\varepsilon+o_k(1),\quad \forall\varepsilon>0.
\end{equation*}
Let $k\rightarrow+\infty$ and then $\varepsilon\rightarrow0$, we get $\nu_{j_0}\leq0$, which achieves a contradiction. For the above $\tau_0$, we still denote $\sigma_k:=\sigma_k(\tau_0)$ and the corresponding sequence $z_k\in\overline{Q}$. Thus $v_k(x)=\sigma_k^{\frac{3-2s}{2}}u_k(\sigma_kx+z_k)$ satisfies
\begin{equation}\label{equ2-9}
\widetilde{Q}_k(1)=\int_{B_1(0)}|v_k|^{2_s^{\ast}}\,{\rm d}z=\tau_0>0.
\end{equation}

Note that
\begin{equation*}
\int_{\mathbb{R}^3}|D_sv_k|^2\,{\rm d}z=\int_{\mathbb{R}^3}|D_su_k|^2\,{\rm d}z,
\end{equation*}
by the boundedness of $\{u_k\}$ in $\mathcal{D}^{s,2}(\mathbb{R}^3)$, up to a subsequence, we may assume that there exists $v\in\mathcal{D}^{s,2}(\mathbb{R}^3)$ such that $v_k\rightharpoonup v$ in $\mathcal{D}^{s,2}(\mathbb{R}^3)$.

For any $\phi\in C_0^{\infty}(\mathbb{R}^3)$, denote $\phi_k(x)=\phi((x-z_k)/\sigma_k)$. By the fact that $z_k\in \overline{Q}$ and $\sigma_k\rightarrow0$, we see that for $k$ large enough, ${\rm supp}\phi_k\subset B_{\sigma_k}(z_k)\subset V$, then $\phi_k\in C_0^{\infty}(V)$. From \eqref{equ2-7}, we have that
\begin{align*}
&\int_{\mathbb{R}^3}\frac{(v_k(z)-v_k(y))(\phi(z)-\phi(y))}{|z-y|^{3+2s}}\,{\rm d}y\,{\rm d}z-\int_{\mathbb{R}^3}|v_k|^{2_s^{\ast}-2}v_k\phi\,{\rm d}z\\
&=\int_{\mathbb{R}^3}\frac{(u_k(z)-u_k(y))(\phi_k(z)-\phi_k(y))}{|z-y|^{3+2s}}\,{\rm d}y\,{\rm d}z-\int_{\mathbb{R}^3}|u_k|^{2_s^{\ast}-2}u_k\phi_k\,{\rm d}z\\
&=o(1)\|\varphi_k\|=o(1)\|\varphi\|_{\mathcal{D}^{s,2}}+o(1).
\end{align*}
Thus, $v$ is a solution of equation \eqref{equ2-8}. Next, we will prove that $v$ is nontrivial. By virtue of \eqref{equ2-9}, we only need to show that
\begin{equation}\label{equ2-10}
\int_{B_1(0)}|v_k|^{2_s^{\ast}}\,{\rm d}z\rightarrow\int_{B_1(0)}|v|^{2_s^{\ast}}\,{\rm d}z
\end{equation}
If \eqref{equ2-10} holds true, from \eqref{equ2-9}, we know that
\begin{equation*}
\int_{B_1(0)}|v|^{2_s^{\ast}}\,{\rm d}z=\tau_0>0
\end{equation*}
which implies that $v$ is nontrivial.

By the boundedness of $\{v_k\}$ in $\mathcal{D}^{s,2}(\mathbb{R}^3)$ and $v_k\rightharpoonup v$ in $\mathcal{D}^{s,2}(\mathbb{R}^3)$, by Phrokorov¡¯s theorem (Theorem 8.6.2 in \cite{B}), there exist $\mu,\nu\in\mathcal{M}(\mathbb{R}^3)$ such that
\begin{equation*}
|D_sv_k|^2\rightharpoonup\mu\,\,\text{and}\,\, |v_k|^{2_s^{\ast}}\rightharpoonup\nu\,\, \text{weakly-}\ast\,\, \text{in}\,\, \mathcal{M}(\mathbb{R}^3)\,\,\text{as}\,\, k\rightarrow\infty.
\end{equation*}
By Lemma \ref{lem2-2}, there exist an at most countable index set $J$, sequence $\{x_j\}_{j\in J}\subset\mathbb{R}^3$ and $\{\nu_j\}\subset(0,\infty)$ such that
\begin{equation}\label{equ2-11}
\mu\geq|D_sv|^2+\sum_{j\in J}\mu_j\delta_{x_j},\quad \nu=|v|^{2_s^{\ast}}+\sum_{j\in J}\nu_j\delta_{x_j}.
\end{equation}
and
\begin{equation}\label{equ2-12}
\nu_j\leq(\mathcal{S}_s^{-1}\mu_j)^{\frac{2_s^{\ast}}{2}}\,\, \text{for any}\,\, j\in J.
\end{equation}
Next, we show that $\{x_j\}_{j\in J}\cap\overline{B_1(0)}=\emptyset$. Suppose by contradiction that there exists $j_0\in J$ such that $x_{j_0}\in\overline{B_1(0)}$, and define the function $\phi_{\rho}=:\phi(\frac{x-x_{j_0}}{\rho})$, where $\phi$ is a smooth cut-off function such that $\phi=1$ on $B_1(0)$, $\phi=0$ on $\mathbb{R}^3\backslash B_2(0)$, $0\leq\phi\leq1$ and $|\nabla\phi|\leq C$. Denote $\phi_{k,\rho}(x)=\phi_{\rho}(\frac{x-z_k}{\sigma_k})$, by the fact that $z_k\in\overline{Q}$, $x_{j_0}\in\overline{B_1(0)}$ and $\sigma_k\rightarrow0$ as $k\rightarrow\infty$, we see that for $k$ large, ${\rm supp}\phi_{k,\rho}\subset B_{2\sigma_k\rho}(z_k+\sigma_kx_{j_0})\subset V$. Direct computation, it can be checked that $\phi_{k,\rho}u_k\in H^s(\mathbb{R}^3)$. Indeed, by H\"{o}lder's inequality, we have that
\begin{align*}
\int_{\mathbb{R}^3}|D_s(\phi_{\rho,k}u_k)|^2\,{\rm d}z&\leq2\int_{\mathbb{R}^3}\phi_{\rho,k}^2|D_su_k|^2\,{\rm d}z+2\int_{\mathbb{R}^3}u_k^2|D_s\phi_{\rho,k}|^2\,{\rm d}z\\
&\leq2\int_{\mathbb{R}^3}|D_su_k|^2\,{\rm d}z+\Big(\int_{\mathbb{R}^3}|u_k|^{2_s^{\ast}}\,{\rm d}z\Big)^{\frac{2}{2_s^{\ast}}}\Big(\int_{\mathbb{R}^3}|D_s\phi_{\rho,k}|^{\frac{3}{2s}}\Big)^{\frac{2s}{3}}
\end{align*}
and directly computations, we get
\begin{align*}
&\int_{\mathbb{R}^3}\Big|\int_{\mathbb{R}^3}\frac{|\phi_{\rho,k}(z)-\phi_{\rho,k}(y)|^2}{|z-y|^{3+2s}}\,{\rm d}y\Big|^{\frac{3}{2s}}\,{\rm d}z=\int_{\mathbb{R}^3}\Big|\int_{\mathbb{R}^3}\frac{|\phi(z)-\phi(y)|^2}{|z-y|^{3+2s}}\,{\rm d}y\Big|^{\frac{3}{2s}}\,{\rm d}z\\
&=\int_{\mathbb{R}^3\backslash B_2(0)}\Big|\int_{\mathbb{R}^3}\frac{|\phi(z)-\phi(y)|^2}{|z-y|^{3+2s}}\,{\rm d}y\Big|^{\frac{3}{2s}}\,{\rm d}z+\int_{B_2(0)}\Big|\int_{\mathbb{R}^3}\frac{|\phi(z)-\phi(y)|^2}{|z-y|^{3+2s}}\,{\rm d}y\Big|^{\frac{3}{2s}}\,{\rm d}z\\
&=\int_{\mathbb{R}^3\backslash B_2(0)}\Big|\int_{B_2(0)}\frac{|\phi(z)-\phi(y)|^2}{|z-y|^{3+2s}}\,{\rm d}y\Big|^{\frac{3}{2s}}\,{\rm d}z+\int_{B_2(0)}\Big|\int_{\mathbb{R}^3}\frac{|\phi(z)-\phi(y)|^2}{|z-y|^{3+2s}}\,{\rm d}y\Big|^{\frac{3}{2s}}\,{\rm d}z\\
&\leq C\Big[\int_{B_3(0)}\Big|\int_{|z-y|\leq1}\frac{1}{|z-y|^{1+2s}}\,{\rm d}y\Big|^{\frac{3}{2s}}\,{\rm d}z+\int_{\mathbb{R}^3\backslash B_2(0)}\Big|\int_{|z-y|>1,y\in B_2(0)}\frac{|\phi(z)-\phi(y)|^2}{|z-y|^{3+2s}}\,{\rm d}y\Big|^{\frac{3}{2s}}\,{\rm d}z\\
&+\int_{B_2(0)}\Big|\int_{|z-y|\leq1}\frac{1}{|z-y|^{1+2s}}\,{\rm d}y+\int_{|z-y|>1}\frac{1}{|z-y|^{3+2s}}\,{\rm d}y\Big|^{\frac{3}{2s}}\,{\rm d}z\Big]\\
&\leq C\Big(1+\int_{\mathbb{R}^3\backslash B_2(0)}\Big|\int_{|z-y|>1,y\in B_2(0)}\frac{|\phi(z)-\phi(y)|^2}{|z-y|^{3+2s}}\,{\rm d}y\Big|^{\frac{3}{2s}}\,{\rm d}z\Big)\\
&=C\Big(1+\int_{\mathbb{R}^3\backslash B_2(0)}\Big|\int_{|z-y|>\frac{|z|}{2},y\in B_2(0)}\frac{1}{|z-y|^{3+2s}}\,{\rm d}y\Big|^{\frac{3}{2s}}\,{\rm d}z\\
&+\int_{\mathbb{R}^3\backslash B_2(0)}\Big|\int_{1<|z-y|\leq\frac{|z|}{2},y\in B_2(0)}\frac{1}{|z-y|^{1+2s}}\,{\rm d}y\Big|^{\frac{3}{2s}}\,{\rm d}z\Big)\\
&=C\Big(1+\int_{\mathbb{R}^3\backslash B_2(0)}\Big|\int_{|z-y|>\frac{|z|}{2},y\in B_2(0)}\frac{1}{|z-y|^{3+2s}}\,{\rm d}y\Big|^{\frac{3}{2s}}\,{\rm d}z+\int_{B_4(0)}\Big|\int_{1<|z-y|\leq\frac{|z|}{2}}\frac{1}{|z-y|^{1+2s}}\,{\rm d}y\Big|^{\frac{3}{2s}}\,{\rm d}z\Big)\\
&\leq C\Big(1+\int_{\mathbb{R}^3\backslash B_2(0)}\frac{1}{|z|^{(3+2s)\frac{3}{2s}}}\,{\rm d}z\Big)\leq C
\end{align*}
which implies that $\phi_{\rho,k}u_k\in H^s(\mathbb{R}^3)$.

\begin{align}\label{equ2-13}
&\int_{\mathbb{R}^3}\frac{(v_k(z)-v_k(y))(\phi_{\rho}(z)v_k(z)-\phi_{\rho}(y)v_k(y))}{|z-y|^{3+2s}}\,{\rm d}y\,{\rm d}z-\int_{\mathbb{R}^3}|v_k|^{2_s^{\ast}-2}v_k(\phi_{\rho}v_k)\,{\rm d}z\nonumber\\
&=\int_{\mathbb{R}^3}\frac{(u_k(z)-u_k(y))(\phi_{k,\rho}(z)u_k(z)-\phi_{k,\rho}(y)u_k(y))}{|z-y|^{3+2s}}\,{\rm d}y\,{\rm d}z-\int_{\mathbb{R}^3}|u_k|^{2_s^{\ast}-2}u_k(\phi_{k,\rho}u_k)\,{\rm d}z\nonumber\\
&=o_k(1)\|\phi_{k,\rho} u_k\|=o_k(1).
\end{align}
Since
\begin{align*}
&\int_{\mathbb{R}^3}\frac{(v_k(z)-v_k(y))(\phi_{\rho}(z)v_k(z)-\phi_{\rho}(y)v_k(y))}{|z-y|^{3+2s}}\,{\rm d}y\,{\rm d}z\\
&=C_s\int_{\mathbb{R}^3}|D_sv_k|^2\phi_{\rho}\,{\rm d}z+C_s\int_{\mathbb{R}^3\times\mathbb{R}^3}\frac{(\phi_{\rho}(z)-\phi_{\rho}(y))(v_k(z)-v_k(y))v_k(y)}{|z-y|^{3+2s}}\,{\rm d}y\,{\rm d}z,
\end{align*}
and
\begin{equation*}
\int_{\mathbb{R}^3\times\mathbb{R}^3}\frac{(\phi_{\rho}(z)-\phi_{\rho}(y))(v_k(z)-v_k(y))v_k(y)}{|z-y|^{3+2s}}\,{\rm d}y\,{\rm d}z\leq\|D_sv_k\|_{L^2}\|v_kD_s\phi_{\rho}\|_{L^2},
\end{equation*}
then we claim that
\begin{equation}\label{equ2-13-0}
\lim_{\rho\rightarrow0{+}}\limsup_{k\rightarrow\infty}\int_{\mathbb{R}^3}v_k^2|D_s\phi_{\rho}|^2\,{\rm d}z=0.
\end{equation}
Indeed, since
\begin{align*}
\mathbb{R}^3\times\mathbb{R}^3&=(B_{2\rho}^{c}(x_j)\times B_{2\rho}^{c}(x_j))\cup (B_{2\rho}(x_j)\times B_{2\rho}(x_j))\cup(B_{2\rho}^{c}(x_j)\times B_{\rho}(x_j))\cup (B_{\rho}(x_j)\times B_{2\rho}^{c}(x_j))\\
&\cup( B_{2\rho}(x_j)\backslash B_{\rho}(x_j)\times B_{2\rho}^{c}(x_j))\cup(B_{2\rho}^{c}(x_j)\times B_{2\rho}(x_j)\backslash B_{\rho}(x_j)),
\end{align*}
where $B_{\rho}^{c}(x_j)=\mathbb{R}^3\backslash B_{\rho}(x_j)$ and $B_{2\rho}^{c}(x_j)=\mathbb{R}^3\backslash B_{2\rho}(x_j)$. Next we will discuss the six cases on the above domains, respectively.

$\bullet$ $(y,z)\in B_{2\rho}^{c}(x_j)\times B_{2\rho}^{c}(x_j)$. Clearly $|\phi_{\rho}(z)-\phi_{\rho}(y)|=0$ and so
\begin{equation*}
\int_{B_{2\rho}^{c}(x_j)\times B_{2\rho}^{c}(x_j)}v_k^2(y)\frac{|\phi_{\rho}(z)-\phi_{\rho}(y)|^2}{|z-y|^{3+2s}}\,{\rm d}z\,{\rm d}y=0.
\end{equation*}

$\bullet$ $(y,z)\in B_{2\rho}(x_j)\times B_{2\rho}(x_j)$. Since $|\phi_{\rho}(z)-\phi_{\rho}(y)|\leq \frac{C}{\rho}|z-y|$ and $|y-z|\leq|y-x_j|+|z-x_j|\leq4\rho$, we have
\begin{align*}
\int_{B_{2\rho}(x_j)}v_k^2(y)\int_{B_{2\rho}(x_j)}\frac{|\phi_{\rho}(z)-\phi_{\rho}(y)|^2}{|z-y|^{3+2s}}\,{\rm d}z\,{\rm d}y&\leq\frac{C}{\rho^2}\int_{B_{2\rho}(x_j)}v_k^2(y)\int_{|y-z|\leq4\rho}\frac{1}{|y-z|^{1+2s}}\,{\rm d}z\,{\rm d}y\\
&\leq\frac{C}{\rho^{2s}}\int_{B_{2\rho}(x_j)}v_k^2\,{\rm d}y.
\end{align*}

$\bullet$ $(y,z)\in B_{\rho}(x_j)\times B_{2\rho}^{c}(x_j)$. There holds $|z-y|\geq|z-x_j|-|y-x_j|\geq\rho$ and thus
\begin{align*}
\int_{B_{\rho}(x_j)}v_k^2(y)\int_{|z-y|\geq\rho, z\in B_{2\rho}^{c}(x_j)}\frac{|\phi_{\rho}(y)-\phi_{\rho}(z)|^2}{|z-y|^{3+2s}}\,{\rm d}z\,{\rm d}y&\leq C\int_{B_{\rho}(x_j)}v_k^2\int_{|z-y|\geq\rho}\frac{1}{|z-y|^{3+2s}}\,{\rm d}z\,{\rm d}y\\
&\leq\frac{C}{\rho^{2s}}\int_{B_{\rho}(x_j)}v_k^2\,{\rm d}y.
\end{align*}

$\bullet$ $(y,z)\in B_{2\rho}^{c}(x_j)\times B_{\rho}(x_j)$. Obviously, $|y-z|\geq\rho$. Observe that for any fixed $K\geq4$, $ B_{2\rho}^c(x_j)\times B_{\rho}(x_j)\subset B_{K\rho}(x_j)\times B_{\rho}(x_j)\cup B_{K\rho}^c(x_j)\times B_{\rho}(x_j)$. Hence, if $|y-z|>\rho$ and $(y,z)\in B_{K\rho}(x_j)\times B_{\rho}(x_j)$, we have
\begin{align*}
\int_{B_{K\rho}(x_j)}v_k^2\int_{|y-z|\geq\rho,z\in B_{\rho}(x_j)}\frac{|\phi_{\rho}(y)-\phi_{\rho}(z)|^2}{|y-z|^{3+2s}}\,{\rm d}z\,{\rm d}y&\leq C\int_{B_{K\rho}(x_j)}v_k^2\int_{|y-z|\geq\rho}\frac{1}{|y-z|^{3+2s}}\,{\rm d}z\,{\rm d}y\\
&\leq\frac{C}{\rho^{2s}}\int_{B_{K\rho}(x_j)}v_k^2\,{\rm d}y.
\end{align*}
If $(y,z)\in B_{K\rho}^c(x_j)\times B_{\rho}(x_j)$, $|y-z|\geq|y-x_j|-|z-x_j|\geq\frac{3|y-x_j|}{4}+\frac{K}{4}\rho-\rho>\frac{3|y-x_j|}{4}$. By H\"{o}lder's inequality, we have
\begin{align*}
&\int_{B_{K\rho}^c(x_j)}v_k^2\int_{|x-y|\geq\rho, z\in B_{\rho}(x_j)}\frac{|\phi_{\rho}(z)-\phi_{\rho}(y)|^2}{|z-y|^{3+2s}}\,{\rm d}z\,{\rm d}y\leq C\rho^3\int_{B_{K\rho}^c(x_j)}v_k^2\frac{1}{|y-x_j|^{3+2s}}\,{\rm d}y\\
&\leq C\rho^3\Big(\int_{B_{K\rho}^c(x_j)}|v_k|^{2_s^{\ast}}\,{\rm d}y\Big)^{\frac{3-2s}{3}}\Big(\int_{B_{K\rho}^c(x_j)}\frac{1}{|y-x_j|^{(3+2s)\frac{3}{2s}}}\,{\rm d}y\Big)^{\frac{2s}{3}}\\
&\leq \frac{C}{K^3}\Big(\int_{B_{K\rho}^c(x_j)}|v_k|^{2_s^{\ast}}\,{\rm d}y\Big)^{\frac{3-2s}{3}}\leq\frac{C}{K^3}.
\end{align*}

$\bullet$ $(y,z)\in B_{2\rho}^c(x_j)\times B_{2\rho}(x_j)\backslash B_{\rho}(x_j)$. If $|y-z|\leq\rho$, then $|y-x_j|\leq|y-z|+|z-x_j|\leq3\rho$ and thus
\begin{align*}
\int_{ B_{2\rho}^c(x_j)}v_k^2\int_{|y-z|\leq\rho}\frac{|\phi_{\rho}(y)-\phi_{\rho}(z)|^2}{|y-z|^{3+2s}}\,{\rm d}z\,{\rm d}y&\leq\frac{C}{\rho^2}\int_{B_{3\rho}(x_j)}v_k^2\int_{|y-z|\leq\rho}\frac{|y-z|^2}{|y-z|^{3+2s}}\,{\rm d}z\,{\rm d}y\\
&\leq\frac{C}{\rho^{2s}}\int_{B_{3\rho}(x_j)}v_k^2\,{\rm d}y.
\end{align*}
Observe that for any fixed $K\geq4$, $ B_{2\rho}^c(x_j)\times B_{2\rho}(x_j)\backslash B_{\rho}(x_j) A \subset B_{K\rho}(x_j)\times B_{2\rho}(x_j)\cup B_{K\rho}^c(x_j)\times B_{2\rho}(x_j)$. Hence, if $|y-z|>\rho$ and $(y,z)\in B_{K\rho}(x_j)\times B_{2\rho}(x_j)$, we have
\begin{align*}
\int_{B_{K\rho}(x_j)}v_k^2\int_{|y-z|>\rho,z\in B_{2\rho}(x_j)}\frac{|\phi_{\rho}(y)-\phi_{\rho}(z)|^2}{|y-z|^{3+2s}}\,{\rm d}z\,{\rm d}y&\leq C\int_{B_{K\rho}(x_j)}v_k^2\int_{|y-z|>\rho}\frac{1}{|y-z|^{3+2s}}\,{\rm d}z\,{\rm d}y\\
&\leq\frac{C}{\rho^{2s}}\int_{B_{K\rho}(x_j)}v_k^2\,{\rm d}y.
\end{align*}
If $(y,z)\in B_{K\rho}^c(x_j)\times B_{2\rho}(x_j)$, $|y-z|\geq|y-x_j|-|z-x_j|\geq\frac{|y-x_j|}{2}+\frac{K}{2}\rho-2\rho\geq\frac{|y-x_j|}{2}$. By H\"{o}lder's inequality, we have
\begin{align*}
&\int_{B_{K\rho}^c(x_j)}v_k^2\int_{|x-y|>\rho, z\in B_{2\rho}(x_j)}\frac{|\phi_{\rho}(z)-\phi_{\rho}(y)|^2}{|z-y|^{3+2s}}\,{\rm d}z\,{\rm d}y\leq C\rho^3\int_{B_{K\rho}^c(x_j)}v_k^2\frac{1}{|y-x_j|^{3+2s}}\,{\rm d}y\\
&\leq C\rho^3\Big(\int_{B_{K\rho}^c(x_j)}|v_k|^{2_s^{\ast}}\,{\rm d}y\Big)^{\frac{3-2s}{3}}\Big(\int_{B_{K\rho}^c(x_j)}\frac{1}{|y-x_j|^{(3+2s)\frac{3}{2s}}}\,{\rm d}y\Big)^{\frac{2s}{3}}\\
&\leq \frac{C}{K^3}\Big(\int_{B_{K\rho}^c(x_j)}|v_k|^{2_s^{\ast}}\,{\rm d}y\Big)^{\frac{3-2s}{3}}\leq\frac{C}{K^3}.
\end{align*}

$\bullet$ $(y,z)\in B_{2\rho}(x_j)\backslash B_{\rho}(x_j)\times B_{2\rho}^{c}(x_j)$. If $|y-z|\leq\rho$, we have
\begin{align*}
\int_{ B_{2\rho}(x_j)\backslash B_{\rho}(x_j)}v_k^2\int_{|y-z|\leq\rho,z\in B_{2\rho}^{c}(x_j)}\frac{|\phi_{\rho}(y)-\phi_{\rho}(z)|^2}{|x-y|^{3+2s}}\,{\rm d}z\,{\rm d}y&\leq\frac{C}{\rho^2}\int_{B_{2\rho}(x_j)}v_k^2\int_{|y-z|\leq\rho}\frac{|y-z|^2}{|y-z|^{3+2s}}\,{\rm d}z\,{\rm d}y\\
&\leq\frac{C}{\rho^{2s}}\int_{B_{2\rho}(x_j)}v_k^2\,{\rm d}y.
\end{align*}
If $|y-z|>\rho$, then $|y-z|\geq\frac{|y-x_j|}{2}$. One has
\begin{align*}
\int_{B_{2\rho}(x_j)\backslash B_{\rho}(x_j)}v_k^2(y)\int_{|z-y|>\rho, z\in B_{2\rho}^{c}(x_j)}\frac{|\phi_{\rho}(y)-\phi_{\rho}(z)|^2}{|z-y|^{3+2s}}\,{\rm d}z\,{\rm d}y&\leq C\int_{B_{2\rho}(x_j)}v_k^2\int_{|z-y|>\rho}\frac{1}{|z-y|^{3+2s}}\,{\rm d}z\,{\rm d}y\\
&\leq\frac{C}{\rho^{2s}}\int_{B_{2\rho}(x_j)}v_k^2\,{\rm d}y.
\end{align*}

From all the above estimates and using H\"{o}lder's inequality, we get that
\begin{align*}
&\limsup_{k\rightarrow\infty}\int_{\mathbb{R}^3}v_k^2|D_s\phi_{\rho}|^2\,{\rm d}x\\
&\leq \frac{C}{\rho^{2s}}\Big(\int_{B_{2\rho}(x_j)}v^2\,{\rm d}x+\int_{B_{3\rho}(x_j)}v^2\,{\rm d}x+\int_{B_{K\rho}(x_j)}v^2\,{\rm d}x\Big)+\frac{C}{K^3}\\
&\leq C\Big[\Big(\int_{B_{2\rho}(x_j)}|v|^{2_s^{\ast}}\,{\rm d }x\Big)^{\frac{2}{2_s^{\ast}}}+\Big(\int_{B_{3\rho}(x_j)}|v|^{2_s^{\ast}}\,{\rm d }x\Big)^{\frac{2}{2_s^{\ast}}}\Big]
+K^{2s}\Big(\int_{B_{K\rho}(x_j)}|v|^{2_s^{\ast}}\,{\rm d }x\Big)^{\frac{2}{2_s^{\ast}}}+\frac{C}{K^3}.
\end{align*}
Letting $\rho\rightarrow0^{+}$ and then letting $K\rightarrow+\infty$, \eqref{equ2-13-0} follows.
Thus, from \eqref{equ2-13}, \eqref{equ2-11} and \eqref{equ2-12}, we get that
\begin{align*}
\mu_{j_0}&=\lim_{\rho\rightarrow0^{+}}\int_{B_{2\rho}(x_{j_0})}\phi_{\rho}\,{\rm d}\mu=\lim_{\rho\rightarrow0^{+}}\lim_{k\rightarrow\infty}\int_{\mathbb{R}^3}|D_sv_k|^2\phi_{\rho}\,{\rm d}x\\
&=\lim_{\rho\rightarrow0^{+}}\lim_{k\rightarrow\infty}\int_{\mathbb{R}^3}|v_k|^{2_s^{\ast}}\phi_{\rho}\,{\rm d}x=\lim_{\rho\rightarrow0^{+}}\int_{\mathbb{R}^3}|v|^{2_s^{\ast}}\phi_{\rho}\,{\rm d}x+\nu_{j_0}=\nu_{j_0}
\end{align*}
which leads to
\begin{equation*}
\nu_{j_0}\geq\mathcal{S}_s^{\frac{3}{2s}}.
\end{equation*}
But, by \eqref{equ2-9}, we have that
\begin{equation*}
\mathcal{S}_s^{\frac{3}{2s}}\leq\nu_{j_0}\leq\int_{B_1(0)}|v_k|^{2_s^{\ast}}\,{\rm d}x+o(1)=\tau_0+o(1)
\end{equation*}
which contradicts with $\tau_0<\mathcal{S}_s^{\frac{3}{2s}}$. Hence, $\{x_j\}_{j\in J}\cap \overline{B_1(0)}=\emptyset$ and then \eqref{equ2-10} holds. We complete the proof.
\end{proof}

\begin{proposition}\label{pro2-2}(\cite{Teng3}, Proposition 5.1)
Assume that $u_n$ are nonnegative weak solution of
\begin{equation*}
\left\{
  \begin{array}{ll}
   (-\Delta)^su+V_n(x) u+ \phi u=f_n(x,u) & \hbox{in $\mathbb{R}^3$,} \\
    (-\Delta)^t\phi=u^2& \hbox{in $\mathbb{R}^3$,}
  \end{array}
\right.
\end{equation*}
where $\{V_n\}$ satisfies $V_n(x)\geq V_0>0$ for all $x\in\mathbb{R}^3$ and $f_n(x,\tau)$ is a Carathedory function satisfying that for any $\delta>0$, there exists $C_{\varepsilon}>0$ such that
\begin{equation*}
|f_n(x,\tau)|\leq\delta|\tau|+C_{\delta}|\tau|^{2_s^{\ast}-1},\quad \forall (x,\tau)\in\mathbb{R}^3\times\mathbb{R}.
\end{equation*}
satisfying $u_n$ convergence strongly in $H^s(\mathbb{R}^3)$ or $u_n$ convergence strongly in $L^{2_s^{\ast}}(\mathbb{R}^3)$. Then there exists $C>0$ such that
\begin{equation*}
\|u_n\|_{L^{\infty}}\leq C\quad \text{for all}\,\, n.
\end{equation*}
\end{proposition}

\begin{lemma}\label{lem2-3}(\cite{S}, Proposition 2.9)
Let $w=(-\Delta)^su$. Assume $w\in L^{\infty}(\mathbb{R}^n)$ and $u\in L^{\infty}(\mathbb{R}^n)$ for $s>0$.\\
If $2s\leq1$, then $u\in C^{0,\alpha}(\mathbb{R}^n)$ for any $\alpha\leq2s$. Moreover
\begin{equation*}
\|u\|_{C^{0,\alpha}(\mathbb{R}^n)}\leq C\Big(\|u\|_{L^{\infty}(\mathbb{R}^n)}+\|w\|_{L^{\infty}(\mathbb{R}^n)}\Big)
\end{equation*}
for some constant $C$ depending only on $n$, $\alpha$ and $s$.\\
If $2s>1$, then $u\in C^{1,\alpha}(\mathbb{R}^n)$ for any $\alpha<2s-1$. Moreover
\begin{equation*}
\|u\|_{C^{1,\alpha}(\mathbb{R}^n)}\leq C\Big(\|u\|_{L^{\infty}(\mathbb{R}^n)}+\|w\|_{L^{\infty}(\mathbb{R}^n)}\Big)
\end{equation*}
for some constant $C$ depending only on $n$, $\alpha$ and $s$.
\end{lemma}

\section{Limiting problem}

In this section, we consider the "limiting problem" associated with problem \eqref{main-2-1}
\begin{equation}\label{equ3-1}
\left\{
  \begin{array}{ll}
    (-\Delta)^su+\mu u+\phi u=f(u)+u^{2_s^{\ast}-1} & \hbox{in $\mathbb{R}^3$,} \\
    (-\Delta)^t\phi=u^2,\,\, u>0& \hbox{in $\mathbb{R}^3$}
  \end{array}
\right.
\end{equation}
for $\mu>0$. The energy functional for the limiting problem \eqref{equ3-1} is given by
\begin{align*}
\mathcal{I}_{\mu}(u)&=\frac{1}{2}\int_{\mathbb{R}^3}|D_su|^2\,{\rm d}x+\frac{\mu}{2}\int_{\mathbb{R}^3}|u|^2\,{\rm d}x+\frac{1}{4}\int_{\mathbb{R}^3}\phi_u^tu^2\,{\rm d}x-\int_{\mathbb{R}^3}F(u)\,{\rm d}x\\
&-\frac{1}{2_s^{\ast}}\int_{\mathbb{R}^3}(u^{+})^{2_s^{\ast}}\,{\rm d}x,\quad u\in H^s(\mathbb{R}^3).
\end{align*}
Let
\begin{align*}
\mathcal{P}_{\mu}(u)=&\frac{3-2s}{2}\int_{\mathbb{R}^3}|D_su|^2\,{\rm d}x+\frac{3}{2}\int_{\mathbb{R}^3}\mu|u|^2\,{\rm d}x+\frac{3+2t}{4}\int_{\mathbb{R}^3}\phi_{u}^tu^2\,{\rm d}x\\
&-3\int_{\mathbb{R}^3}F(u)\,{\rm d}x-\frac{3}{2_s^{\ast}}\int_{\mathbb{R}^3}(u^{+})^{2_s^{\ast}}\,{\rm d}x
\end{align*}
and
\begin{align*}
\mathcal{G}_{\mu}(u)&=(s+t)\langle \mathcal{I}'_{\mu}(u),u\rangle-\mathcal{P}_{\mu}(u)\\
&=\frac{4s+2t-3}{2}\int_{\mathbb{R}^3}|D_su|^2\,{\rm d}x+\frac{2s+2t-3}{2}\mu\int_{\mathbb{R}^3}|u|^2\,{\rm d}x\\
&+\frac{4s+2t-3}{4}\int_{\mathbb{R}^3}\phi_u^tu^2\,{\rm d}x+\int_{\mathbb{R}^3}\Big(3F(u)-(s+t)f(u)u\Big)\,{\rm d}x\\
&+\frac{3-(s+t)2_s^{\ast}}{2_s^{\ast}}\int_{\mathbb{R}^3}(u^{+})^{2_s^{\ast}}\,{\rm d}x.
\end{align*}
We define the Nehari-Pohozaev manifold
\begin{equation*}
\mathcal{M}_{\mu}=\{u\in H^s(\mathbb{R}^3)\backslash\{0\}\,\,\Big|\,\, \mathcal{G}_{\mu}(u)=0\}
\end{equation*}
and set $b_{\mu}=\inf\limits_{u\in\mathcal{M}_{\mu}}\mathcal{I}_{\mu}(u)$. By standard arguments, we can show the following properties of $\mathcal{M}_{\mu}$.
\begin{proposition}\label{pro3-2}
The set $\mathcal{M}_{\mu}$ possesses the following properties:\\
$(i)$ $0\not\in\partial\mathcal{M}_{\mu}$;\\
$(ii)$ for any $u\in H^s(\mathbb{R}^3)\backslash\{0\}$, there exists a unique $\tau_0:=\tau(u)>0$ such that $u_{\tau_0}\in\mathcal{M}_{\mu}$, where $u_{\tau}=\tau^{s+t}u(\tau x)$. Moreover,
\begin{equation*}
\mathcal{I}_{\mu}(u_{\tau_0})=\max_{\tau\geq0}\mathcal{I}_{\mu}(u_{\tau});
\end{equation*}
\end{proposition}

Now, it is easy to check that $\mathcal{I}_{\mu}$ satisfies the mountain pass geometry.
\begin{lemma}\label{lem3-1}
$(i)$ there exist $\rho_0,\beta_0>0$ such that $\mathcal{I}_{\mu}(u)\geq\beta_0$ for all $u\in H^s(\mathbb{R}^3)$ with $\|u\|=\rho_0$;\\
$(ii)$ there exists $u_0\in H^s(\mathbb{R}^3)$ such that $\mathcal{I}_{\mu}(u_0)<0$.
\end{lemma}

%

From Lemma \ref{lem3-1}, the mountain-pass level of $\mathcal{I}_{\mu}$ defined as follows
\begin{equation*}
c_{\mu}=\inf_{\gamma\in\Gamma_{\mu}}\sup_{t\in[0,1]}\mathcal{I}_{\mu}(\gamma(t))
\end{equation*}
where
\begin{equation*}
\Gamma_{\mu}=\Big\{\gamma\in C([0,1],H^s(\mathbb{R}^3))\,\,\Big|\,\, \gamma(0)=0,\,\, \mathcal{I}_{\mu}(\gamma(1))<0\Big\}
\end{equation*}
satisfies that $c_{\mu}>0$. Furthermore, by $(f_3)$, it is easy to verify that
\begin{equation}\label{equ3-2}
f'(\tau)\tau-(q-1)f(\tau)>0\,\,\text{and}\,\,  f(\tau)\tau-qF(\tau)>0 \,\,\text{for any}\,\,\tau>0.
\end{equation}
By using Lemma \ref{lem3-1} and \eqref{equ3-2}, we can show the equivalent characterization of mountain-pass level $c_{\mu}$.
\begin{lemma}\label{lem3-2}
\begin{equation*}
c_{\mu}=b_{\mu}.
\end{equation*}
\end{lemma}
In order to obtain the boundedness of $(PS)$ sequence, we will construct a $(PS)$ sequence $\{u_n\}$ for $\mathcal{I}_{\mu}$ at the level $c_{\mu}$ that satisfies $\mathcal{G}_{\mu}(u_n)\rightarrow0$ as $n\rightarrow+\infty$ i.e.,
\begin{lemma}\label{lem3-3}
There exists a sequence $\{u_n\}$ in $H^s(\mathbb{R}^3)$ such that as $n\rightarrow+\infty$,
\begin{equation}\label{equ3-5}
\mathcal{I}_{\mu}(u_n)\rightarrow c_{\mu},\quad\mathcal{I}'_{\mu}(u_n)\rightarrow0,\quad \mathcal{G}_{\mu}(u_n)\rightarrow0.
\end{equation}
\end{lemma}

\begin{lemma}\label{lem3-4}
Every sequence $\{u_n\}\subset H^s(\mathbb{R}^3)$ satisfying \eqref{equ3-5} is bounded in $H^s(\mathbb{R}^3)$.
\end{lemma}

\begin{proof}
By \eqref{equ3-5}, \eqref{equ3-2} and $q>\frac{4s+2t}{s+t}$, we have that
\begin{align*}
&c_{\mu}+o_n(1)=\mathcal{I}_{\mu}(u_n)-\frac{1}{q(s+t)-3}\mathcal{G}_{\mu}(u_n)\\
&=\frac{(q-4)s+(q-2)t}{2(q(s+t)-3)}\int_{\mathbb{R}^3}|D_su_n|^2\,{\rm d}x+\frac{(q-2)(s+t)}{2(q(s+t)-3)}\mu\int_{\mathbb{R}^3}|u_n|^2\,{\rm d}x\\
&+\frac{(q-4)s+(q-2)t}{4(q(s+t)-3)}\int_{\mathbb{R}^3}\phi_{u_n}^tu_n^2\,{\rm d}x+\frac{(2_s^{\ast}-q)(s+t)}{2_s^{\ast}(q(s+t)-3)}\int_{\mathbb{R}^3}(u_n^{+})^{2_s^{\ast}}\,{\rm d}x\\
&+\frac{s+t}{q(s+t)-3}\int_{\mathbb{R}^3}\Big(f(u_n)u_n-qF(u_n)\Big)\,{\rm d}x.
\end{align*}
Hence, sequence $\{u_n\}$ is bounded in $H^s(\mathbb{R}^3)$.
\end{proof}

For obtaining the compactness of the above bounded sequence $\{u_n\}$, we need the estimate of the Mountain-Pass level $c_{\mu}$ which is given as the following Lemma.

\begin{lemma}\label{lem3-4}
\begin{equation*}
c_{\mu}<\frac{s}{3}\mathcal{S}_s^{\frac{3}{2s}}
\end{equation*}
if in the case $s>\frac{3}{4}$, $q\in(2_s^{\ast}-2,2_s^{\ast})$ for all $\lambda>0$ or $q\in(\frac{4s+2t}{s+t},2_s^{\ast}-2]$ for $\lambda>0$ large; if in the case $\frac{1}{2}<s\leq\frac{3}{4}$, $q\in(\frac{4s+2t}{s+t},2_s^{\ast})$ for any $\lambda>0$, where $\mathcal{S}_s$ is the best Sobolev consatnt for the embedding $\mathcal{D}^{s,2}(\mathbb{R}^3)\hookrightarrow L^{2_s^{\ast}}(\mathbb{R}^3)$.
\end{lemma}
\begin{proof}
Let
\begin{equation*}
u_{\delta}(x)=\psi(x)U_{\delta}(x),\quad x\in\mathbb{R}^3,
\end{equation*}
where $U_{\delta}(x)=\delta^{-\frac{3-2s}{2}}u^{\ast}(x/\delta)$, $u^{\ast}(x)=\frac{\widetilde{u}(x/\mathcal{S}_s^{\frac{1}{2s}})}{\|\widetilde{u}\|_{2_s^{\ast}}}$, $\kappa\in\mathbb{R}\backslash\{0\}$, $\mu>0$ and $x_0\in\mathbb{R}^3$ are fixed constants, $\widetilde{u}(x)=\kappa(\mu^2+|x-x_0|^2)^{-\frac{3-2s}{2}}$, and $\psi\in C^{\infty}(\mathbb{R}^3)$ such that $0\leq\psi\leq1$ in $\mathbb{R}^3$, $\psi(x)\equiv1$ in $B_{R}$ and $\psi\equiv0$ in $\mathbb{R}^3\backslash B_{2R}$. From Proposition 21 and Proposition 22 in \cite{SV}, Lemma 3.3 in \cite{Teng}, we know that
\begin{equation}\label{equ3-8}
\int_{\mathbb{R}^{3}}|D_su_{\delta}(x)|^2\,{\rm d}x\leq\mathcal{S}_s^{\frac{3}{2s}}+O(\delta^{3-2s}),
\end{equation}
\begin{equation}\label{equ3-9}
\int_{\mathbb{R}^{3}}|u_{\delta}(x)|^{2_s^{\ast}}\,{\rm d}x=\mathcal{S}_s^{\frac{3}{2s}}+O(\delta^3),
\end{equation}
and
\begin{equation}\label{equ3-10}
\int_{\mathbb{R}^3}|u_{\delta}(x)|^p\,{\rm d}x=\left\{
\begin{array}{ll}
O(\delta^{\frac{(2-p)3+2sp}{2}})&\hbox{$p>\frac{3}{3-2s}$,} \\
O(\delta^{\frac{(2-p)3+2sp}{2}}|\log\delta|) & \hbox{$p=\frac{3}{3-2s}$,} \\
O(\delta^{\frac{3-2s}{2}p}) & \hbox{$p<\frac{3}{3-2s}$.}
\end{array}
\right.
\end{equation}
Here $a_{\delta}=O(b_{\delta})$ means that $C_1\leq \frac{a_{\delta}}{b_{\delta}}\leq C_2$ for some $C_1, C_2>0$, independent of $\delta$.

Set $u_{\delta}^{\tau}(x)=\tau^{s+t} u_{\delta}(\tau x)$ for any $\tau\geq0$, by $(f_2)$, we deduce that
\begin{align*}
&\mathcal{I}_{\mu}(u_{\delta}^{\tau})\leq h_{\delta}(\tau):=\frac{\tau^{4s+2t-3}}{2}\int_{\mathbb{R}^3}|D_su_{\delta}|^2\,{\rm d}x+\frac{\tau^{2s+2t-3}}{2}\int_{\mathbb{R}^3}\mu |u_{\delta}|^2\,{\rm d}x\\
&+\frac{\tau^{4s+2t-3}}{4}\int_{\mathbb{R}^3}\phi_{u_{\delta}}^tu_{\delta}^2\,{\rm d}x-\lambda\frac{\tau^{q(s+t)-3}}{q}\int_{\mathbb{R}^3}|u_{\delta}|^q\,{\rm d}x-\frac{\tau^{2_s^{\ast}(s+t)-3}}{2_s^{\ast}}\int_{\mathbb{R}^3}|u_{\delta}|^{2_s^{\ast}}\,{\rm d}x.
\end{align*}
Since $h_{\delta}(\tau)\rightarrow-\infty$ as $\tau\rightarrow+\infty$, we have that $\sup\{h_{\delta}(\tau):\,\, \tau\geq0\}=h_{\delta}(\tau_{\delta})$ for some $\tau_{\delta}>0$. Hence, $\tau_{\delta}$ verifies the following equality:
\begin{align}\label{equ3-9}
&\frac{4s+2t-3}{2}\tau_{\delta}^{2s}\int_{\mathbb{R}^3}|D_su_{\delta}|^2\,{\rm d}x+\frac{2s+2t-3}{2}\int_{\mathbb{R}^3}\mu |u_{\delta}|^2\,{\rm d}x\nonumber\\
&+\frac{4s+2t-3}{4}\tau_{\delta}^{2s}\int_{\mathbb{R}^3}\phi_{u_{\delta}}^tu_{\delta}^2\,{\rm d}x\nonumber\\
&=\frac{\lambda(q(s+t)-3)}{q}\tau_{\delta}^{(q-2)(s+t)}\int_{\mathbb{R}^3}|u_{\delta}|^q\,{\rm d}x+\frac{(2_s^{\ast}(s+t)-3)}{2_s^{\ast}}\tau_{\delta}^{(2_s^{\ast}-2)(s+t)}\int_{\mathbb{R}^3}|u_{\delta}|^{2_s^{\ast}}\,{\rm d}x.
\end{align}
We claim that $\{\tau_{\delta}\}$ is bounded from below by a positive constant for $\delta$ small. Otherwise, there exists a sequence $\delta_n\rightarrow0$ such that $\tau_{\delta_n}\rightarrow0$ as $n\rightarrow+\infty$. Thus $0<c_{\mu}\leq\sup_{\tau\geq0}\mathcal{I}_{\mu}(u_{\delta_n}^{\tau})\leq\sup_{\tau\geq0}h_{\delta_n}(\tau)=h_{\delta_n}(\tau_{\delta_n})\rightarrow0$ as $n\rightarrow\infty$, a contradiction. So there exists a constant $C_0>0$ independent of $\delta$ such that $\tau_{\delta}\geq C_0$. Using the similar argument in \eqref{equ3-9}, we can show that the sequence $\{\tau_{\delta}\}$ is bounded from above by a constant $C$ independent of $\delta$. Thus $0<C_0\leq\tau_{\delta}\leq C$ for $\delta$ small.

Let $g_{\delta}(\tau)=\frac{\tau^{4s+2t-3}}{2}\int_{\mathbb{R}^3}|D_su_{\delta}|^2\,{\rm d}x-\frac{\tau^{2_s^{\ast}(s+t)-3}}{2_s^{\ast}}\int_{\mathbb{R}^3}|u_{\delta}|^{2_s^{\ast}}\,{\rm d}x$, then we get for some universal constant $C>0$ so that
\begin{align}\label{equ3-10}
\mathcal{I}_{\mu}(u_{\delta}^{\tau})&\leq \sup_{\tau\geq0}g_{\delta}(\tau)+C\int_{\mathbb{R}^3}\mu|u_{\delta}|^2\, {\rm d}x+C\int_{\mathbb{R}^3}\phi_{u_{\delta}}^tu_{\delta}^2\,{\rm d}x-C\lambda\int_{\mathbb{R}^3}|u_{\delta}|^{q}\, {\rm d}x.
\end{align}
Directly computation, we get that $\frac{4s+2t-3}{2}\frac{2_s^{\ast}}{2_s^{\ast}(s+t)-3}=1$ and $\frac{2_s^{\ast}(s+t)-3}{(2_s^{\ast}-4)s+(2_s^{\ast}-2)t}=\frac{3}{2s}$. Thus, by \eqref{equ3-8}, we deduce that
\begin{align*}
\sup_{\tau\geq0}g_{\delta}(\tau)&=g_{\delta}(\tau_0)=\frac{s}{3}\frac{\Big(\int_{\mathbb{R}^3}|D_su_{\delta}|^2\,{\rm d}x\Big)^{\frac{3}{2s}}}{\Big(\int_{\mathbb{R}^3}|u_{\delta}|^{2_s^{\ast}}\,{\rm d}x\Big)^{\frac{3-2s}{3}}}\leq\frac{s}{3}\frac{(\mathcal{S}_s^{\frac{3}{2s}}+O(\delta^{3-2s}))^{\frac{3}{2s}}}{(\mathcal{S}_s^{\frac{3}{2s}}+O(\delta^{3}))^{\frac{3-2s}{3}}}\\
&\leq\frac{s}{3} \mathcal{S}_s^{\frac{3}{2s}}+O(\delta^{3-2s}),
\end{align*}
where
\begin{align*}
\tau_0&=\Big(\frac{4s+2t-3}{2}\frac{2_s^{\ast}}{2_s^{\ast}(s+t)-3}\frac{\int_{\mathbb{R}^3}|D_su_{\delta}|^2\,{\rm d}x}{\int_{\mathbb{R}^3}|u_{\delta}|^{2_s^{\ast}}\,{\rm d}x}\Big)^{\frac{1}{(2_s^{\ast}-4)s+(2_s^{\ast}-2)t}}\\
&=\Big(\frac{\int_{\mathbb{R}^3}|D_su_{\delta}|^2\,{\rm d}x}{\int_{\mathbb{R}^3}|u_{\delta}|^{2_s^{\ast}}\,{\rm d}x}\Big)^{\frac{1}{(2_s^{\ast}-4)s+(2_s^{\ast}-2)t}}.
\end{align*}
By $(i)$ of Lemma \ref{lem2-1}, \eqref{equ3-8} and \eqref{equ3-9}, using the elementary inequality $(a+b)^{\alpha}\leq a^{\alpha}+\alpha (a+b)^{\alpha-1}b$, $\alpha\geq1$ and $a,b\geq0$, we deduce that
\begin{align*}
\mathcal{I}_{\mu}(u_{\delta}^{\tau})&\leq \frac{s}{3}\mathcal{S}_s^{\frac{3}{2s}}+C \delta^{3-2s}+C\int_{\mathbb{R}^3}\mu|u_{\delta}|^2\, {\rm d}x+C\|u_{\delta}\|_{L^{\frac{12}{3+2t}}}^4-C\lambda\int_{\mathbb{R}^3}|u_{\delta}|^{q}\, {\rm d}x.
\end{align*}

$\bullet$ In the case $s>\frac{3}{4}$, by \eqref{equ3-8}, we deduce that
\begin{equation*}
\mathcal{I}_{\mu}(u_{\delta}^{\tau})\leq\frac{s}{3}\mathcal{S}_s^{\frac{3}{2s}}+C\delta^{3-2s}+C\|u_{\delta}\|_{L^{\frac{12}{3+2t}}}^4-C\lambda\int_{\mathbb{R}^3}|u_{\delta}|^{q}\, {\rm d}x.
\end{equation*}
In view of \eqref{equ3-8}, we have that
\begin{align*}
\lim_{\delta\rightarrow0^{+}}\frac{\Big(\int_{\mathbb{R}^3}|u_{\delta}|^{\frac{12}{3+2t}}\,{\rm d}x\Big)^{\frac{3+2t}{3}}}{\delta^{3-2s}}\leq\left\{
                                                                                                                            \begin{array}{ll}
                                                                                                                              \lim\limits_{\delta\rightarrow0^{+}}\frac{O(\delta^{2t+4s-3})}{\delta^{3-2s}}=0, & \hbox{$\frac{12}{3+2t}>\frac{3}{3-2s}$,} \\
                                                                                                                              \lim\limits_{\delta\rightarrow0^{+}}\frac{O(\delta^{2t+4s-3}|\log\delta|^{\frac{3+2t}{3}})}{\delta^{3-2s}}=0, & \hbox{$\frac{12}{3+2t}=\frac{3}{3-2s}$,} \\
                                                                                                                             \lim\limits_{\delta\rightarrow0^{+}} \frac{O(\delta^{2(3-2s)})}{\delta^{3-2s}}=0, & \hbox{$\frac{12}{3+2t}<\frac{3}{3-2s}$,}
                                                                                                                            \end{array}
                                                                                                                          \right.
\end{align*}
and noting that $2s-\frac{3-2s}{2}q<0$ if $\frac{4s}{3-2s}< q<\frac{6}{3-2s}$, we have
\begin{align*}
\lim_{\delta\rightarrow0^{+}}\lambda\frac{\int_{\mathbb{R}^3}|u_{\delta}|^q\,{\rm d}x}{\delta^{3-2s}}=
\left\{
  \begin{array}{ll}
\lim\limits_{\delta\rightarrow0}\lambda\frac{O(\delta^{3-\frac{3-2s}{2}q})}{\delta^{3-2s}}=+\infty, & \hbox{$\frac{4s}{3-2s}<q<\frac{6}{3-2s}$,} \\
    \lim\limits_{\delta\rightarrow0}\lambda\frac{O(\delta^{3-\frac{3-2s}{2}q})}{\delta^{3-2s}}, & \hbox{$\frac{3}{3-2s}<q\leq\frac{4s}{3-2s}$,} \\
     \lim\limits_{\delta\rightarrow0}\lambda\frac{O(\delta^{3-\frac{3-2s}{2}q}|\log\delta|)}{\delta^{3-2s}}, & \hbox{$q=\frac{3}{3-2s}$,} \\
    \lim\limits_{\delta\rightarrow0}\lambda\frac{O(\delta^{\frac{3-2s}{2}q})}{\delta^{3-2s}}, & \hbox{$\frac{4s+2t}{s+t}<q<\frac{3}{3-2s}$.}
  \end{array}
\right.
\end{align*}
We can choosing $\lambda$ large enough such that the above three limit equal to $+\infty$, for instance, $\lambda=\delta^{-2s}$.

$\bullet$ If $s=\frac{3}{4}$, it follows from \eqref{equ3-8} that
\begin{align*}
\mathcal{I}_{\mu}(u_{\delta}^{\tau})&\leq\frac{s}{3}\mathcal{S}_s^{\frac{3}{2s}}+C\delta^{\frac{3}{2}}+C\delta^{\frac{3}{2}}|\log\delta|+C\|u_{\delta}\|_{L^{\frac{12}{3+2t}}}^4-C\lambda\int_{\mathbb{R}^3}|u_{\delta}|^{q}\, {\rm d}x\\
&\leq\frac{s}{3}\mathcal{S}_s^{\frac{3}{2s}}+C\delta^{\frac{3}{2}}|\log\delta|+C\|u_{\delta}\|_{L^{\frac{12}{3+2t}}}^4-C\lambda\int_{\mathbb{R}^3}|u_{\delta}|^{q}\, {\rm d}x.
\end{align*}
Since $\frac{12}{3+2t}>2=\frac{3}{3-2s}$, we get that
\begin{align*}
\lim_{\delta\rightarrow0^{+}}\frac{\Big(\int_{\mathbb{R}^3}|u_{\delta}|^{\frac{12}{3+2t}}\,{\rm d}x\Big)^{\frac{3+2t}{3}}}{\delta^{2s}|\log\delta|}\leq
\lim\limits_{\delta\rightarrow0^{+}}\frac{O(\delta^{2t+4s-3})}{\delta^{2s}|\log\delta|}=0,\quad\frac{12}{3+2t}>\frac{3}{3-2s}=2
\end{align*}
and owing to $\frac{3}{3-2s}=2<\frac{4s+2t}{s+t}<q$, then for any $\lambda>0$, we obtain that
\begin{align*}
\lim_{\delta\rightarrow0^{+}}\lambda\frac{\int_{\mathbb{R}^3}|u_{\delta}|^q\,{\rm d}x}{\delta^{2s}|\log\delta|}=
\lim\limits_{\delta\rightarrow0^{+}}\lambda\frac{O(\delta^{3-\frac{3-2s}{2}q})}{\delta^{2s}|\log\delta|}=+\infty,\quad \frac{4s+2t}{s+t}<q<\frac{6}{3-2s}.
\end{align*}

$\bullet$ In the case $\frac{1}{2}<s<\frac{3}{4}$, by means of \eqref{equ3-8}, we get
\begin{align*}
\mathcal{I}_{\mu}(u_{\delta}^{\tau})&\leq\frac{s}{3}\mathcal{S}_s^{\frac{3}{2s}}+C\delta^{3-2s}+C\delta^{2s}+C\|u_{\delta}\|_{L^{\frac{12}{3+2t}}}^4-C\lambda\int_{\mathbb{R}^3}|u_{\delta}|^{q}\, {\rm d}x\\
&\leq\frac{s}{3}\mathcal{S}_s^{\frac{3}{2s}}+C\delta^{2s}+C\|u_{\delta}\|_{L^{\frac{12}{3+2t}}}^4-C\lambda\int_{\mathbb{R}^3}|u_{\delta}|^{q}\, {\rm d}x.
\end{align*}

Observing that $\frac{3}{3-2s}\in(\frac{3}{2},2)$, thus $\frac{12}{3+2t}>\frac{3}{3-2s}$ and $\frac{3}{3-2s}<\frac{4s+2t}{s+t}<q<\frac{6}{3-2s}$. Hence
\begin{align*}
\lim_{\delta\rightarrow0^{+}}\frac{\Big(\int_{\mathbb{R}^3}|u_{\delta}|^{\frac{12}{3+2t}}\,{\rm d}x\Big)^{\frac{3+2t}{3}}}{\delta^{2s}}\leq
\lim\limits_{\delta\rightarrow0^{+}}\frac{O(\delta^{2t+4s-3})}{\delta^{2s}}=0,\quad \frac{12}{3+2t}>\frac{3}{3-2s}                                                                                                                             \end{align*}
and for any $\lambda>0$, we have
\begin{align*}
\lim_{\delta\rightarrow0^{+}}\lambda\frac{\int_{\mathbb{R}^3}|u_{\delta}|^q\,{\rm d}x}{\delta^{2s}}=
\lim\limits_{\delta\rightarrow0^{+}}\lambda\frac{O(\delta^{3-\frac{3-2s}{2}q})}{\delta^{2s}}=+\infty,\quad \frac{4s+2t}{s+t}<q<\frac{6}{3-2s}.
\end{align*}
From the above arguments, we conclude the proof.
\end{proof}

From the estimate of mountain pass level, using the Vanishing Lemma, it is not difficult to deduce that the bounded sequence $\{u_n\}\subset H^s(\mathbb{R}^3)$ given in \eqref{equ3-5} is non-vanishing. That is,
\begin{lemma}\label{lem3-5}
There exists a sequence $\{x_n\}\subset\mathbb{R}^3$ and $R>0$, $\beta>0$ such that $\int_{B_R(x_n)}|u_n|^2\,{\rm d}x\geq\beta$.
\end{lemma}

Combining with Lemma \ref{lem3-3} and Lemma \ref{lem3-5}, we can show the existence of positive ground state solution for the limiting problem \eqref{equ3-1}.
\begin{proposition}\label{pro3-3}
Problem \eqref{equ3-1} possesses a positive ground state solution $u\in H^s(\mathbb{R}^3)$.
\end{proposition}
\begin{proof}
Let $\{u_n\}$ be the sequence given in \eqref{equ3-5}. Set $\widehat{u}_n(x)=u_n(x+x_n)$, where $\{x_n\}$ is the sequence obtained in Lemma \ref{lem3-5}. Thus $\{\widehat{u}_n\}$ is still bounded in $H^s(\mathbb{R}^3)$ and so up to a subsequence, still denoted by $\{\widehat{u}_n\}$, we may assume that there exists $\widehat{u}\in H^s(\mathbb{R}^3)$ such that
\begin{equation*}
\left\{
  \begin{array}{ll}
    \widehat{u}_n\rightharpoonup \widehat{u} & \hbox{in $H^s(\mathbb{R}^3)$,} \\
    \widehat{u}_n\rightarrow \widehat{u} & \hbox{in $L_{loc}^p(\mathbb{R}^3)$ for all $1\leq p<2_s^{\ast}$,}\\
    \widehat{u}_n\rightarrow \widehat{u} & \hbox{a.e. $\mathbb{R}^3$.}
  \end{array}
\right.
\end{equation*}
It follows from Lemma \ref{lem3-5} that $\widehat{u}$ is nontrivial. Moreover, using $(iv)$ of Lemma \ref{lem2-1}, it is not difficult to verify that $\widehat{u}$ is a nontrivial solution of problem \eqref{equ3-1}, and since $f\in C^1(\mathbb{R}^3)$, standard arguments lead to $\mathcal{G}_{\mu}(\widehat{u})=0$. By Fatou's Lemma and \eqref{equ3-5}, we have
\begin{align*}
c_{\mu}&=b_{\mu}\leq\mathcal{I}_{\mu}(\widehat{u})=\mathcal{I}_{\mu}(\widehat{u})-\frac{1}{4s+2t-3}\mathcal{G}_{\mu}(\widehat{u})=\frac{s}{4s+2t-3}\int_{\mathbb{R}^3}\mu|\widehat{u}|^2\,{\rm d}x\\
&+\frac{s+t}{4s+2t-3}\int_{\mathbb{R}^3}\Big(f(\widehat{u})\widehat{u}-\frac{4s+2t}{s+t}F(\widehat{u})\Big)\,{\rm d}x+\frac{(2_s^{\ast}-4)s+(2_s^{\ast}-2)t}{2_s^{\ast}(4s+2t-3)}\int_{\mathbb{R}^3}(\widehat{u}^{+})^{2_s^{\ast}}\,{\rm d}x\\
&\leq\liminf_{n\rightarrow\infty}\Big[\frac{s+t}{4s+2t-3}\int_{\mathbb{R}^3}\Big(f(\widehat{u}_n)\widehat{u}_n-\frac{4s+2t}{s+t}F(\widehat{u}_n)\Big)\,{\rm d}x\\
&+\frac{(2_s^{\ast}-4)s+(2_s^{\ast}-2)t}{2_s^{\ast}(4s+2t-3)}\int_{\mathbb{R}^3}(\widehat{u}_n^{+})^{2_s^{\ast}}\,{\rm d}x+\frac{s}{4s+2t-3}\int_{\mathbb{R}^3}\mu|\widehat{u}_n|^2\,{\rm d}x\Big]\\
&=\liminf_{n\rightarrow\infty}\Big[\mathcal{I}_{\mu}(\widehat{u}_n)-\frac{1}{4s+2t-3}\mathcal{G}_{\mu}(\widehat{u}_n)\Big]\\
&=\liminf_{n\rightarrow\infty}\Big[\mathcal{I}_{\mu}(u_n)-\frac{1}{4s+2t-3}\mathcal{G}_{\mu}(u_n)\Big]=c_{\mu}
\end{align*}
which implies that $\widehat{u}_n\rightarrow \widehat{u}$ in $H^s(\mathbb{R}^3)$. Indeed, from the above inequality, we get that
\begin{equation*}
\int_{\mathbb{R}^3}\widehat{u}_n^2\,{\rm d}x\rightarrow\int_{\mathbb{R}^3}\widehat{u}^2\,{\rm d}x,\quad \int_{\mathbb{R}^3}(\widehat{u}_n^{+})^{2_s^{\ast}}\,{\rm d}x\rightarrow\int_{\mathbb{R}^3}(\widehat{u}^{+})^{2_s^{\ast}}\,{\rm d}x.
\end{equation*}
By virtue of the Brezis-Lieb Lemma and interpolation argument, we conclude that
\begin{equation*}
\widehat{u}_n\rightarrow\widehat{u}\quad \text{in}\,\, L^r(\mathbb{R}^3)\,\, \text{for all}\,\, 2\leq r\leq 2_s^{\ast}.
\end{equation*}
Hence, from the standard arguments, it follows that $\widehat{u}_n\rightarrow \widehat{u}$ in $H^s(\mathbb{R}^3)$. Therefore, by Lemma \ref{lem3-2}, we conclude that $\mathcal{I}_{\mu}(\widehat{u})=c_{\mu}$ and $\mathcal{I}'_{\mu}(\widehat{u})=0$.

Next, we show that the ground state solution of \eqref{equ3-1} is positive. Indeed, by standard argument to the proof Proposition 4.4 in \cite{Teng2}, using Lemma \ref{lem2-3} two times, we have that $\widehat{u}\in C^{2,\alpha}(\mathbb{R}^3)$ for some $\alpha\in(0,1)$ for $s>\frac{1}{2}$. Using $-\widehat{u}^{-}$ as a testing function, it is easy to see that $\widehat{u}\geq0$. Since $\widehat{u}\in C^{2,\alpha}(\mathbb{R}^3)$, by Lemma 3.2 in \cite{NPV}, we have that
\begin{equation*}
(-\Delta)^s\widehat{u}(x)=-\frac{1}{2}C_s\int_{\mathbb{R}^3}\frac{\widehat{u}(x+y)+\widehat{u}(x-y)-2\widehat{u}(x)}{|x-y|^{3+2s}}\,{\rm d}x\,{\rm d}y,\quad \forall\,\, x\in\mathbb{R}^3.
\end{equation*}
Assume that there exists $x_0\in\mathbb{R}^3$ such that $\widehat{u}(x_0)=0$, then from $\widehat{u}\geq0$ and $\widehat{u}\not\equiv0$, we get
\begin{equation*}
(-\Delta)^s\widehat{u}(x_0)=-\frac{1}{2}C_s\int_{\mathbb{R}^3}\frac{\widehat{u}(x_0+y)+\widehat{u}(x_0-y)}{|x_0-y|^{3+2s}}\,{\rm d}x\,{\rm d}y<0.
\end{equation*}
However, observe that $(-\Delta)^s\widehat{u}(x_0)=-\mu \widehat{u}(x_0)-(\phi_{\widehat{u}}^t\widehat{u})(x_0)+f(\widehat{u}(x_0))+\widehat{u}(x_0)^{2_s^{\ast}-1}=0$, a contradiction. Hence, $\widehat{u}(x)>0$, for every $x\in\mathbb{R}^3$. The proof is completed.

\end{proof}

Let $\mathcal{L}_{\mu}$ be the set of ground state solutions $W$ of \eqref{equ3-1} satisfying $W(0)=\max\limits_{\mathbb{R}^3}W(x)$. By similar proof of Proposition 3.8 in \cite{Teng4}, we can establish the following compactness of $\mathcal{L}_{\mu}$.

\begin{proposition}\label{pro3-4}
$(i)$ For each $\mu>0$, $\mathcal{L}_{\mu}$ is compact in $H^s(\mathbb{R}^3)$.\\
$(ii)$ $0<W(x)\leq\frac{C}{1+|x|^{3+2s}}$, for any $x\in\mathbb{R}^3$.
\end{proposition}

\section{The penalization scheme}

For the bounded domain $\Lambda$ given in $(V_1)$, $k>2$, $a>0$ such that $f(a)+a^{2_s^{\ast}-1}=\frac{V_0}{k}a$ where $\alpha_0$ is mentioned in $(V_0)$, we consider a new problem
\begin{equation}\label{main-4-1}
(-\Delta)^su+V(\varepsilon z)u+\phi_u^t u=g(\varepsilon z,u) \quad \text{in}\,\,\mathbb{R}^3
\end{equation}
where $g(\varepsilon z,\tau)=\chi_{\Lambda_{\varepsilon}}(\varepsilon z)(f(\tau)+(\tau^{+})^{2_s^{\ast}-1})+(1-\chi_{\Lambda_{\varepsilon}}(\varepsilon z))\tilde{f}(\tau)$ with
\begin{equation*}
\tilde{f}(\tau)=\left\{
  \begin{array}{ll}
    f(\tau)+(\tau^{+})^{2_s^{\ast}-1} & \hbox{if $\tau\leq a$,} \\
    \frac{V_0}{k}\tau & \hbox{if $\tau>a$}
  \end{array}
\right.
\end{equation*}
and $\chi_{\Lambda_{\varepsilon}}(\varepsilon z)=1$ if $z\in\Lambda_{\varepsilon}$, $\chi(z)=0$ if $z\not\in\Lambda_{\varepsilon}$, where $\Lambda_{\varepsilon}=\Lambda/\varepsilon$. It is easy to see that under the assumptions $(f_1)$-$(f_3)$, $g(z,\tau)$ is a Caratheodory function and satisfies the following assumptions:\\
$(g_1)$ $g(z,\tau)=o(\tau^3)$ as $\tau\rightarrow0$ uniformly on $z\in\mathbb{R}^3$;\\
$(g_2)$ $g(z,\tau)\leq f(\tau)+\tau^{2_s^{\ast}-1}$ for all $\tau\in\mathbb{R}^{+}$ and $z\in\mathbb{R}^3$, $g(z,\tau)=0$ for all $z\in\mathbb{R}^3$ and $\tau<0$, $g(z,\tau)=f(\tau)+(\tau^{+})^{2_s^{\ast}-1}$ for $z\in\mathbb{R}^3$, $\tau\in[0,a]$;\\
$(g_3)$ $0<2\tilde{F}(\tau)\leq\tilde{f}(\tau)\tau\leq\frac{V_0}{k}\tau^2\leq\frac{V(x)}{k}\tau^2$ for all $s\geq0$ with the number $k>2$, where $\tilde{F}(\tau)$ is a prime function of $\tilde{f}$;\\
$(g_4)$ $0<q G(z,\tau)\leq g(z,\tau)\tau$ for all $z\in\Lambda$, $\tau>0$ or $z\in\mathbb{R}^3\backslash\Lambda$, $\tau\leq a$, where $G(z,\tau)$ is a prime function of $g(z,\tau)$;\\
$(g_5)$ $\frac{g(z,s\tau)}{s}$ is nondecreasing in $\tau\in\mathbb{R}^{+}$ uniformly for $z\in\mathbb{R}^3$, $\frac{g(z,s\tau)}{\tau^q}$ is nondecreasing in $\tau\in\mathbb{R}^{+}$ and $z\in\Lambda$, $\frac{g(z,s\tau)}{\tau^q}$ is nondecreasing in $\tau\in(0,a)$ and $z\in\mathbb{R}^3\backslash\Lambda$.

Obviously, if $u_{\varepsilon}$ is a solution of \eqref{main-4-1} satisfying $u_{\varepsilon}(z)\leq a$ for $z\in\mathbb{R}^3$, then $u_{\varepsilon}$ is indeed a solution of the original problem \eqref{R-1}.

For $u\in H_{\varepsilon}$, let
\begin{equation*}
P_{\varepsilon}(u)=\frac{1}{2}\int_{\mathbb{R}^3}(|D_su|^2+V(\varepsilon z)u^2)\,{\rm d}z+\frac{1}{4}\int_{\mathbb{R}^3}\phi_u^tu^2\,{\rm d}z-\int_{\mathbb{R}^3}G(\varepsilon z,u)\,{\rm d}z
\end{equation*}
and
\begin{equation*}
Q_{\varepsilon}(u)=\Big(\int_{\mathbb{R}^3\backslash\Lambda_{\varepsilon}}u^2\,{\rm d}x-\varepsilon\Big)_{+}^2.
\end{equation*}
Let us define the functional $\mathcal{J}_{\varepsilon}: H_{\varepsilon}\rightarrow\mathbb{R}$ as follows
\begin{equation*}
\mathcal{J}_{\varepsilon}(u)=P_{\varepsilon}(u)+Q_{\varepsilon}(u).
\end{equation*}
Clearly, $\mathcal{J}_{\varepsilon}\in C^1(H_{\varepsilon},\mathbb{R})$. To find solutions of \eqref{main-4-1} which concentrates in $\Lambda$ as $\varepsilon\rightarrow0$, we shall search critical points of $\mathcal{J}_{\varepsilon}$ such that $Q_{\varepsilon}$ is zero.

Set
\begin{equation*}
\delta_0=\frac{1}{10}{\rm dist}(\mathcal{M},\mathbb{R}^3\backslash\Lambda),\quad \beta\in(0,\delta_0).
\end{equation*}
Fix a cut-off function $\varphi\in C_0^{\infty}(\mathbb{R}^3)$ such that $0\leq\varphi\leq1$, $\varphi=1$ for $|z|\leq\beta$, $\varphi=0$ for $|z|\geq 2\beta$ and $|\nabla\varphi|\leq C/\beta$. Set $\varphi_{\varepsilon}(z)=\varphi(\varepsilon z)$, for any $W\in \mathcal{L}_{V_0}$ and any point $y\in\mathcal{M}^{\beta}=\{y\in\mathbb{R}^3\,\,|\,\,\inf\limits_{z\in\mathcal{M}}|y-z|\leq\beta\}$, we define
\begin{equation*}
W_{\varepsilon}^{y}(z)=\varphi_{\varepsilon}(z-\frac{y}{\varepsilon})W(z-\frac{y}{\varepsilon}).
\end{equation*}
For $A\subset H_{\varepsilon}$, we use the notation
\begin{equation*}
A^{a}=\{u\in H_{\varepsilon}\,\,\Big|\,\,\inf_{v\in A}\|u-v\|_{H_{\varepsilon}}\leq a\}.
\end{equation*}

We want to find a solution near the set
\begin{equation*}
\mathcal{N}_{\varepsilon}=\{W_{\varepsilon}^{y}(z)\,\,\Big|\,\,y\in\mathcal{M}^{\beta}, \,\, W\in \mathcal{L}_{V_0}\}
\end{equation*}
for $\varepsilon>0$ sufficiently small.

Similar arguments as the proof of Lemma 4.1, Lemma 4.2 and Lemma 4.3 in \cite{Teng4}, we can show that\\
\begin{itemize}
  \item $\mathcal{N}_{\varepsilon}$ is uniformly bounded in $H_{\varepsilon}$ and it is compact in $H_{\varepsilon}$ for any $\varepsilon>0$;
  \item \begin{equation*}
\sup_{\tau\in[0,\tau_0]}\Big|\mathcal{J}_{\varepsilon}(W_{\varepsilon,\tau})-\mathcal{I}_{V_0}(W^{\ast}_{\tau}(z))\Big|\rightarrow0\quad \text{as}\,\, \varepsilon\rightarrow0;
\end{equation*}
\item \begin{equation}\label{equ4-1}
\lim_{\varepsilon\rightarrow0}\mathcal{C}_{\varepsilon}=\lim_{\varepsilon\rightarrow0}\mathcal{D}_{\varepsilon}:=\lim_{\varepsilon\rightarrow0}\max_{\tau\in[0,1]}\mathcal{J}_{\varepsilon}(\gamma_{\varepsilon}(\tau))=c_{V_0}
\end{equation}
\end{itemize}
where \begin{equation*}
\mathcal{C}_{\varepsilon}:=\inf_{\gamma\in\mathcal{A}_{\varepsilon}}\max_{\tau\geq0}\mathcal{J}_{\varepsilon}(\gamma(\tau)),
\end{equation*}
$\mathcal{A}_{\varepsilon}=\{\gamma\in C([0,1],H_{\varepsilon})\,\,|\,\, \gamma(0)=0,\,\,\gamma(1)=U_{\varepsilon,\tau_0}\}$, $\gamma_{\varepsilon}(\tau)=W_{\varepsilon,\tau\tau_0}$ for $\tau\in[0,1]$ and $c_{V_0}=\mathcal{I}_{V_0}(W^{\ast})$ for $W^{\ast}\in\mathcal{L}_{V_0}$. Moreover, $\mathcal{J}_{\varepsilon}(U_{\varepsilon,\tau})$ possesses the mountain-pass geometry.

\begin{lemma}\label{lem4-1}
There exists a small $d_0>0$ such that for any $\{\varepsilon_i\}$, $\{u_{\varepsilon_i}\}$ satisfying $\lim\limits_{i\rightarrow\infty}\varepsilon_i\rightarrow0$, $u_{\varepsilon_i}\in \mathcal{N}_{\varepsilon_i}^{d_0}$ and
\begin{equation*}
\lim_{i\rightarrow\infty}\mathcal{J}_{\varepsilon_i}(u_{\varepsilon_i})\leq c_{V_0}\quad \text{and}\quad \lim_{i\rightarrow\infty}\mathcal{J}_{\varepsilon_i}'(u_{\varepsilon_i})=0,
\end{equation*}
there exist, up to a subsequence, $\{x_i\}\subset\mathbb{R}^3$, $x_0\in\mathcal{M}$, $W\in \mathcal{L}_{V_0}$ such that
\begin{equation*}
\lim_{i\rightarrow\infty}|\varepsilon_ix_i-x_0|=0\quad \text{and}\quad \lim_{i\rightarrow\infty}\|u_{\varepsilon_i}-\varphi_{\varepsilon}(\cdot-x_{i})W(\cdot-x_{i})\|_{E_{\varepsilon_i}}=0.
\end{equation*}
\end{lemma}

\begin{proof}
In the proof we will drop the index $i$ and write $\varepsilon$ instead of $\varepsilon_i$ for simplicity, and we still use $\varepsilon$ after taking a subsequence. By the definition of $\mathcal{N}_{\varepsilon}^{d_0}$, there exist $\{W_{\varepsilon}\}\subset\mathcal{L}_{V_0}$ and $\{x_{\varepsilon}\}\subset\mathcal{M}^{\beta}$ such that
\begin{equation*}
\|u_{\varepsilon}-\varphi_{\varepsilon}(\cdot-\frac{x_{\varepsilon}}{\varepsilon})W_{\varepsilon}(\cdot-\frac{x_{\varepsilon}}{\varepsilon})\|_{H_{\varepsilon}}\leq\frac{3}{2}d_0.
\end{equation*}
Since $\mathcal{L}_{V_0}$ and $\mathcal{M}^{\beta}$ are compact, there exist $W_0\in\mathcal{L}_{V_0}$, $x_0\in\mathcal{M}^{\beta}$ such that $W_{\varepsilon}\rightarrow W_0$ in $H^s(\mathbb{R}^3)$ and $x_{\varepsilon}\rightarrow x_0$ as $\varepsilon\rightarrow0$. Thus, for $\varepsilon>0$ small,
\begin{equation}\label{equ4-2}
\|u_{\varepsilon}-\varphi_{\varepsilon}(\cdot-\frac{x_{\varepsilon}}{\varepsilon})W_0(\cdot-\frac{x_{\varepsilon}}{\varepsilon})\|_{H_{\varepsilon}}\leq2d_0.
\end{equation}

{\bf Step 1.} We claim that
\begin{equation}\label{equ4-3}
\lim_{\varepsilon\rightarrow0}\sup_{y\in A_{\varepsilon}}\int_{B_1(y)}|u_{\varepsilon}|^{2_s^{\ast}}\,{\rm d}z=0,
\end{equation}
where $A_{\varepsilon}=B_{3\beta/\varepsilon}(x_{\varepsilon}/\varepsilon)\backslash B_{\beta/2\varepsilon}(x_{\varepsilon}/\varepsilon)$. Suppose by contradiction that there exists $r>0$ such that
\begin{equation*}
\liminf_{\varepsilon\rightarrow0}\sup_{y\in A_{\varepsilon}}\int_{B_1(y)}|u_{\varepsilon}|^{2_s^{\ast}}\,{\rm d}z=2r>0.
\end{equation*}
Thus, there exists $y_{\varepsilon}\in A_{\varepsilon}$ such that $\int_{B_1(y_{\varepsilon})}|u_{\varepsilon}|^{2_s^{\ast}}\,{\rm d}z\geq r>0$ for $\varepsilon>0$ small. Since $y_{\varepsilon}\in A_{\varepsilon}$, there exists $y^{\ast}\in\mathcal{M}^{4\beta}\subset\Lambda$ such that $\varepsilon y_{\varepsilon}\rightarrow y^{\ast}$ as $\varepsilon\rightarrow0$. Set $v_{\varepsilon}(z)=u_{\varepsilon}(z+y_{\varepsilon})$, then for $\varepsilon>0$ small,
\begin{equation}\label{equ4-4}
\int_{B_1(0)}|v_{\varepsilon}|^{2_s^{\ast}}\,{\rm d}z\geq r>0.
\end{equation}
Thus, up to a subsequence, we may assume that there exists $v\in H^s(\mathbb{R}^3)$ such that $v_{\varepsilon}\rightharpoonup v$ in $H^s(\mathbb{R}^3)$, $v_{\varepsilon}\rightarrow v$ in $L_{loc}^p(\mathbb{R}^3)$ for $1\leq p<2_s^{\ast}$ and $v_{\varepsilon}\rightarrow v$ a.e. in $\mathbb{R}^3$. It is easy to check that $v$ satisfies
\begin{equation}\label{equ4-4-0}
(-\Delta)^s v+V(y^{\ast}) v+\phi_{v}^t v=f(v)+(v^{+})^{2_s^{\ast}-1}    \quad x\in \mathbb{R}^3.
\end{equation}
Indeed, by the definition of weakly convergence, we have
\begin{align*}
\int_{\mathbb{R}^3}\int_{\mathbb{R}^3}\frac{(v_{\varepsilon}(z)-v_{\varepsilon}(y))(\varphi(z)-\varphi(y)}{|z-y|^{3+2s}}\,{\rm d }y\,{\rm d }z+\int_{\mathbb{R}^3}V(y^{\ast})v_{\varepsilon}\varphi\,{\rm d}z\rightarrow\\
\int_{\mathbb{R}^3}\int_{\mathbb{R}^3}\frac{(v(z)-v(y))(\varphi(z)-\varphi(y)}{|z-y|^{3+2s}}\,{\rm d }y\,{\rm d }z+\int_{\mathbb{R}^3}V(y^{\ast})v\varphi\,{\rm d}z
\end{align*}
for any $\varphi\in C_0^{\infty}(\mathbb{R}^3)$. Now given $\varphi\in C_0^{\infty}(\mathbb{R}^3)$, we have $\|\varphi(\cdot-y_{\varepsilon})\|_{H_{\varepsilon}}\leq C$ and so $\langle\mathcal{J}_{\varepsilon}'(u_{\varepsilon}),\varphi(\cdot-y_{\varepsilon})\rangle\rightarrow0$ as $\varepsilon\rightarrow0$. Using the fact that $v_{\varepsilon}\rightarrow v$ in $L_{loc}^p(\mathbb{R}^3)$ for $1\leq p<2_s^{\ast}$, the Lebesgue dominated convergence Theorem, the boundedness of ${\rm supp}(\varphi)$ and $(g_0)$--$(g_1)$, it follows that
\begin{align*}
\int_{\mathbb{R}^3}(V(\varepsilon z+\varepsilon y_{\varepsilon})-V(y^{\ast}))v_{\varepsilon}\varphi\,{\rm d }z\rightarrow0,\quad \int_{\mathbb{R}^3}(\phi_{v_{\varepsilon}}^tv_{\varepsilon}-\phi_v^tv)\varphi\,{\rm d }z\rightarrow0,
\end{align*}
\begin{align*}
\int_{\mathbb{R}^3\backslash\Lambda_{\varepsilon}}u_{\varepsilon}(z)\varphi(z-y_{\varepsilon})\,{\rm d}z=\int_{\mathbb{R}^3\backslash\Lambda_{\varepsilon}+y_{\varepsilon}}v_{\varepsilon}(z)\varphi(z)\,{\rm d}z\rightarrow0
\end{align*}
and
\begin{align*}
\int_{\mathbb{R}^3}\Big[g(\varepsilon z+\varepsilon y_{\varepsilon},v_{\varepsilon})-f(v)-(v^{+})^{2_s^{\ast}-1}\Big]\varphi\,{\rm d }z\rightarrow0
\end{align*}
for any $\varphi\in C_0^{\infty}(\mathbb{R}^3)$. Therefore, we get that
\begin{align*}
\int_{\mathbb{R}^3}\int_{\mathbb{R}^3}\frac{(v(z)-v(y))(\varphi(z)-\varphi(y))}{|z-y|^{3+2s}}\,{\rm d}y\,{\rm d}z&+\int_{\mathbb{R}^3}V(y^{\ast})v\varphi\,{\rm d}z&+\int_{\mathbb{R}^3}\phi_v^tv\varphi\,{\rm d}z\\
&-\int_{\mathbb{R}^3}g(v)\varphi\,{\rm d}z=0
\end{align*}
for any $\varphi\in C_0^{\infty}(\mathbb{R}^3)$. Since $\varphi$ is arbitrary and $C_0^{\infty}(\mathbb{R}^3)$ is dense in $H_{\varepsilon}$, it follows that $v$ satisfies \eqref{equ4-4-0}.

{\bf Case 1.} If $v\neq0$, then
\begin{align*}
&c_{V(y^{\ast})}\leq \mathcal{I}_{V(y^{\ast})}(v)=\mathcal{I}_{V(y^{\ast})}(v)-\frac{1}{4s+2t-3}\mathcal{G}_{V(y^{\ast})}(v)\\
&=s\int_{\mathbb{R}^3}V(y^{\ast})|v|^2\,{\rm d}z+\frac{s+t}{4s+2t-3}\int_{\mathbb{R}^3}[f(v)v-\frac{4s+2t}{s+t}F(v)]\,{\rm d}z\\
&+\frac{(2_s^{\ast}-4)s+(2_s^{\ast}-2)t}{2_s^{\ast}(4s+2t-3)}\int_{\mathbb{R}^3}(v^{+})^{2_s^{\ast}}\,{\rm d}z\\
&\leq s\|V\|_{L^{\infty}(\overline{\Lambda})}\int_{\mathbb{R}^3}|v|^2\,{\rm d}z+\frac{s+t}{4s+2t-3}\int_{\mathbb{R}^3}[f(v)v-\frac{4s+2t}{s+t}F(v)]\,{\rm d}z\\
&+\frac{(2_s^{\ast}-4)s+(2_s^{\ast}-2)t}{2_s^{\ast}(4s+2t-3)}\int_{\mathbb{R}^3}(v^{+})^{2_s^{\ast}}\,{\rm d}z.
\end{align*}
Hence, for sufficiently large $R>0$, by Fatou's Lemma, we have that
\begin{align*}
&\liminf_{\varepsilon\rightarrow0}\Big[s\|V\|_{L^{\infty}(\overline{\Lambda})}\int_{B_R(y_{\varepsilon})}|u_{\varepsilon}|^2\,{\rm d}z+\frac{s+t}{4s+2t-3}\int_{B_R(y_{\varepsilon})}[f(u_{\varepsilon})u_{\varepsilon}-\frac{4s+2t}{s+t}F(u_{\varepsilon})]\,{\rm d}z\\
&+\frac{(2_s^{\ast}-4)s+(2_s^{\ast}-2)t}{2_s^{\ast}(4s+2t-3)}\int_{B_R(y_{\varepsilon})}(u_{\varepsilon}^{+})^{2_s^{\ast}}\,{\rm d}z\Big]\\
&=\liminf_{\varepsilon\rightarrow0}\Big[s\|V\|_{L^{\infty}(\overline{\Lambda})}\int_{B_R(0)}|v_{\varepsilon}|^2\,{\rm d}z+\frac{s+t}{4s+2t-3}\int_{B_R(0)}[f(v_{\varepsilon})v_{\varepsilon}-\frac{4s+2t}{s+t}F(v_{\varepsilon})]\,{\rm d}z\\
&+\frac{(2_s^{\ast}-4)s+(2_s^{\ast}-2)t}{2_s^{\ast}(4s+2t-3)}\int_{B_R(0)}(v_{\varepsilon}^{+})^{2_s^{\ast}}\,{\rm d}z\Big]\\
&\geq\Big[s\|V\|_{L^{\infty}(\overline{\Lambda})}\int_{B_R(0)}|v|^2\,{\rm d}z+\frac{s+t}{4s+2t-3}\int_{B_R(0)}[f(v)v-\frac{4s+2t}{s+t}F(v)]\,{\rm d}z\\
&+\frac{(2_s^{\ast}-4)s+(2_s^{\ast}-2)t}{2_s^{\ast}(4s+2t-3)}\int_{B_R(0)}(v^{+})^{2_s^{\ast}}\,{\rm d}z\Big]\\
&\geq\frac{1}{2}\Big[s\|V\|_{L^{\infty}(\overline{\Lambda})}\int_{\mathbb{R}^3}|v|^2\,{\rm d}z+\frac{s+t}{4s+2t-3}\int_{\mathbb{R}^3}[f(v)v-\frac{4s+2t}{s+t}F(v)]\,{\rm d}z\\
&+\frac{(2_s^{\ast}-4)s+(2_s^{\ast}-2)t}{2_s^{\ast}(4s+2t-3)}\int_{B_R(0)}(v^{+})^{2_s^{\ast}}\,{\rm d}z\Big]\geq\frac{1}{2}c_{V(x^{\ast})}>0.
\end{align*}
On the other hand, by the Sobolev embedding theorem and \eqref{equ4-2}, one has
\begin{align*}
&s\|V\|_{L^{\infty}(\overline{\Lambda})}\int_{B_R(y_{\varepsilon})}|u_{\varepsilon}|^2\,{\rm d}z+\frac{s+t}{4s+2t-3}\int_{B_R(y_{\varepsilon})}[f(u_{\varepsilon})u_{\varepsilon}-\frac{4s+2t}{s+t}F(u_{\varepsilon})]\,{\rm d}z\\
&+\frac{(2_s^{\ast}-4)s+(2_s^{\ast}-2)t}{2_s^{\ast}(4s+2t-3)}\int_{B_R(y_{\varepsilon})}(u_{\varepsilon}^{+})^{2_s^{\ast}}\,{\rm d}z\\
&\leq Cd_0+C\int_{B_R(y_{\varepsilon})}\Big|\varphi(\varepsilon z-x_{\varepsilon})W_0(z-\frac{x_{\varepsilon}}{\varepsilon})\Big|^2\,{\rm d}z+C\int_{B_R(y_{\varepsilon})}\Big|\varphi(\varepsilon z-x_{\varepsilon})W_0(z-\frac{x_{\varepsilon}}{\varepsilon})\Big|^{2_s^{\ast}}\,{\rm d}z\\
&\leq Cd_0+C\int_{B_R(y_{\varepsilon}-\frac{x_{\varepsilon}}{\varepsilon})}|W_0|^2\,{\rm d}x+C\int_{B_R(y_{\varepsilon}-\frac{x_{\varepsilon}}{\varepsilon})}|W_0|^{2_s^{\ast}}\,{\rm d}x
\end{align*}
Observing that $y_{\varepsilon}\in A_{\varepsilon}$, implies that $|y_{\varepsilon}-\frac{x_{\varepsilon}}{\varepsilon}|\geq\frac{\beta}{2\varepsilon}$, then for $\varepsilon>0$ small enough, there hold
\begin{equation*}
\int_{B_R(y_{\varepsilon}-\frac{x_{\varepsilon}}{\varepsilon})}|W_0|^2\,{\rm d}z+\int_{B_R(y_{\varepsilon}-\frac{x_{\varepsilon}}{\varepsilon})}|W_0|^{2_s^{\ast}}\,{\rm d}z=o(1),
\end{equation*}
where $o(1)\rightarrow0$ as $\varepsilon\rightarrow0$. Thus, we have proved that
\begin{align*}
&\frac{1}{2}c_{V(y^{\ast})}\leq s\|V\|_{L^{\infty}(\overline{\Lambda})}\int_{B_R(y_{\varepsilon})}|u_{\varepsilon}|^2\,{\rm d}z+\frac{s+t}{4s+2t-3}\int_{B_R(y_{\varepsilon})}[f(u_{\varepsilon})u_{\varepsilon}\\
&-\frac{4s+2t}{s+t}F(u_{\varepsilon})]\,{\rm d}z+\frac{(2_s^{\ast}-4)s+(2_s^{\ast}-2)t}{2_s^{\ast}(4s+2t-3)}\int_{B_R(y_{\varepsilon})}(u_{\varepsilon}^{+})^{2_s^{\ast}}\,{\rm d}z\\
&\leq Cd_0+o(1).
\end{align*}
This leads to a contradiction if $d_0$ is small enough.

{\bf Case 2.} If $v=0$, i.e., $v_{\varepsilon}\rightharpoonup0$ in $H^s(\mathbb{R}^3)$, $v_{\varepsilon}\rightarrow 0$ in $L_{loc}^p(\mathbb{R}^3)$ for $1\leq p<2_s^{\ast}$ and $v_{\varepsilon}\rightarrow 0$ a.e. in $\mathbb{R}^3$. Now we claim that
\begin{equation}\label{equ4-5}
\lim_{\varepsilon\rightarrow0}\sup_{\varphi\in C_0^{\infty}(B_2(0)),\|\varphi\|=1}|\langle\rho_{\varepsilon},\varphi\rangle|=0,
\end{equation}
where $\rho_{\varepsilon}=-(-\Delta)^sv_{\varepsilon}+(v_{\varepsilon}^{+})^{2_s^{\ast}-1}\in \Big(H^s(\mathbb{R}^3)\Big)'$. It is easy to check that for $\varepsilon>0$ small, $\int_{\mathbb{R}^3\backslash\Lambda_{\varepsilon}}u_{\varepsilon}\varphi(z-y_{\varepsilon})\,{\rm d}z=0$ uniformly for any $\varphi\in C_0^{\infty}(B_2(0))$.
Thus, for any $\varphi\in C_0^{\infty}(B_2(0))$ with $\|\varphi\|=1$, we deduce that
\begin{align*}
&\langle\rho_{\varepsilon},\varphi\rangle=-\int_{\mathbb{R}^3}\frac{(v_{\varepsilon}(z)-v_{\varepsilon}(y))(\varphi(z)-\varphi(y))}{|z-y|^{3+2s}}\,{\rm d}y\,{\rm d}z+\int_{\mathbb{R}^3}(v_{\varepsilon}^{+})^{2_s^{\ast}-1}\varphi\,{\rm d}z\\
&=-\langle\mathcal{J}'_{\varepsilon}(u_{\varepsilon}),\varphi(\cdot-y_{\varepsilon})\rangle+\int_{\mathbb{R}^3}V(\varepsilon z)u_{\varepsilon}\varphi(z-y_{\varepsilon})\,{\rm d}z+\int_{\mathbb{R}^3}\phi_{u_{\varepsilon}}^tu_{\varepsilon}\varphi(z-y_{\varepsilon})\,{\rm d}z\\
&+\int_{\mathbb{R}^3}[(u_{\varepsilon}^{+})^{2_s^{\ast}-1}-g(\varepsilon z, u_{\varepsilon})]\varphi(z-y_{\varepsilon})\,{\rm d}z\\
&:=A_1+A_2+A_3+A_4.
\end{align*}
By the facts that $\lim\limits_{\varepsilon\rightarrow0}\mathcal{J}'_{\varepsilon}(u_{\varepsilon})=0$, ${\rm supp}\varphi\subset B_2$, $\sup\limits_{z\in B_2(0)}V(\varepsilon z+\varepsilon y_{\varepsilon})\leq C$ uniformly for all $\varepsilon>0$ small, $v_{\varepsilon}\rightarrow0$ in $L_{loc}^p(\mathbb{R}^3)$ for $1\leq p<2_s^{\ast}$, $\frac{12}{3+2t}<2_s^{\ast}$, and using H\"{o}lder's inequality, we deduce that
\begin{align*}
|A_1|\leq\|\mathcal{J}'_{\varepsilon}(u_{\varepsilon})\|_{(H_{\varepsilon})'}\|\varphi(\cdot-y_{\varepsilon})\|_{H_{\varepsilon}}\leq o(1)\|\varphi(\cdot-y_{\varepsilon})\|=o(1)\|\varphi\|\rightarrow0,
\end{align*}
\begin{align*}
|A_2|\leq\sup_{x\in B_2(0)}V(\varepsilon x+\varepsilon y_{\varepsilon})\Big(\int_{B_2(0)}|v_{\varepsilon}|^2\,{\rm d}z\Big)^{\frac{1}{2}}\Big(\int_{B_2(0)}|\varphi|^2\,{\rm d}z\Big)^{\frac{1}{2}}\rightarrow0,
\end{align*}
\begin{align*}
|A_3|&\leq\Big(\int_{\mathbb{R}^3}|\phi_{v_{\varepsilon}}^t|^{2_t^{\ast}}\,{\rm d}z\Big)^{\frac{1}{2_t^{\ast}}}\Big(\int_{B_2(0)}|v_{\varepsilon}|^{\frac{12}{3+2t}}\,{\rm d}z\Big)^{\frac{3+2t}{12}}\Big(\int_{B_2(0)}|\varphi|^{\frac{12}{3+2t}}\,{\rm d}z\Big)^{\frac{3+2t}{12}}\rightarrow0,
\end{align*}
uniformly for $\varphi\in C_0^{\infty}(B_2(0))$ with $\|\varphi\|=1$. By $(f_0)$ and $(f_1)$, for any $\eta>0$, there exists $C_{\eta}>0$ such that
\begin{equation*}
f(\tau)\leq C_{\eta}|\tau|+\eta|\tau|^{2_s^{\ast}-1}\quad \text{for any}\,\, \tau\geq0
\end{equation*}
and
\begin{equation*}
\lim_{\tau\rightarrow+\infty}\frac{g(\varepsilon z+\varepsilon y_{\varepsilon},\tau)-f(\tau)-(\tau^{+})^{2_s^{\ast}-1}}{(\tau^{+})^{2_s^{\ast}-1}}=0
\end{equation*}
uniformly for $z\in B_2(0)$ and for $\varepsilon$ small enough and then, we have
\begin{align*}
|A_4|&=\Big|\int_{\mathbb{R}^3}[(v_{\varepsilon}^{+})^{2_s^{\ast}-1}-g(\varepsilon z+\varepsilon y_{\varepsilon}, v_{\varepsilon})]\varphi\,{\rm d}z\Big|=\Big|-\int_{\mathbb{R}^3}f( v_{\varepsilon})\varphi\,{\rm d}z\\
&+\int_{\mathbb{R}^3}[(v_{\varepsilon}^{+})^{2_s^{\ast}-1}+f(v_{\varepsilon})-g(\varepsilon z+\varepsilon y_{\varepsilon}, v_{\varepsilon})]\varphi\,{\rm d}z\Big|\\
&\leq\eta\Big(\int_{\mathbb{R}^3}|v_{\varepsilon}|^{2_s^{\ast}}\,{\rm d}z\Big)^{\frac{2_s^{\ast}-1}{2_s^{\ast}}}\Big(\int_{\mathbb{R}^3}|\varphi|^{2_s^{\ast}}\,{\rm d}z\Big)^{\frac{1}{2_s^{\ast}}}+C_{\eta}\Big(\int_{B_2(0)}|v_{\varepsilon}|^2\,{\rm d}z\Big)^{\frac{1}{2}}\Big(\int_{B_2(0)}|\varphi|^{2}\,{\rm d}z\Big)^{\frac{1}{2}}\\
&+\Big|\int_{B_2(0)}[(v_{\varepsilon}^{+})^{2_s^{\ast}-1}+f(v_{\varepsilon})-g(\varepsilon z+\varepsilon y_{\varepsilon}, v_{\varepsilon})]\varphi\,{\rm d}z\Big|\\
&\leq2\eta\Big(\int_{\mathbb{R}^3}|v_{\varepsilon}|^{2_s^{\ast}}\,{\rm d}z\Big)^{\frac{2_s^{\ast}-1}{2_s^{\ast}}}\Big(\int_{\mathbb{R}^3}|\varphi|^{2_s^{\ast}}\,{\rm d}z\Big)^{\frac{1}{2_s^{\ast}}}+2C_{\eta}\Big(\int_{B_2(0)}|v_{\varepsilon}|^2\,{\rm d}z\Big)^{\frac{1}{2}}\Big(\int_{B_2(0)}|\varphi|^{2}\,{\rm d}z\Big)^{\frac{1}{2}}.
\end{align*}
for $\varepsilon$ sufficiently small. Letting $\varepsilon\rightarrow0$ and then $\eta\rightarrow0$ in the above inequality, we see that $A_4\rightarrow0$ as $\varepsilon\rightarrow0$ uniformly for $\varphi\in C_0^{\infty}(B_2(0))$ with $\|\varphi\|=1$. Hence \eqref{equ4-5} holds.

By Proposition \ref{pro2-3}, taking $Q=B_1(0)$ and $V=B_2(0)$, from \eqref{equ4-4} and \eqref{equ4-5}, it follows that there exists $\widetilde{z}_{\varepsilon}\in\mathbb{R}^3$ and $\sigma_{\varepsilon}>0$ with $\widetilde{z}_{\varepsilon}\rightarrow \widetilde{z}\in\overline{B_1(0)}$, $\sigma_{\varepsilon}\rightarrow0$ as $\varepsilon\rightarrow0$, such that
\begin{equation*}
\widetilde{v}_{\varepsilon}(z):=\sigma_{\varepsilon}^{\frac{3-2s}{2}}v_{\varepsilon}(\sigma_{\varepsilon}z+\widetilde{z}_{\varepsilon})\rightharpoonup \widetilde{v}\quad \text{in}\,\, \mathcal{D}^{s,2}(\mathbb{R}^3)
\end{equation*}
and $\widetilde{v}\geq0$ is a nontrivial solution of
\begin{equation}\label{equ4-6}
(-\Delta)^sv=v^{2_s^{\ast}-1},\quad v\in\mathcal{D}^{s,2}(\mathbb{R}^3).
\end{equation}
By Theorem 1.1 in \cite{CKT} and Claim 6 in \cite{SV}, we see that
\begin{equation*}
\widetilde{v}(z)=\frac{\bar{u}(z/\mathcal{S}_s^{\frac{1}{2s}})}{\|\bar{u}\|_{L^{2_s^{\ast}}}},\quad \bar{u}(z)=\kappa(\mu^2+|z-x_0|^2)^{-\frac{3-2s}{2}}
\end{equation*}
for some $\kappa>0$, $\mu>0$, $x_0\in\mathbb{R}^3$, and
\begin{equation}\label{equ4-6-0}
\int_{\mathbb{R}^3}|D_s\widetilde{v}|^2\,{\rm d}z=\int_{\mathbb{R}^3}|\widetilde{v}|^{2_s^{\ast}}\,{\rm d}z=\mathcal{S}_s^{\frac{3}{2s}}.
\end{equation}
Thus, there exists $R>0$ such that
\begin{equation*}
\int_{B_R(0)}|\widetilde{v}|^{2_s^{\ast}}\,{\rm d}z\geq\frac{1}{2}\int_{\mathbb{R}^3}|\widetilde{v}|^{2_s^{\ast}}\,{\rm d}z=\frac{1}{2}\mathcal{S}_s^{\frac{3}{2s}}>0.
\end{equation*}
On the other hand, using the facts that $\sigma_{\varepsilon}\rightarrow0$ and $\widetilde{z}_{\varepsilon}\rightarrow \widetilde{z}\in\overline{B_1(0)}$ (imply that $B_{\sigma_{\varepsilon}R}(\widetilde{z}_{\varepsilon}+z_{\varepsilon})\subset B_2(z_{\varepsilon})$ for $\varepsilon$ small), we have that
\begin{align}\label{equ4-7}
&\int_{B_R(0)}\widetilde{v}^{2_s^{\ast}}\,{\rm d}z\leq\liminf_{\varepsilon\rightarrow0}\int_{B_R(0)}\widetilde{v}_{\varepsilon}^{2_s^{\ast}}\,{\rm d}z=\liminf_{\varepsilon\rightarrow0}\int_{B_{\sigma_{\varepsilon}R}(\widetilde{z}_{\varepsilon})}v_{\varepsilon}^{2_s^{\ast}}\,{\rm d}z\nonumber\\
&=\liminf_{\varepsilon\rightarrow0}\int_{B_{\sigma_{\varepsilon}R}(\widetilde{z}_{\varepsilon}+y_{\varepsilon})}u_{\varepsilon}^{2_s^{\ast}}\,{\rm d}z\leq\liminf_{\varepsilon\rightarrow0}\int_{B_2(y_{\varepsilon})}u_{\varepsilon}^{2_s^{\ast}}\,{\rm d}z.
\end{align}
But, by the Sobolev imbedding Theorem and \eqref{equ4-2}, we get
\begin{align*}
\int_{B_2(y_{\varepsilon})}u_{\varepsilon}^{2_s^{\ast}}\,{\rm d}z&\leq Cd_0+C\int_{B_2(y_{\varepsilon})}(\varphi(\varepsilon x-x_{\varepsilon})W_0(x-\frac{x_{\varepsilon}}{\varepsilon}))^{2_s^{\ast}}\,{\rm d}x\\
&\leq Cd_0+C\int_{B_2(y_{\varepsilon}-\frac{x_{\varepsilon}}{\varepsilon})}W_0^{2_s^{\ast}}\,{\rm d}x=cd_0+o(1)
\end{align*}
since $|y_{\varepsilon}-\frac{x_{\varepsilon}}{\varepsilon}|\geq\frac{\beta}{2\varepsilon}$. This leads to a contradiction with \eqref{equ4-7} for $d_0>0$ small. Hence the claim \eqref{equ4-3} holds.

From \eqref{equ4-3}, by Proposition \ref{pro2-1}, we conclude that
\begin{equation}\label{equ4-8}
\lim_{\varepsilon\rightarrow0}\int_{A_{\varepsilon}^1}|u_{\varepsilon}|^{2_s^{\ast}}\,{\rm d}z=0,
\end{equation}
where $A_{\varepsilon}^1=B_{2\beta/\varepsilon}(\frac{x_{\varepsilon}}{\varepsilon})\backslash B_{\beta/\varepsilon}(\frac{x_{\varepsilon}}{\varepsilon})$. Indeed, taking a smooth cut-off function $\psi_{\varepsilon}\in C_0^{\infty}(\mathbb{R}^3)$ such that $\psi_{\varepsilon}=1$ on $B_{2\beta/\varepsilon}(\frac{x_{\varepsilon}}{\varepsilon})\backslash B_{\beta/\varepsilon}(\frac{x_{\varepsilon}}{\varepsilon})$, $\psi_{\varepsilon}=0$ on $A_{\varepsilon}^2=B_{3\beta/\varepsilon-1}(\frac{x_{\varepsilon}}{\varepsilon})\backslash B_{\beta/2\varepsilon+1}(\frac{x_{\varepsilon}}{\varepsilon})$.  Since $u_{\varepsilon}\in H_{\varepsilon}$ and using $(V_0)$, it is easy to check that $u_{\varepsilon}\psi_{\varepsilon}\in H^s(\mathbb{R}^3)$. Moreover,
\begin{equation*}
\sup_{y\in A_{\varepsilon}}\int_{B_1(y)}|u_{\varepsilon}|^{2_s^{\ast}}\,{\rm d}z\geq\sup_{y\in\mathbb{R}^3}\int_{B_1(y)}|u_{\varepsilon}\psi_{\varepsilon}|^{2_s^{\ast}}\,{\rm d}z.
\end{equation*}
By Proposition \ref{pro2-1}, we have
\begin{equation*}
\int_{\mathbb{R}^3}|u_{\varepsilon}\psi_{\varepsilon}|^{2_s^{\ast}}\,{\rm d}z\rightarrow0\quad \text{as}\,\, \varepsilon\rightarrow0.
\end{equation*}
Since $A_{\varepsilon}^1\subset A_{\varepsilon}^2$ for $\varepsilon>0$ small, so \eqref{equ4-8} holds. Therefore, by the interpolation inequality, it is not difficult to verify that
\begin{equation}\label{equ4-9}
\lim_{\varepsilon\rightarrow0}\int_{A_{\varepsilon}^1}|u_{\varepsilon}|^p\,{\rm d}z=0\quad \text{for all} \,\, p\in(2,2_s^{\ast}].
\end{equation}

Set $u_{\varepsilon,1}(z)=\varphi(\varepsilon z-x_{\varepsilon})u_{\varepsilon}(z)$, $u_{\varepsilon,2}(z)=(1-\varphi(\varepsilon z-x_{\varepsilon}))u_{\varepsilon}(z)$. As the same proof of (26) in \cite{Teng3}, we obtain that
\begin{align}\label{equ4-9-0}
\int_{\mathbb{R}^3}|D_s u_{\varepsilon}|^2\,{\rm d}z&=\int_{\mathbb{R}^3}|D_s u_{\varepsilon,1}|^2\,{\rm d}z+\int_{\mathbb{R}^3}|D_s u_{\varepsilon,2}|^2\,{\rm d}z\nonumber\\
&+2\int_{\mathbb{R}^3}\frac{(u_{\varepsilon,1}(x)-u_{\varepsilon,1}(y))(u_{\varepsilon,2}(x)-u_{\varepsilon,2}(y))}{|x-y|^{3+2s}}\,{\rm d}y\,{\rm d}z\nonumber\\
&\geq\int_{\mathbb{R}^3}|D_su_{\varepsilon,1}|^2\,{\rm d}z+\int_{\mathbb{R}^3}|D_su_{\varepsilon,2}|^2\,{\rm d}z+o(1).
\end{align}

By \eqref{equ4-9}, we deduce that
\begin{align*}
\int_{\mathbb{R}^3}V(\varepsilon z)|u_{\varepsilon}|^2\,{\rm d}z\geq\int_{\mathbb{R}^3}V(\varepsilon z)|u_{\varepsilon,1}|^2\,{\rm d}z+\int_{\mathbb{R}^3}V(\varepsilon z)|u_{\varepsilon,2}|^2\,{\rm d}z
\end{align*}
\begin{align*}
\int_{\mathbb{R}^3}\phi_{u_{\varepsilon}}^t|u_{\varepsilon}|^2\,{\rm d}z\geq\int_{\mathbb{R}^3}\phi_{u_{\varepsilon,1}}^t|u_{\varepsilon,1}|^2\,{\rm d}z+\int_{\mathbb{R}^3}\phi_{u_{\varepsilon,2}}^t|u_{\varepsilon,2}|^2\,{\rm d}z
\end{align*}
\begin{align*}
\int_{\mathbb{R}^3}G(\varepsilon z, u_{\varepsilon})\,{\rm d}z=\int_{\mathbb{R}^3}G(\varepsilon z, u_{\varepsilon,1})\,{\rm d}z+\int_{\mathbb{R}^3}G(\varepsilon z, u_{\varepsilon,2})\,{\rm d}z+o(1)\,\, \text{as}\,\,\varepsilon\rightarrow0
\end{align*}
\begin{align*}
\int_{\mathbb{R}^3}(u_{\varepsilon}^{+})^{2_s^{\ast}}\,{\rm d}z=\int_{\mathbb{R}^3}(u_{\varepsilon,1}^{+})^{2_s^{\ast}}\,{\rm d}z+\int_{\mathbb{R}^3}(u_{\varepsilon,2}^{+})^{2_s^{\ast}}\,{\rm d}z+o(1)\,\, \text{as}\,\,\varepsilon\rightarrow0
\end{align*}
and
\begin{equation*}
Q_{\varepsilon}(u_{\varepsilon,1})=0, \quad Q_{\varepsilon}(u_{\varepsilon,2})=Q_{\varepsilon}(u_{\varepsilon})\geq0.
\end{equation*}
Hence, we get
\begin{equation}\label{equ4-10}
\mathcal{J}_{\varepsilon}(u_{\varepsilon})\geq P_{\varepsilon}(u_{\varepsilon,1})+P_{\varepsilon}(u_{\varepsilon,2})+o(1),
\end{equation}
where $o(1)\rightarrow0$ as $\varepsilon\rightarrow0$.

We now estimate $P_{\varepsilon}(u_{\varepsilon,2})$. It follows from \eqref{equ4-2} that
\begin{align*}
\|u_{\varepsilon,2}\|_{H_{\varepsilon}}&\leq\|u_{\varepsilon,1}-\varphi_{\varepsilon}(\cdot-\frac{x_{\varepsilon}}{\varepsilon})W_0(\cdot-\frac{x_{\varepsilon}}{\varepsilon})\|_{H_{\varepsilon}}+2d_0\\
&=\|u_{\varepsilon,1}-\varphi_{\varepsilon}(\cdot-\frac{x_{\varepsilon}}{\varepsilon})W_0(\cdot-\frac{x_{\varepsilon}}{\varepsilon})\|_{H_{\varepsilon}(B_{2\beta/\varepsilon}(x_{\varepsilon}/\varepsilon))}\\
&+\|u_{\varepsilon,1}-\varphi_{\varepsilon}(\cdot-\frac{x_{\varepsilon}}{\varepsilon})W_0(\cdot-\frac{x_{\varepsilon}}{\varepsilon})\|_{H_{\varepsilon}(\mathbb{R}^3\backslash B_{2\beta/\varepsilon}(x_{\varepsilon}/\varepsilon))}+2d_0\\
&=\|u_{\varepsilon,1}-\varphi_{\varepsilon}(\cdot-\frac{x_{\varepsilon}}{\varepsilon})W_0(\cdot-\frac{x_{\varepsilon}}{\varepsilon})\|_{H_{\varepsilon}(B_{2\beta/\varepsilon}(x_{\varepsilon}/\varepsilon))}+2d_0+o(1)\\
&\leq\|u_{\varepsilon,2}\|_{H_{\varepsilon}(B_{2\beta/\varepsilon}(x_{\varepsilon}/\varepsilon))}+4d_0+o(1)\\
&=\|u_{\varepsilon,2}\|_{H_{\varepsilon}(B_{2\beta/\varepsilon}(x_{\varepsilon}/\varepsilon)\backslash B_{\beta/\varepsilon}(x_{\varepsilon}/\varepsilon))}+\|u_{\varepsilon,2}\|_{H_{\varepsilon}(B_{\beta/\varepsilon}(x_{\varepsilon}/\varepsilon))}+4d_0+o(1)\\
&=\|u_{\varepsilon,2}\|_{H_{\varepsilon}(B_{2\beta/\varepsilon}(x_{\varepsilon}/\varepsilon)\backslash B_{\beta/\varepsilon}(x_{\varepsilon}/\varepsilon))}+4d_0+o(1)\\
&\leq C\|u_{\varepsilon}\|_{H_{\varepsilon}(B_{2\beta/\varepsilon}(x_{\varepsilon}/\varepsilon)\backslash B_{\beta/\varepsilon}(x_{\varepsilon}/\varepsilon))}+4d_0+o(1)\\
&\leq C\|\varphi_{\varepsilon}(\cdot-\frac{x_{\varepsilon}}{\varepsilon})W_0(\cdot-\frac{x_{\varepsilon}}{\varepsilon})\|_{H_{\varepsilon}(B_{2\beta/\varepsilon}(x_{\varepsilon}/\varepsilon)\backslash B_{\beta/\varepsilon}(x_{\varepsilon}/\varepsilon))}+6d_0+o(1)\\
&\leq\|W_0(\cdot-\frac{x_{\varepsilon}}{\varepsilon})\|_{H_{\varepsilon}(B_{2\beta/\varepsilon}(x_{\varepsilon}/\varepsilon)\backslash B_{\beta/\varepsilon}(x_{\varepsilon}/\varepsilon))}+6d_0+o(1)\\
&=\|W_0\|_{H_{\varepsilon}(B_{2\beta/\varepsilon}(0)\backslash B_{\beta/\varepsilon}(0))}+6d_0+o(1)\\
&\leq6d_0+o(1),
\end{align*}
where $o(1)\rightarrow0$ as $\varepsilon\rightarrow0$ and using the similar arguments as \eqref{equ4-9-0}, we can prove that
\begin{equation*}
\|u_{\varepsilon,1}-\varphi_{\varepsilon}(\cdot-\frac{x_{\varepsilon}}{\varepsilon})W_0(\cdot-\frac{x_{\varepsilon}}{\varepsilon})\|_{H_{\varepsilon}(\mathbb{R}^3\backslash B_{2\beta/\varepsilon}(x_{\varepsilon}/\varepsilon))}=o(1)
\end{equation*}
and
\begin{equation*}
\|u_{\varepsilon,2}\|_{H_{\varepsilon}(B_{\beta/\varepsilon}(x_{\varepsilon}/\varepsilon))}=o(1).
\end{equation*}
Furthermore, the above inequality implies that
\begin{equation}\label{equ4-11}
\limsup\limits_{\varepsilon\rightarrow0}\|u_{\varepsilon,2}\|_{H_{\varepsilon}}\leq 6d_0.
\end{equation}
Then, we get
\begin{align}\label{equ4-12}
P_{\varepsilon}(u_{\varepsilon,2})&\geq\frac{1}{2}\|u_{\varepsilon,2}\|_{H_{\varepsilon}}^2-\int_{\mathbb{R}^3}F(u_{\varepsilon,2})\,{\rm d}z-\frac{1}{2_s^{\ast}}\int_{\mathbb{R}^3}(u_{\varepsilon,2}^{+})^{2_s^{\ast}}\,{\rm d}z\nonumber\\
&\geq\frac{1}{4}\|u_{\varepsilon,2}\|_{H_{\varepsilon}}^2-C\|u_{\varepsilon,2}\|_{H_{\varepsilon}}^{2_s^{\ast}}\nonumber\\
&=\|u_{\varepsilon,2}\|_{H_{\varepsilon}}^2(\frac{1}{4}-C\|u_{\varepsilon,2}\|_{H_{\varepsilon}}^{2_s^{\ast}-2})\nonumber\\
&\geq\|u_{\varepsilon,2}\|_{H_{\varepsilon}}^2(\frac{1}{4}-C(6d_0)^{2_s^{\ast}-2}).
\end{align}
In particular, taking $d_0>0$ small enough, we can assume that $P_{\varepsilon}(u_{\varepsilon,2})\geq0$. Hence, from \eqref{equ4-10}, it holds
\begin{equation}\label{equ4-13}
\mathcal{J}_{\varepsilon}(u_{\varepsilon})\geq P_{\varepsilon}(u_{\varepsilon,1})+o(1).
\end{equation}
Furthermore, by \eqref{equ4-9}, it is easy to check that
\begin{align*}
\int_{\mathbb{R}^3}\phi_{u_{\varepsilon}}^tu_{\varepsilon,1}u_{\varepsilon,2}\,{\rm d}z\leq\int_{\mathcal{T}_{\varepsilon}^1}\phi_{u_{\varepsilon}}^t|u_{\varepsilon}|^2\,{\rm d}z\leq\|\phi_{u_{\varepsilon}}^t\|_{2_t^{\ast}}\|u_{\varepsilon}\|_{L^{\frac{12}{3+2t}}(\mathcal{T}_{\varepsilon}^1)}^2\rightarrow0
\end{align*}
and
\begin{align*}
\int_{\mathbb{R}^3}\frac{(u_{\varepsilon,1}(z)-u_{\varepsilon,1}(y))(u_{\varepsilon,2}(z)-u_{\varepsilon,2}(y))}{|x-y|^{3+2s}}\,{\rm d}y\,{\rm d}z\geq o(1).
\end{align*}
Hence, using the facts that $\langle\mathcal{J}_{\varepsilon}'(u_{\varepsilon}),u_{\varepsilon,2}\rangle\rightarrow0$ as $\varepsilon\rightarrow0$, $\langle Q_{\varepsilon}'(u_{\varepsilon}),u_{\varepsilon,2}\rangle\geq0$, we have that
\begin{align*}
&\|u_{\varepsilon,2}\|_{H_{\varepsilon}}^2+o(1)\\
&\leq\|u_{\varepsilon,2}\|_{H_{\varepsilon}}^2+\int_{\mathbb{R}^3}\frac{(u_{\varepsilon,1}(z)-u_{\varepsilon,1}(y))(u_{\varepsilon,2}(z)-u_{\varepsilon,2}(y))}{|x-y|^{3+2s}}\,{\rm d}y\,{\rm d}z\\
&+\int_{\mathbb{R}^3}V(\varepsilon x)u_{\varepsilon,1}u_{\varepsilon,2}\,{\rm d}z+\int_{\mathbb{R}^3}\phi_{u_{\varepsilon}}^tu_{\varepsilon,1}u_{\varepsilon,2}\,{\rm d}z\\
&\leq\int_{\mathbb{R}^3}\frac{(u_{\varepsilon,1}(z)-u_{\varepsilon,1}(y))(u_{\varepsilon,2}(z)-u_{\varepsilon,2}(y))}{|x-y|^{3+2s}}\,{\rm d}y\,{\rm d}z+\int_{\mathbb{R}^3}V(\varepsilon x)u_{\varepsilon}u_{\varepsilon,2}\,{\rm d}z+\langle Q_{\varepsilon}'(u_{\varepsilon}),u_{\varepsilon,2}\rangle\\
&+\int_{\mathbb{R}^3}\phi_{u_{\varepsilon}}^tu_{\varepsilon}u_{\varepsilon,2}\,{\rm d}z+o(1)=\int_{\mathbb{R}^3}g(\varepsilon z, u_{\varepsilon})u_{\varepsilon,2}\,{\rm d}x+o(1)\\
&\leq \eta\int_{\mathbb{R}^3}|u_{\varepsilon}u_{\varepsilon,2}|\,{\rm d}z+C\int_{\mathbb{R}^3}|u_{\varepsilon}|^{2_s^{\ast}-1}|u_{\varepsilon,2}|\,{\rm d}z+o(1)\\
&\leq\eta\|u_{\varepsilon,2}\|_{L^2}^2+C\int_{\mathbb{R}^3}\Big(|u_{\varepsilon,2}|^{2_s^{\ast}}+|u_{\varepsilon,1}|^{2_s^{\ast}-1}|u_{\varepsilon,2}|\Big)\,{\rm d}x+o(1)\\
&\leq\eta\|u_{\varepsilon,2}\|_{H_{\varepsilon}}^2+C\|u_{\varepsilon,2}\|_{H_{\varepsilon}}^{2_s^{\ast}}+o(1)
\end{align*}
which yields to
\begin{equation*}
\|u_{\varepsilon,2}\|_{H_{\varepsilon}}^2\leq C\|u_{\varepsilon,2}\|_{H_{\varepsilon}}^{2_s^{\ast}}+o(1).
\end{equation*}
Combining with \eqref{equ4-11}, we get that for $d>0$ sufficiently small,
\begin{equation}\label{equ4-13-0}
\lim\limits_{\varepsilon\rightarrow0}\|u_{\varepsilon,2}\|_{H_{\varepsilon}}=0.
\end{equation}

We next estimate $P_{\varepsilon}(u_{\varepsilon,1})$. Denote $\widehat{u}_{\varepsilon}(z)=u_{\varepsilon,1}(z+\frac{x_{\varepsilon}}{\varepsilon})=\varphi(\varepsilon z)u_{\varepsilon}(z+\frac{x_{\varepsilon}}{\varepsilon})$, then $\{\widehat{u}_{\varepsilon}\}$ is bounded in $H^s(\mathbb{R}^3)$ by virtue of $(V_0)$. Thus, up to a subsequence, we may assume that there exists a $\widehat{u}\in H^s(\mathbb{R}^3)$ such that $\widehat{u}_{\varepsilon}\rightharpoonup\widehat{u}$ in $H^s(\mathbb{R}^3)$, $\widehat{u}_{\varepsilon}\rightarrow\widehat{u}$ in $L_{loc}^p(\mathbb{R}^3)$ for $1\leq p<2_s^{\ast}$ and $\widehat{u}_{\varepsilon}\rightarrow\widehat{u}$ a.e. in $\mathbb{R}^3$. We now claim that
\begin{equation}\label{equ4-14}
\widehat{u}_{\varepsilon}\rightarrow\widehat{u}\quad \text{in}\,\, L^{2_s^{\ast}}(\mathbb{R}^3).
\end{equation}
In view of Proposition \ref{pro2-1}, suppose the contrary that there exists $r>0$ such that
\begin{equation*}
\liminf_{\varepsilon\rightarrow0}\sup_{y\in\mathbb{R}^3}\int_{B_1(y)}|\widehat{u}_{\varepsilon}-\widehat{u}|^{2_s^{\ast}}\,{\rm d}z=2r>0.
\end{equation*}
Thus, for $\varepsilon>0$ small, there exists $\widehat{y}_{\varepsilon}\in\mathbb{R}^3$ such that
\begin{equation}\label{equ4-15}
\int_{B_1(\widehat{y}_{\varepsilon})}|\widehat{u}_{\varepsilon}-\widehat{u}|^{2_s^{\ast}}\,{\rm d}z\geq r>0.
\end{equation}

$\bullet$ $\{\widehat{y}_{\varepsilon}\}$ is bounded in $\mathbb{R}^3$, then there exists $r_0>0$ such that $|\widehat{y}_{\varepsilon}|\leq r_0$. Let $\widehat{v}_{\varepsilon}=\widehat{u}_{\varepsilon}-\widehat{u}$, then $\widehat{v}_{\varepsilon}\rightharpoonup0$ in $H^s(\mathbb{R}^3)$, for $\varepsilon>0$ small, by \eqref{equ4-15}, it holds
\begin{equation}\label{equ4-16}
\int_{B_{r_0+1}(0)}|\widehat{v}_{\varepsilon}|^{2_s^{\ast}}\,{\rm d}z\geq r>0.
\end{equation}
We now are to prove that
\begin{equation}\label{equ4-17}
\lim_{\varepsilon\rightarrow0}\sup_{\widehat{\varphi}\in C_0^{\infty}(B_{r_0+2}(0)),\|\widehat{\varphi}\|=1}|\langle \widehat{\rho}_{\varepsilon},\widehat{\varphi}\rangle|=0,
\end{equation}
where $\widehat{\rho}_{\varepsilon}=-(-\Delta)^s\widehat{v}_{\varepsilon}+(\widehat{v}_{\varepsilon}^{+})^{2_s^{\ast}-1}\in (H^s(\mathbb{R}^3))'$. For $\varepsilon>0$ small, $\int_{\mathbb{R}^3\backslash\Lambda_{\varepsilon}}u_{\varepsilon}\widehat{\varphi}(z-\frac{x_{\varepsilon}}{\varepsilon})\,{\rm d}z=0$ uniformly for all $\widehat{\varphi}\in C_0^{\infty}(B_{r_0+2}(0))$. Hence, by virtue of $\lim\limits_{\varepsilon\rightarrow0}\|u_{\varepsilon,2}\|_{H_{\varepsilon}}=0$, we have
\begin{align*}
o(1)&=\langle\mathcal{J}_{\varepsilon}'(u_{\varepsilon}),\widehat{\varphi}(\cdot-\frac{x_{\varepsilon}}{\varepsilon})\rangle=\int_{\mathbb{R}^3}\frac{(u_{\varepsilon}(z+\frac{x_{\varepsilon}}{\varepsilon})-u_{\varepsilon}(y+\frac{x_{\varepsilon}}{\varepsilon}))(\widehat{\varphi}(z)-\widehat{\varphi}(y))}{|x-y|^{3+2s}}\,{\rm d}y\,{\rm d}z\\
&+\int_{\mathbb{R}^3}V(\varepsilon z+x_{\varepsilon})u_{\varepsilon}(z+\frac{x_{\varepsilon}}{\varepsilon})\widehat{\varphi}\,{\rm d}z+\int_{\mathbb{R}^3}\phi_{u_{\varepsilon}(z+\frac{x_{\varepsilon}}{\varepsilon})}^tu_{\varepsilon}(z+\frac{x_{\varepsilon}}{\varepsilon})\widehat{\varphi}\,{\rm d}z\\
&-\int_{\mathbb{R}^3}g(\varepsilon z+x_{\varepsilon},u_{\varepsilon}(z+\frac{x_{\varepsilon}}{\varepsilon}))\widehat{\varphi}\,{\rm d}z\\
&=\int_{\mathbb{R}^3}\frac{(\widehat{u}_{\varepsilon}(z)-\widehat{u}_{\varepsilon}(y))(\widehat{\varphi}(z)-\widehat{\varphi}(y))}{|x-y|^{3+2s}}\,{\rm d}y\,{\rm d}z+\int_{\mathbb{R}^3}V(\varepsilon z+x_{\varepsilon})\widehat{u}_{\varepsilon}\widehat{\varphi}\,{\rm d}z\\
&+\int_{\mathbb{R}^3}\phi_{\widehat{u}_{\varepsilon}}^t\widehat{u}_{\varepsilon}\widehat{\varphi}\,{\rm d}z-\int_{\mathbb{R}^3}g(\varepsilon z+x_{\varepsilon},\widehat{u}_{\varepsilon})\widehat{\varphi}\,{\rm d}z+o(1).
\end{align*}
Combining the above estimate with $x_{\varepsilon}\rightarrow x_0\in\mathcal{M}^{\beta}$ , we see that $\widehat{u}\geq0$ is a solution of
\begin{equation}\label{equ4-18}
(-\Delta)^s \widehat{u}+V(x_0) \widehat{u}+\phi_{\widehat{u}}^t \widehat{u}=f(\widehat{u})+(\widehat{u}^{+})^{2_s^{\ast}-1} \quad  \text{in}\,\,\mathbb{R}^3.
\end{equation}
On the other hand, the following Brezis-Lieb splitting properties hold:
\begin{align*}
\int_{\mathbb{R}^3}\frac{(\widehat{u}_{\varepsilon}(z)-\widehat{u}_{\varepsilon}(y))-(\widehat{v}_{\varepsilon}(z)-\widehat{v}_{\varepsilon}(y))-(\widehat{u}(z)-\widehat{u}(y))}{|x-y|^{3+2s}}(\widehat{\varphi}(z)-\widehat{\varphi}(y))\,{\rm d}y\,{\rm d}z=o(1),
\end{align*}
\begin{align*}
\int_{\mathbb{R}^3}\Big(V(\varepsilon z+x_{\varepsilon})\widehat{u}_{\varepsilon}-V(\varepsilon z+x_{\varepsilon})\widehat{v}_{\varepsilon}-V(x_0)\widehat{u}\Big)\widehat{\varphi}\,{\rm d}z=o(1),
\end{align*}
\begin{align*}
\int_{\mathbb{R}^3}\Big(\phi_{\widehat{u}_{\varepsilon}}^t\widehat{u}_{\varepsilon}-\phi_{\widehat{v}_{\varepsilon}}^t\widehat{v}_{\varepsilon}-\phi_{\widehat{u}}^t\widehat{u}\Big)\widehat{\varphi}\,{\rm d}z=o(1),\quad \int_{\mathbb{R}^3}\Big(f(\widehat{u}_{\varepsilon})-f(\widehat{v}_{\varepsilon})-f(\widehat{u})\Big)\widehat{\varphi}\,{\rm d}z=o(1)
\end{align*}
and
\begin{align*}
\int_{\mathbb{R}^3}\Big(\widehat{u}_{\varepsilon}^{+})^{2_s^{\ast}-1}-(\widehat{v}_{\varepsilon}^{+})^{2_s^{\ast}-1}-(\widehat{u}^{+})^{2_s^{\ast}-1}\Big)\widehat{\varphi}\,{\rm d}z=o(1)
\end{align*}
uniformly for all $\widehat{\varphi}\in C_0^{\infty}(B_{r_0+2}(0))$ with $\|\widehat{\varphi}\|=1$. Thus, \eqref{equ4-17} is proved.

By Proposition \ref{pro2-3}, there exist $\widehat{z}_{\varepsilon}\in\mathbb{R}^3$ and $\widehat{\sigma}_{\varepsilon}>0$ such that $\widehat{z}_{\varepsilon}\rightarrow \widehat{z}\in B_{r_0+1}(0)$, $\widehat{\sigma}_{\varepsilon}\rightarrow0$ and
\begin{equation*}
\overline{w}_{\varepsilon}(z)=\widehat{\sigma}_{\varepsilon}^{\frac{3-2s}{2}}\widehat{v}_{\varepsilon}(\widehat{\sigma}_{\varepsilon} z+\widehat{z}_{\varepsilon})\rightharpoonup\overline{w} \quad \text{in}\,\, \mathcal{D}^{s,2}(\mathbb{R}^3),
\end{equation*}
where $\overline{w}\geq0$ is a nontrivial solution of \eqref{equ4-6} and satisfies \eqref{equ4-6-0}.

Since
\begin{align*}
\int_{\mathbb{R}^3}|\overline{w}|^{2_s^{\ast}}\,{\rm d}z&\leq\liminf_{\varepsilon\rightarrow0}|\overline{w}_{\varepsilon}|^{2_s^{\ast}}\,{\rm d}z=\liminf_{\varepsilon\rightarrow0}\int_{\mathbb{R}^3}|\widehat{v}_{\varepsilon}|^{2_s^{\ast}}\,{\rm d}z\\
&=\liminf_{\varepsilon\rightarrow0}\int_{\mathbb{R}^3}|\widehat{u}_{\varepsilon}|^{2_s^{\ast}}\,{\rm d}z-\int_{\mathbb{R}^3}|\widehat{u}|^{2_s^{\ast}}\,{\rm d}z\\
&\leq\liminf_{\varepsilon\rightarrow0}\int_{\mathbb{R}^3}|u_{\varepsilon}|^{2_s^{\ast}}\,{\rm d}z,
\end{align*}
but by \eqref{equ4-2}, we get
\begin{align*}
\int_{\mathbb{R}^3}|u_{\varepsilon}|^{2_s^{\ast}}\,{\rm d}z&\leq Cd_0+\int_{\mathbb{R}^3}|\varphi(\varepsilon z-x_{\varepsilon})W_0(z-\frac{x_{\varepsilon}}{\varepsilon})|^{2_s^{\ast}}\,{\rm d}z\\
&=Cd_0+\int_{B_{2\beta/\varepsilon}(\frac{x_{\varepsilon}}{\varepsilon})}|\varphi W_0|^{2_s^{\ast}}\,{\rm d}z.
\end{align*}
Thus,
\begin{equation}\label{equ4-20}
\int_{\mathbb{R}^3}|\overline{w}|^{2_s^{\ast}}\,{\rm d}z\leq Cd_0+\int_{\mathbb{R}^3}|W_0|^{2_s^{\ast}}\,{\rm d}z.
\end{equation}

On the other hand, by \eqref{equ4-2}, we have
\begin{align*}
\int_{\mathbb{R}^3}|D_su_{\varepsilon}|^2\,{\rm d}z&\leq cd_0+\int_{\mathbb{R}^3}\Big|D_s\Big(\varphi(\varepsilon z-x_{\varepsilon}) W_0(z-\frac{x_{\varepsilon}}{\varepsilon})\Big)\Big|^2\,{\rm d}z\\
&\leq cd_0+\int_{\mathbb{R}^3}|D_sW_0|^2\,{\rm d}z+o(1).
\end{align*}

\begin{align*}
\int_{\mathbb{R}^3}|D_s\overline{w}|^2\,{\rm d}z&\leq\liminf_{\varepsilon\rightarrow0}\int_{\mathbb{R}^3}|D_s\overline{w}_{\varepsilon}|^2\,{\rm d}z=\liminf_{\varepsilon\rightarrow0}\int_{\mathbb{R}^3}|D_s\widehat{v}_{\varepsilon}|^2\,{\rm d}z\\
&=\liminf_{\varepsilon\rightarrow0}\int_{\mathbb{R}^3}|D_s\widehat{u}_{\varepsilon}|^2\,{\rm d}z-\int_{\mathbb{R}^3}|D_s\widehat{u}|^2\,{\rm d}z\leq\liminf_{\varepsilon\rightarrow0}\int_{\mathbb{R}^3}|D_s\widehat{u}_{\varepsilon}|^2\,{\rm d}z,
\end{align*}
hence, we get
\begin{equation}\label{equ4-21}
\int_{\mathbb{R}^3}|D_s\overline{w}|^2\,{\rm d}z\leq Cd_0+\int_{\mathbb{R}^3}|D_sW_0|^2\,{\rm d}z+o(1).
\end{equation}

From the $W_0\in\mathcal{L}_{V_0}$, and by \eqref{equ4-20}, \eqref{equ4-21}, we have that
\begin{align*}
c_{V_0}&=\mathcal{I}_{V_0}(W_0)-\frac{1}{q(s+t)-3}\mathcal{G}_{\mu}(W_0)>\frac{(q-4)s+(q-2)t}{2(q(s+t)-3)}\int_{\mathbb{R}^3}|D_s W_0|^2\,{\rm d}z\\
&+\frac{(2_s^{\ast}-q)(s+t)}{2_s^{\ast}(q(s+t)-3)}\int_{\mathbb{R}^3}|W_0|^{2_s^{\ast}}\,{\rm d}z\\
&>\frac{(q-4)s+(q-2)t}{2(q(s+t)-3)}\int_{\mathbb{R}^3}|D_s\overline{w}|^2\,{\rm d}z+\frac{(2_s^{\ast}-q)(s+t)}{2_s^{\ast}(q(s+t)-3)}\int_{\mathbb{R}^3}|\overline{w}|^{2_s^{\ast}}\,{\rm d}z-Cd_0\\
&>\frac{s}{3}\mathcal{S}_s^{\frac{3}{2s}}-Cd_0,
\end{align*}
where we have used the fact that $\frac{(q-4)s+(q-2)t}{2(q(s+t)-3)}+\frac{(2_s^{\ast}-q)(s+t)}{2_s^{\ast}(q(s+t)-3)}=\frac{s}{3}$. For $d_0$ sufficient small, we get a contradiction with Lemma \ref{lem3-4}.

$\bullet$ $\{\widehat{z}_{\varepsilon}\}$ is unbounded. Without loss of generality, we may assume that $\lim\limits_{\varepsilon\rightarrow0}|\widehat{z}_{\varepsilon}|=+\infty$. Then, by \eqref{equ4-15}, we have that
\begin{equation}\label{equ4-22}
\liminf_{\varepsilon\rightarrow0}\int_{B_1(\widehat{z}_{\varepsilon})}|\widehat{u}_{\varepsilon}|^{2_s^{\ast}}\,{\rm d}z\geq r>0,
\end{equation}
i.e.,
\begin{equation*}
\liminf_{\varepsilon\rightarrow0}\int_{B_1(\widehat{z}_{\varepsilon})}|\varphi(\varepsilon z)u_{\varepsilon}(z+\frac{x_{\varepsilon}}{\varepsilon})|^{2_s^{\ast}}\,{\rm d}z\geq r>0.
\end{equation*}
Since $\varphi(z)=0$ for $|z|\geq2\beta$, so $|\widehat{z}_{\varepsilon}|\leq\frac{3\beta}{\varepsilon}$ for $\varepsilon$ small. If $|\widehat{z}_{\varepsilon}|\geq\frac{\beta}{2\varepsilon}$, then $\widehat{z}_{\varepsilon}\in B_{3\beta/\varepsilon}(0)\backslash B_{\beta/2\varepsilon}(0)$, and by \eqref{equ4-3}, we get
\begin{align*}
\liminf_{\varepsilon\rightarrow0}\int_{B_1(\widehat{z}_{\varepsilon})}|\widehat{u}_{\varepsilon}|^{2_s^{\ast}}\,{\rm d}z&\leq\liminf_{\varepsilon\rightarrow0}\sup_{y\in B_{3\beta/\varepsilon}(0)\backslash B_{\beta/2\varepsilon}(0)}\int_{B_1(y)}|u_{\varepsilon}(z+\frac{x_{\varepsilon}}{\varepsilon})|^{2_s^{\ast}}\,{\rm d}z\\
&\leq\liminf_{\varepsilon\rightarrow0}\sup_{y\in\mathcal{T}_{\varepsilon}}\int_{B_1(y)}|\widehat{u}_{\varepsilon}|^{2_s^{\ast}}\,{\rm d}z=0
\end{align*}
which contradicts with \eqref{equ4-22}. Thus $|\widehat{z}_{\varepsilon}|\leq\frac{\beta}{2\varepsilon}$ for $\varepsilon>0$ small. Without loss of generality, we may assume that $\varepsilon\widehat{z}_{\varepsilon}\rightarrow z_0\in \overline{B_{\beta/2}(0)}$ and $\widetilde{u}_{\varepsilon}\rightharpoonup\widetilde{u}$ in $H^s(\mathbb{R}^3)$, where $\widetilde{u}_{\varepsilon}(z):=\widehat{u}_{\varepsilon}(z+\widehat{z}_{\varepsilon})$. If $\widetilde{u}\neq0$, it is easy to check that $\widetilde{u}$ satisfies that
\begin{equation*}
(-\Delta)^s v+V(x_0+z_0) v+\phi_{v}^tv=f(v)+v^{2_s^{\ast}-1}\quad \text{in}\,\,\mathbb{R}^3.
\end{equation*}
Similarly as in the proof of the case $v\neq0$ of the claim \eqref{equ4-3}, we can get a contradiction for $d_0$ sufficient small. Thus $\widetilde{u}=0$. Similarly as the proof of the case $u=0$ of the claim \eqref{equ4-3} (where using Proposition \eqref{pro2-3}), we find that there exist $\widetilde{x}_{\varepsilon}\in\mathbb{R}^3$ and $\widetilde{\sigma}_{\varepsilon}>0$ such that $\widetilde{x}_{\varepsilon}\rightarrow \widetilde{x}\in \overline{B_1(0)}$, $\widetilde{\sigma}_{\varepsilon}\rightarrow0$ and
\begin{equation*}
u_{\varepsilon}^{\ast}(\cdot):=\widetilde{\sigma}_{\varepsilon}^{\frac{3-2s}{2}}\widetilde{u}_{\varepsilon}(\widetilde{\sigma}_{\varepsilon} \cdot+\widetilde{x}_{\varepsilon})\rightharpoonup u^{\ast}\quad \text{in}\,\,\mathcal{D}^{s,2}(\mathbb{R}^3),
\end{equation*}
where $u^{\ast}$ is a nontrivial of solution of \eqref{equ4-5} and satisfies \eqref{equ4-6}.
Thus, there exists $R>0$ such that
\begin{equation*}
\int_{B_R(0)}|u^{\ast}|^{2_s^{\ast}}\,{\rm d}z\geq\frac{1}{2}\int_{\mathbb{R}^3}|u^{\ast}|^{2_s^{\ast}}\,{\rm d}z=\frac{1}{2}\mathcal{S}_s^{\frac{3}{2s}}>0.
\end{equation*}
On the other hand, we have that
\begin{align*}
&\int_{B_R(0)}|u^{\ast}|^{2_s^{\ast}}\,{\rm d}z\leq\liminf_{\varepsilon\rightarrow0}\int_{B_R(0)}|u_{\varepsilon}^{\ast}|^{2_s^{\ast}}\,{\rm d}z=\liminf_{\varepsilon\rightarrow0}\int_{B_{\widetilde{\sigma}_{\varepsilon}R}(\widetilde{x}_{\varepsilon})}|\widetilde{u}_{\varepsilon}|^{2_s^{\ast}}\,{\rm d}z\\
&=\liminf_{\varepsilon\rightarrow0}\int_{B_{\widetilde{\sigma}_{\varepsilon}R}(\widetilde{x}_{\varepsilon}+\widehat{z}_{\varepsilon}+\frac{x_{\varepsilon}}{\varepsilon})}|u_{\varepsilon}|^{2_s^{\ast}}\,{\rm d}z\leq\liminf_{\varepsilon\rightarrow0}\int_{B_2(\widehat{z}_{\varepsilon}+\frac{x_{\varepsilon}}{\varepsilon})}|u_{\varepsilon}|^{2_s^{\ast}}\,{\rm d}z.
\end{align*}
which contradicts to \eqref{equ4-2} for $d_0$ small enough. Hence, the claim \eqref{equ4-14} holds and so using the interpolation inequality, we deduce that
\begin{equation}\label{equ4-23}
\widehat{u}_{\varepsilon}\rightarrow \widehat{u}\quad \text{in}\,\, L^p(\mathbb{R}^3), \,\, p\in(2,2_s^{\ast}].
\end{equation}

By \eqref{equ4-13}, recalling that $\widehat{u}_{\varepsilon}(z)=u_{\varepsilon,1}(z+\frac{x_{\varepsilon}}{\varepsilon})$, we have
\begin{align*}
P_{\varepsilon}(\widehat{u}_{\varepsilon})\leq c_{V_0}+o(1).
\end{align*}
Letting $\varepsilon\rightarrow0$, and using \eqref{equ4-23}, $(V_0)$, we get
\begin{align*}
\mathcal{I}_{V(x_0)}(\widehat{u})\leq c_{V_0}.
\end{align*}
On the other hand, in view of $\langle\mathcal{J}'_{\varepsilon}(u_{\varepsilon}),u_{\varepsilon,1}\rangle\rightarrow0$ and \eqref{equ4-13-0}, and $\langle Q_{\varepsilon}'(u_{\varepsilon}),u_{\varepsilon,1}\rangle=0$, we deduce that
\begin{align*}
\int_{\mathbb{R}^3}|D_s\widehat{u}_{\varepsilon}|^2\,{\rm d}z&+\int_{\mathbb{R}^3}V(\varepsilon z+x_{\varepsilon})|\widehat{u}_{\varepsilon}|^2\,{\rm d}z+\int_{\mathbb{R}^3}\phi_{\widehat{u}_{\varepsilon}}^t|\widehat{u}_{\varepsilon}|^2\,{\rm d}z\\
&=\int_{\mathbb{R}^3}g(\varepsilon z+x_{\varepsilon},\widehat{u}_{\varepsilon})\widehat{u}_{\varepsilon}\,{\rm d}z+o(1),
\end{align*}
then by Fatou's Lemma, \eqref{equ4-23} and \eqref{equ4-18}, we have that
\begin{align*}
&\int_{\mathbb{R}^3}|D_s\widehat{u}|^2\,{\rm d}z+\int_{\mathbb{R}^3}V(x_0)|\widehat{u}|^2\,{\rm d}z+\int_{\mathbb{R}^3}\phi_{\widehat{u}}^t|\widehat{u}|^2\,{\rm d}z\\
&\leq\liminf_{\varepsilon\rightarrow0}\Big(\int_{\mathbb{R}^3}|D_s\widehat{u}_{\varepsilon}|^2\,{\rm d}z+\int_{\mathbb{R}^3}V(\varepsilon z+x_{\varepsilon})|\widehat{u}_{\varepsilon}|^2\,{\rm d}z+\int_{\mathbb{R}^3}\phi_{\widehat{u}_{\varepsilon}}^t|\widehat{u}_{\varepsilon}|^2\,{\rm d}z\Big)\\
&=\liminf_{\varepsilon\rightarrow0}\Big(\int_{\mathbb{R}^3}g(\varepsilon z+x_{\varepsilon},\widehat{u}_{\varepsilon})\widehat{u}_{\varepsilon}\,{\rm d}z\Big)=\int_{\mathbb{R}^3}f(\widehat{u})\widehat{u}\,{\rm d}z+\int_{\mathbb{R}^3}(\widehat{u}^{+})^{2_s^{\ast}}\,{\rm d}z\\
&=\int_{\mathbb{R}^3}|D_s\widehat{u}|^2\,{\rm d}z+\int_{\mathbb{R}^3}V(x_0)|\widehat{u}|^2\,{\rm d}z+\int_{\mathbb{R}^3}\phi_{\widehat{u}}^t|\widehat{u}|^2\,{\rm d}z,
\end{align*}
which implies that
\begin{equation*}
\int_{\mathbb{R}^3}|D_s\widehat{u}_{\varepsilon}|^2\,{\rm d}z\rightarrow\int_{\mathbb{R}^3}|D_s\widehat{u}|^2\,{\rm d}z,
\end{equation*}
and
\begin{equation*}
\int_{\mathbb{R}^3}V(\varepsilon z+x_{\varepsilon})|\widehat{u}_{\varepsilon}|^2\,{\rm d}z\rightarrow\int_{\mathbb{R}^3}V(x_0)|\widehat{u}|^2\,{\rm d}z.
\end{equation*}
Hence, by $(V_0)$, we can deduce that
\begin{equation}\label{equ4-24}
\widehat{u}_{\varepsilon}\rightarrow\widehat{u}\quad  \text{in}\,\, H^s(\mathbb{R}^3).
\end{equation}
By \eqref{equ4-2}, \eqref{equ4-23}, it is easy to check that $\widehat{u}\neq0$. It follows from \eqref{equ4-18} that $\mathcal{I}_{V(x_0)}(\widehat{u})\geq c_{V(x_0)}$. Hence, $\mathcal{I}_{V(x_0)}(\widehat{u})= c_{V(x_0)}$ is proved. In view of $x_0\in\mathcal{M}^{\beta}\subset\Lambda$, we have that $V(x_0)=V_0$ and $x_0\in\mathcal{M}$. As a consequence, $\widehat{u}$ is, up to a translation in the $x-$variable, an element of $\mathcal{L}_{V_0}$, namely there exists $W\in\mathcal{L}_{V_0}$ and $z_0\in\mathbb{R}^3$ such that $\widehat{u}(z)=W(z-z_0)$. Consequently, from \eqref{equ4-2}, \eqref{equ4-13-0} and \eqref{equ4-24}, we have that
\begin{equation*}
\|u_{\varepsilon}-\varphi_{\varepsilon}(\cdot-\frac{x_{\varepsilon}}{\varepsilon}-z_0)W(\cdot-\frac{x_{\varepsilon}}{\varepsilon}-z_0)\|_{H_{\varepsilon}}\rightarrow0\quad \text{as}\,\, \varepsilon\rightarrow0.
\end{equation*}
Observing that $\varepsilon(\frac{x_{\varepsilon}}{\varepsilon}+z_0)\rightarrow x_0\in\mathcal{M}$ as $\varepsilon\rightarrow0$, so the proof is completed.
\end{proof}

For $a\in\mathbb{R}$ we define the sublevel set of $\mathcal{J}_{\varepsilon}$ as follows
\begin{equation*}
\mathcal{J}_{\varepsilon}^a=\{u\in H_{\varepsilon}\,\,\Big|\,\, \mathcal{J}_{\varepsilon}(u)\leq a\}.
\end{equation*}

We observe that the result of Lemma \ref{lem4-1} holds for $d_0>0$ sufficiently small independently of the sequences satisfying the assumptions.
\begin{lemma}\label{lem4-2}
Let $d_0$ be the number given in Lemma \ref{lem4-1}. Then for any $d\in(0,d_0)$, there exist positive constants $\varepsilon_d>0$, $\rho_d>0$ and $\alpha_d>0$ such that
\begin{equation*}
\|\mathcal{J}'_{\varepsilon}(u)\|_{(H_{\varepsilon})'}\geq\alpha_d>0\quad \text{for every}\,\, u\in\mathcal{J}_{\varepsilon}^{c_{V_0}+\rho_d}\cap(\mathcal{N}_{\varepsilon}^{d_0}\backslash\mathcal{N}_{\varepsilon}^d)\,\,\text{and}\,\,\varepsilon\in(0,\varepsilon_d).
\end{equation*}
\end{lemma}

We recall the definition \eqref{equ4-1} of $\gamma_{\varepsilon}(\tau)$. The following Lemma holds.

\begin{lemma}\label{lem4-3}
There exists $M_0>0$ such that for any $\delta>0$ small, there exists $\alpha_{\delta}>0$ and $\varepsilon_{\delta}>0$ such that if $\mathcal{J}_{\varepsilon}(\gamma_{\varepsilon}(\tau))\geq c_{V_0}-\alpha_{\delta}$ and $\varepsilon\in(0,\varepsilon_{\delta})$, then $\gamma_{\varepsilon}(\tau)\in\mathcal{N}_{\varepsilon}^{M_0\delta}$.
\end{lemma}

We are now ready to show that the penalized functional $\mathcal{J}_{\varepsilon}$ possesses a critical point for every $\varepsilon>0$ sufficiently small. Choose $\delta_1>0$ such that $M_0\delta_1<\frac{d_0}{4}$ in Lemma \ref{lem4-3}, and fixing $d=\frac{d_0}{4}:=d_1$ in Lemma \ref{lem4-2}. Similar to the proof of Lemma 4.6 in \cite{HL}, we can prove the following result.
\begin{lemma}\label{lem4-4}
There exists $\overline{\varepsilon}>0$ such that for each $\varepsilon\in(0,\overline{\varepsilon})$, there exists a sequence $\{u_{\varepsilon,n}\}\subset \mathcal{J}_{\varepsilon}^{\widetilde{\mathcal{C}}_{\varepsilon}+\varepsilon}\cap\mathcal{N}_{\varepsilon}^{d_0}$ such that $\mathcal{J}'_{\varepsilon}(u_{\varepsilon,n})\rightarrow0$ in $(H_{\varepsilon})'$ as $n\rightarrow\infty$.
\end{lemma}
\begin{lemma}\label{lem4-5}
$\mathcal{J}_{\varepsilon}$ possesses a nontrivial critical point $u_{\varepsilon}\in\mathcal{N}_{\varepsilon}^{d_0}\cap\mathcal{J}_{\varepsilon}^{\mathcal{D}_{\varepsilon}+\varepsilon}$ for $\varepsilon\in(0,\bar{\varepsilon}]$.
\end{lemma}
\begin{proof}
By Lemma \ref{lem4-4}, there exists $\bar{\varepsilon}>0$ such that for each $\varepsilon\in(0,\bar{\varepsilon}]$, there exists a sequence $\{u_{\varepsilon,n}\}\subset \mathcal{J}_{\varepsilon}^{\mathcal{D}_{\varepsilon}+\varepsilon}\cap\mathcal{N}_{\varepsilon}^{d_0}$ such that $\mathcal{J}_{\varepsilon_n}'(u_{\varepsilon,n})\rightarrow0$ as $n\rightarrow\infty$ in $(H_{\varepsilon})'$. Since $\mathcal{N}_{\varepsilon}^{d_0}$ is bounded, then $\{u_{\varepsilon,n}\}$ is bounded in $H_{\varepsilon}$ and up to a subsequence, we may assume that there exists $u_{\varepsilon}\in H_{\varepsilon}$ such that $u_{\varepsilon,n}\rightharpoonup u_{\varepsilon}$ in $H_{\varepsilon}$, $u_{\varepsilon,n}\rightarrow u_{\varepsilon}$ in $L_{loc}^p(\mathbb{R}^3)$ for $1\leq p<2_s^{\ast}$ and $u_{\varepsilon,n}\rightarrow u_{\varepsilon}$ a.e. in $\mathbb{R}^3$. It is easy to check that $u_{\varepsilon}$ satisfies
\begin{equation}\label{equ4-26}
(-\Delta)^su_{\varepsilon}+V(\varepsilon z)u_{\varepsilon}+\phi_{u_{\varepsilon}}^tu_{\varepsilon}=-4\mu_{\varepsilon}\chi_{\mathbb{R}^3\backslash\Lambda_{\varepsilon}}u_{\varepsilon}+g(\varepsilon z,u_{\varepsilon})\quad \text{in}\,\,\mathbb{R}^3,
\end{equation}
where $\mu_{\varepsilon,n}=\Big(\int_{\mathbb{R}^3\backslash\Lambda_{\varepsilon}}u_{\varepsilon,n}^2\,{\rm d}z-\varepsilon\Big)_{+}\rightarrow\mu_{\varepsilon}$ as $n\rightarrow\infty$.

We claim that
\begin{equation}\label{equ4-27}
\lim_{R\rightarrow\infty}\sup_{n\geq1}\int_{|x|\geq R}(|D_su_{\varepsilon,n}|^2+V(\varepsilon z)|u_{\varepsilon,n}|^2)\,{\rm d}z=0.
\end{equation}

Indeed, Choosing a cutoff function $\psi_{\rho}\in C^{\infty}(\mathbb{R}^3)$ such that $\psi_{\rho}(z)=1$ on $\mathbb{R}^3\backslash B_{2\rho}(0)$, $\psi_{\rho}(z)=0$ on $B_{\rho}(0)$, $0\leq\psi_{\rho}\leq1$ and $|\nabla \psi_{\rho}|\leq\frac{C}{\rho}$. Since $\psi_{\rho}u_{\varepsilon,n}\in H_{\varepsilon}$, then $\langle\mathcal{J}_{\varepsilon_n}'(u_{\varepsilon,n}),\psi_{\rho}u_{\varepsilon,n}\rangle\rightarrow0$ as $n\rightarrow\infty$. Thus, for sufficiently large $\rho$ such that $\Lambda_{\varepsilon}\subset B_{\rho}(0)$, we have
\begin{align*}
&\int_{\mathbb{R}^3}(|D_su_{\varepsilon,n}|^2+V(\varepsilon z)|u_{\varepsilon,n}|^2)\psi_{\rho}\,{\rm d}z+\int_{\mathbb{R}^3}\int_{\mathbb{R}^3}\frac{(u_{\varepsilon,n}(z)-u_{\varepsilon,n}(y))(\psi_{\rho}(z)-\psi_{\rho}(y))u_{\varepsilon, n}(y)}{|z-y|^{3+2s}}\,{\rm d}y\,{\rm d}z\\
&=\int_{\mathbb{R}^3}g(\varepsilon z,u_{\varepsilon,n})u_{\varepsilon,n}\psi_{\rho}\,{\rm d}z-\int_{\mathbb{R}^3}\phi_{u_{\varepsilon,n}}^t|u_{\varepsilon,n}|^2\psi_{\rho}\,{\rm d}z-4\Big(\int_{\mathbb{R}^3\backslash\Lambda_{\varepsilon}}|u_{\varepsilon,n}|^2\,{\rm d}z-\varepsilon\Big)_{+}\int_{\mathbb{R}^3\backslash\Lambda_{\varepsilon}}|u_{\varepsilon,n}|^2\psi_{\rho}\,{\rm d}z\\
&\leq\int_{\mathbb{R}^3}g(\varepsilon z,u_{\varepsilon,n})u_{\varepsilon,n}\psi_{\rho}\,{\rm d}z\leq\frac{V_0}{k}\int_{\mathbb{R}^3}|u_{\varepsilon,n}|^2\psi_{\rho}\,{\rm d}z.
\end{align*}
In view of the fact that $|D_s\psi_{\rho}|^2\leq\frac{C}{\rho^{2s}}$ for any $z\in\mathbb{R}^3$ and H\"{o}lder's inequality, we deduce that
\begin{align*}
&\int_{\mathbb{R}^3}\int_{\mathbb{R}^3}\frac{(u_{\varepsilon,n}(z)-u_{\varepsilon,n}(y))(\psi_{\rho}(z)-\psi_{\rho}(y))u_{\varepsilon, n}(y)}{|z-y|^{3+2s}}\,{\rm d}y\,{\rm d}z\\
&\leq\Big(\int_{\mathbb{R}^3}|D_su_{\varepsilon,n}|^2\,{\rm d}z\Big)^{\frac{1}{2}}\Big(\int_{\mathbb{R}^3}|D_s\psi_R|^2|u_{\varepsilon,n}|^2\,{\rm d}z\Big)^{\frac{1}{2}}\leq \frac{C}{\rho^s}\|u_{\varepsilon,n}\|_2\leq\frac{C}{\rho^s}.
\end{align*}
Therefore, from the estimates above, we obtain
\begin{equation*}
\int_{\mathbb{R}^3\backslash B_{2\rho}(0)}(|D_su_{\varepsilon,n}|^2+V(\varepsilon z)|u_{\varepsilon,n}|^2)\,{\rm d}z\leq\frac{C}{\rho^s}.
\end{equation*}
Thus, the claim follows. From \eqref{equ4-27}, we see that $u_{\varepsilon,n}\rightarrow u_{\varepsilon}$ in $L^2(\mathbb{R}^3)$.

Next, we claim that $u_{\varepsilon,n}\rightarrow u_{\varepsilon}$ in $L^{2_s^{\ast}}(\mathbb{R}^3)$ as $n\rightarrow\infty$. Indeed, from Lemma \ref{lem2-2}, we may assume that
\begin{equation*}
|D_su_{\varepsilon,n}|^2\rightharpoonup\mu\,\,\text{and}\,\, (u_{\varepsilon,n})^{2_s^{\ast}}\rightharpoonup\nu\,\, \text{weakly-}\ast\,\, \text{in}\,\, \mathcal{M}(\mathbb{R}^3)\,\,\text{as}\,\, n\rightarrow\infty
\end{equation*}
and there exist a (at most countable) set of distinct points $\{x_j\}_{j\in J}\subset\mathbb{R}^3$, $\mu_j\geq0$, $\nu_j\geq0$ with $\mu_j+\nu_j>0$ ($j\in J$) such that
\begin{equation}\label{equ4-28}
\mu\geq|D_su_{\varepsilon}|^2+\sum_{j\in J}\mu_j\delta_{x_j},\quad \nu=(u_{\varepsilon})^{2_s^{\ast}}+\sum_{j\in J}\nu_j\delta_{x_j},\quad\nu_j\leq(\mathcal{S}_s^{-1}\mu_j)^{\frac{2_s^{\ast}}{2}}\,\, \text{for any}\,\, j\in J.
\end{equation}
We are going to show that $J={\O}$. Suppose by contradiction that $J\neq{\O}$, i.e., there exists $x_{j_0}\in \mathbb{R}^3$ for some $j_0\in J$. Similar to the arguments in Proposition \ref{pro2-3}, we get $\nu_{j_0}\geq\mathcal{S}_s^{\frac{3}{2s}}$. On the other hand, since $\{u_{\varepsilon,n}\}\subset \mathcal{N}_{\varepsilon}^{d_0}$, by the definition of $\mathcal{N}_{\varepsilon}^{d_0}$, there exist $\{W_n\}\subset \mathcal{L}_{V_0}$, $\{x_n\}_{n=1}^{\infty}\subset\mathcal{M}^{\beta}$ such that
\begin{equation*}
\|u_{\varepsilon,n}-\varphi_{\varepsilon}(\cdot-\frac{x_n}{\varepsilon})W_n(\cdot-\frac{x_n}{\varepsilon})\|_{H_{\varepsilon}}\leq \frac{3}{2}d_0.
\end{equation*}
Since $\mathcal{L}_{V_0}$ and $\mathcal{M}^{\beta}$ are compact, there exist $W_0\in\mathcal{L}_{V_0}$, $x'\in\mathcal{M}^{\beta}$ such that $W_n\rightarrow W_0$ in $H^s(\mathbb{R}^3)$ and $x_n\rightarrow x'$ as $n\rightarrow\infty$. Thus, for $\varepsilon>0$ small,
\begin{equation}\label{equ4-29}
\|u_{\varepsilon,n}-\varphi_{\varepsilon}(\cdot-\frac{x'}{\varepsilon})W_0(\cdot-\frac{x'}{\varepsilon})\|_{H_{\varepsilon}}\leq2d_0.
\end{equation}
It follows from \eqref{equ4-28}, \eqref{equ4-29}, $(f_3)$, $v_{j_0}\geq\mathcal{S}_s^{\frac{3}{2s}}$ and $W_0\in\mathcal{L}_{V_0}$ that
\begin{align*}
\widetilde{\mathcal{C}}_{\varepsilon}+\varepsilon&\geq \mathcal{J}_{\varepsilon}(u_{\varepsilon,n})=\mathcal{J}_{\varepsilon}(u_{\varepsilon,n})-\frac{1}{q(s+t)-3}\mathcal{G}_{V_0}(W_0)\\
&\geq\frac{(q-4)s+(q-2)t}{2(q(s+t)-3)}\int_{\mathbb{R}^3}|D_su_{\varepsilon,n}|^2\,{\rm d}z\\
&+\frac{4s+2t-3}{2(q(s+t)-3)}\Big(\int_{\mathbb{R}^3}|D_su_{\varepsilon,n}|^2\,{\rm d}z-\int_{\mathbb{R}^3}|D_sW_0|^2\,{\rm d}z\Big)\\
&+\frac{2s+2t-3}{2(q(s+t)-3)}\Big(\int_{\mathbb{R}^3}V(\varepsilon z)u_{\varepsilon,n}^2\,{\rm d}z-\int_{\mathbb{R}^3}V_0W_0^2\,{\rm d}z\Big)\\
&+\frac{4s+2t-3}{4(q(s+t)-3)}\Big(\int_{\mathbb{R}^3}\phi_{u_{\varepsilon,n}}^tu_{\varepsilon,n}^2\,{\rm d}z-\int_{\mathbb{R}^3}\phi_{W_0}^tW_0^2\,{\rm d}z\Big)\\
&+\frac{(2_s^{\ast}-q)(s+t)}{2_s^{\ast}(q(s+t)-3)}\int_{\mathbb{R}^3}u_{\varepsilon,n}^{2_s^{\ast}}\,{\rm d}z+\frac{2_s^{\ast}(s+t)-3}{2_s^{\ast}(q(s+t)-3)}\Big(\int_{\mathbb{R}^3}u_{\varepsilon,n}^{2_s^{\ast}}\,{\rm d}z-\int_{\mathbb{R}^3}W_0^{2_s^{\ast}}\,{\rm d}z\Big)\\
&-\Big(\int_{\mathbb{R}^3}F(u_{\varepsilon,n})\,{\rm d}z-\int_{\mathbb{R}^3}F(W_0)\,{\rm d}z\Big)\\
&\geq\frac{(q-4)s+(q-2)t}{2(q(s+t)-3)}\int_{\mathbb{R}^3}|D_su_{\varepsilon,n}|^2\,{\rm d}z+\frac{(2_s^{\ast}-q)(s+t)}{2_s^{\ast}(q(s+t)-3)}\int_{\mathbb{R}^3}u_{\varepsilon,n}^{2_s^{\ast}}\,{\rm d}z-Cd_0+o(1)\\
&\geq\frac{(q-4)s+(q-2)t}{2(q(s+t)-3)}\mu_{j_0}+\frac{(2_s^{\ast}-q)(s+t)}{2_s^{\ast}(q(s+t)-3)}\nu_{j_0}-Cd_0+o(1)\\
&\geq\frac{(q-4)s+(q-2)t}{2(q(s+t)-3)}\mathcal{S}_s\nu_{j_0}^{2/2_s^{\ast}}+\frac{(2_s^{\ast}-q)(s+t)}{2_s^{\ast}(q(s+t)-3)}\nu_{j_0}-Cd_0+o(1)\\
&\geq\Big(\frac{(q-4)s+(q-2)t}{2(q(s+t)-3)}+\frac{(2_s^{\ast}-q)(s+t)}{2_s^{\ast}(q(s+t)-3)}\Big)\mathcal{S}_s^{\frac{3}{2s}}-Cd_0+o(1)\\
&=\frac{s}{3}\mathcal{S}_s^{\frac{3}{2s}}-Cd_0+o(1)
\end{align*}
where $o(1)\rightarrow0$ as $\varepsilon\rightarrow0$. Taking $\varepsilon\rightarrow0$ and $d_0\rightarrow0$, we have that $c_{V_0}\geq\frac{s}{3}\mathcal{S}_s^{\frac{3}{2s}}$, contradicts with Lemma \ref{lem3-4}. Therefore, $u_{\varepsilon,n}^{+}\rightarrow u_{\varepsilon}^{+}$ in $L^{2_s^{\ast}}(\mathbb{R}^3)$. Together with $u_{\varepsilon,n}\rightarrow u_{\varepsilon}$ in $L^2(\mathbb{R}^3)$, Lebesgue Dominated Convergence Theorem implies that for $\varepsilon>0$ small,
\begin{equation}\label{equ4-30}
\int_{\mathbb{R}^3}g(\varepsilon z,u_{\varepsilon,n})u_{\varepsilon,n}\,{\rm d}z\rightarrow\int_{\mathbb{R}^3}g(\varepsilon z,u_{\varepsilon})u_{\varepsilon}\,{\rm d}z,\quad n\rightarrow\infty.
\end{equation}
From \eqref{equ4-26} and $\mathcal{J}_{\varepsilon}'(u_{\varepsilon,n})\rightarrow0$ as $n\rightarrow\infty$, we get that
\begin{align}\label{equ4-31}
\int_{\mathbb{R}^3}|D_su_{\varepsilon}|^2\,{\rm d}z&+\int_{\mathbb{R}^3}V(\varepsilon z)|u_{\varepsilon}|^2\,{\rm d}z+\int_{\mathbb{R}^3}\phi_{u_{\varepsilon}}^t|u_{\varepsilon}|^2\,{\rm d}z\nonumber\\
&+4\mu_{\varepsilon}\int_{\mathbb{R}^3}\chi_{\mathbb{R}^3\backslash\Lambda_{\varepsilon}}|u_{\varepsilon}|^2\,{\rm d}z=\int_{\mathbb{R}^3}g(\varepsilon z,u_{\varepsilon})u_{\varepsilon}\,{\rm d}z
\end{align}
and
\begin{align}\label{equ4-32}
\int_{\mathbb{R}^3}|D_su_{\varepsilon,n}|^2\,{\rm d}z&+\int_{\mathbb{R}^3}V(\varepsilon z)|u_{\varepsilon,n}|^2\,{\rm d}z+\int_{\mathbb{R}^3}\phi_{u_{\varepsilon,n}}^t|u_{\varepsilon,n}|^2\,{\rm d}z\nonumber\\
&+4\mu_{\varepsilon,n}\int_{\mathbb{R}^3}\chi_{\mathbb{R}^3\backslash\Lambda_{\varepsilon}}|u_{\varepsilon,n}|^2\,{\rm d}z=\int_{\mathbb{R}^3}g(\varepsilon z,u_{\varepsilon,n})u_{\varepsilon,n}\,{\rm d}z+o(1),
\end{align}
where $o(1)\rightarrow0$ as $n\rightarrow\infty$. From \eqref{equ4-30}, \eqref{equ4-31} and \eqref{equ4-32}, we can deduce that
\begin{equation*}
u_{\varepsilon,n}\rightarrow u_{\varepsilon} \quad \text{in}\,\, H_{\varepsilon}\quad \text{and}\quad \mu_{\varepsilon}=\Big(\int_{\mathbb{R}^3\backslash\Lambda_{\varepsilon}}u_{\varepsilon}^2\,{\rm d}z-\varepsilon\Big)_{+}.
\end{equation*}
Since $0\not\in\mathcal{N}_{\varepsilon}^{d_0}$, $u_{\varepsilon}\neq0$ and $u_{\varepsilon}\in\mathcal{N}_{\varepsilon}^{d_0}\cap\mathcal{J}_{\varepsilon}^{\mathcal{D}_{\varepsilon}+\varepsilon}$. The proof is completed.
\end{proof}

{\bf Proof of Theorem \ref{thm1-1}.} Using Lemma \ref{lem4-5} and by similar arguments as the proof the Theorem 1.1 in \cite{Teng4}, we can complete the proof of Theorem \ref{thm1-1}.

{\bf Acknowledgements.}
The work is supported by NSFC grant 11501403.


\begin{thebibliography}{10}
\bibitem{AM}C. O. Alves, O. H. Miyagaki, Existence and concentration of solution for a class of fractional elliptic equation in $\mathbb{R}^N$ via penalization method, Calc. Var. Partial Differential Equations (2016) 55: 19.
\bibitem{AR}A. Ambrosetti and D. Ruiz, Multiple bound states for the Schr\"{o}dinger-Poisson problem, Comm. Contemp. Math. 10 (2008) 391--404.
\bibitem{A}V. Ambrosio, Multiplicity of positive solutions for a class of fractional Schr\"{o}dinger equations via penalization method, Annali di Matematica Pura e Applicata, 196 (2017) 2043--2062.
\bibitem{B}V. I. Bogachev, Measure Theory, Vol. II. Springer-Verlag: Berlin, 2007.
\bibitem{BCPS}C. Brandle, E. Colorado, A. de Pablo and U. S\'{a}nchez, A concave-convex elliptic problem involving the fractional Laplacian, Proc. Roy. Soc. Edinburgh Sect. A 143
(2013) 39--71.
\bibitem{BF}V. Benci and D. Fortunato, An eigenvalue problem for the Schr\"{o}dinger-Maxwell equations, Top. Methods. Nonlinear Anal. 11 (1998) 283--293.
\bibitem{BJ}J. Byeon and L. Jeanjean, Standing waves for nonlinear Schr\"{o}dinger equations with a general nonlinearity, Arch. Ration. Mech. Anal. 185 (2007) 185--200.
\bibitem{BJ1}J. Byeon and L. Jeanjean, Multi-peak standing waves for nonlinear Schr\"{o}dinger equations with a general nonlinearity, Discrete Contin. Dyn. Syst. 19 (2007)
255--269.
\bibitem{BV}C. Bucur and E. Valdinoci, Nonlocal diffusion and applications, arXiv:1504.08292v1.
\bibitem{BW}J. Byeon and Z. Q. Wang, Standing waves witha criticak frequency for nonlinear Schr\"{o}dinger equations II, Calc. Var. Partial Differential Equations, 18 (2003) 207--219.
\bibitem{CKT}A. Cotsiolis and N. K. Tavoularis, Best constants for Sobolev inequalities for higher order fractional derivatives, J. Math. Anal. Appl. 295 (2004) 225--236.
\bibitem{CM}S. Y. A. Chang and M. del Mar Gonz\'{a}lez, Fractional Laplacian in conformal geometry, Adv. Math. 226 (2011) 1410--1432.
\bibitem{CS}L. Caffarelli and L. Silvestre, An extension problem related to the fractional Laplacian, Comm. Partial Differential Equations 32 (2007) 1245--1260.
\bibitem{CSi}X. Cabr\'{e} and Y. Sire, Nonlinear equations for fractional Laplacians, I: Regularity, maximum principles, and Hamiltonian estimates, Ann. I. H. Poincar\'{e}-AN 31 (2014) 23--53.
\bibitem{CT}R. Cont and P. Tankov, Financial modeling with jump processes, Chapman Hall/CRC Financial Mathematics Series, Boca Raton, 2004.
\bibitem{CW}X. Chang and Z. Wang, Ground state of scalar field equations involving a fractional Laplacian with general nonlinearity, Nonlinearity 26 (2013) 479--494.
\bibitem{DF}M. del Pino and P. L. Felmer, Local mountain pass for semilinear elliptic problems in unbounded domains, Calc. Var. Partial Differential Equations 4 (1996) 121--137.
\bibitem{DMV}S. Dipierro, M. Medina and E. Valdinoci, Fractional elliptic problems with critical growth in the whole of $\mathbb{R}^n$, arXiv:1506.01748v1.
\bibitem{DPDV}J. Davila, M. del Pino, S. Dipierro and E. Valdinoci, Concentration phenomena for the nonlocal Schr\"{o}dinger equation with Dirichlet datum,
Anal. PDE, 8 (2015) 1165--1235.
\bibitem{DPW}J. D\'{a}vila, M. Del Pino and J. C. Wei, Concentrating standing waves for fractional nonlinear Schr\"{o}dinger equation, J. Differerntial Equations, 256 (2014) 858--892.
\bibitem{DW}T. D'Aprile and J. C. Wei, On bound states concentrating on spheres for the Maxwell-Schr\"{o}dinger equation, SIAM J. Math. Anal. 37 (2005) 321--342.
\bibitem{FL}R. Frank and E. Lenzmann, Uniqueness of ground states for fractional Laplacians in $\mathbb{R}$, Acta Math. 210 (2013) 261--318.
\bibitem{FLS}R. L. Frank, E. Lenzmann and L. Silvestre, Uniqueness of radial solutions for the fractional Laplacian, Commu. Pure Appl. Math. LXIX (2016) 1671--1726.
\bibitem{FQT}P. Felmer, A. Quaas and J. Tan, Positive solutions of nonlinear Schr\"{o}dinger equation with the fractional Laplacian, Proc. Royal Soc. Edinburgh A 142 (2012) 1237--1262.
\bibitem{GT}D. Gilbarg and N. S. Trudinger, Elliptic partial equations of second order, Springer-Verlag, New York, 1998.
\bibitem{H}X. He, Multiplicity and concentration of positive solutions for the Schr\"{o}dinger-Poisson equations, Z. Angew. Math. Phys. 62 (2011) 869--889.
\bibitem{HZ}X. He and W. Zou, Existence and concentration of ground states for Schr\"{o}dinger-Poisson equations with critical growth, J. Math. Phy. 53 (2012) 023702.
\bibitem{HZ1}X. He and W. Zou, Existence and concentration result for the fractional Schr\"{o}dinger equations with critical nonlinearities, Calc. Var. Partial Differential Equations (2016) 55: 91.
\bibitem{HL}Y. He and G. B. Li, Standing waves for a class of Schr\"{o}dinger-Poisson equations in $\mathbb{R}^3$ involving critical Sobolev exponents, Ann. Acad. Sci. Fenn. Math. 40 (2015) 729--766.
\bibitem{IV1}I. Ianni and G. Vaira, Solutions of the Schr\"{o}dinger-Poisson problem concentrating on spheres, Part I: Necessary conditions, Math. Models Meth. Appl. Sci. 19 (2009) 707--720.
\bibitem{IV}I. Ianni and G. Vaira, On concentration of positive bound states for the Schr\"{o}dinger-Poisson problem with potentials, Adv. Nonlinear Stud. 8 (2008) 573--595.
\bibitem{JLX}T. Jin, Y. Li and J. Xiong, On a fractional Nirenberg problem, part I: blow up analysis and compactness of solutions, J. Eur. Math. Soc. 16 (2014) 1111--1171.
\bibitem{La}N. Laskin, Fractional quantum mechanics and L\'{e}vy path integrals, Physics Letters A 268 (2000) 298--305.
\bibitem{La1}N. Laskin, Fractional Schr\"{o}dinger equation, Physical Review 66 (2002)  56--108.
\bibitem{LM}J. L. Lions and E. Magenes, Non-homogeneous boundary value problems and applications. Vol. I. Translated from the French by P. Kenneth. Die Grundlehren der mathematischen Wissenschaften, Band 181. Springer-Verlag, New York-Heidelberg, 1972.
\bibitem{LGF}Z. Liu, S. Guo and Y. Fang, Multiple semiclassical states for coupled Schr\"{o}dinger-Poisson equations with critical exponential growth, J. Math. Phys. 56 (2015) 041505.
\bibitem{LS}E. H. Lieb and B. Simon, The Thomas-Fermi theory of atoms, molecules and solids, Adv. Math. 23 (1977) 22--116.
\bibitem{LZ}Z. S. Liu and J. J. Zhang, Multiplicity and concentration of positive solutions for the fractional Schr\"{o}dinger-Poisson systems with critical growth, ESAIM: Control, Optim. Calc. Var., DOI: 10.1051/cocv/2016063, (2016).
\bibitem{MK}R. Metzler and J. Klafter, The random walks guide to anomalous diffusion: a fractional dynamics approach, Phys. Rep. 339 (2000) 1--77.
\bibitem{MS}E. G. Murcia and G. Siciliano, Positive semiclassical states for a fractional Schr\"{o}dinger-Poisson system, arXiv:1601.00485v1.
\bibitem{MRS}P. Markowich, C. Ringhofer and C. Schmeiser, Semiconductor Equations, Springer-Verlag, Vienna, 1990.
\bibitem{NPV}E. Di Nezza, G. Palatucci and E. Valdinoci, Hitchhiker's guide to the fractional sobolev spaces, Bulletin des Sciences Mathematiques
136 (2012) 521--573.
\bibitem{PP}G. Palatucci and A. Pisante, Improved Sobolev embeddings,profile decomposition, and concentration-compactness for fractional Sobolev spaces,
Calc. Var. Partial Differential Equations 50 (2014) 799--829.
\bibitem{RBL}R. Benguria, H. Br\'{e}zis and E. H. Lieb, The Thomas-Fermi-von Weizs\"{a}cker theory of atoms and molecules, Comm. Math. Phys. 79 (1981) 167--180.
\bibitem{Ruiz}D. Ruiz, The Schr\"{o}dinger-Poisson equation under the effect of a nonlinear local term, J. Func. Anal. 237 (2006) 655--674.
\bibitem{Ruiz1}D. Ruiz, Semiclassical states for coupled Schr\"{o}dinger-Maxwell equations: Concentration around a sphere, Math. Models Methods Appl. Sci. 15 (2005) 141--164.
\bibitem{RV}D. Ruiz and G. Vaira, Cluster solutions for the Schr\"{o}dinger-Poinsson-Slater problem around a local minimum of potential, Rev. Mat. Iberoamericana 27 (2011) 253--271.
\bibitem{S}L. Silvestre, Regularity of the obstacle problem for a fractional power of the Laplace operator, Comm. Pure
Appl. Math. 60 (2007) 67--112.
\bibitem{SZ}X. D. Shang and J. H. Zhang, Ground states for fractional Schr\"{o}dinger equations with critical growth, Nonlinearity 27 (2014) 187--207.
\bibitem{SV}R. Servadei and E. Valdinoci, The Brezis--Nirenberg result for the fractional Laplacian, Trans. Amer. Math. Soc. 367 (2015) 67--102.
\bibitem{Teng-1}K. M. Teng, Multiple solutions for a class of fractional Schr\"{o}dinger equation in $\mathbb{R}^N$, Nonlinear Anal. Real World Appl. 21 (2015) 76--86.
\bibitem{Teng-2}K. M. Teng and X. M. He, Ground state solution for fractional Schr\"{o}dinger equations with critical Sobolev exponent, Commu. Pure Appl. Anal. 15 (2016) 991--1008.
\bibitem{Teng}K. M. Teng, Existence of ground state solutions for the nonlinear fractional Schr\"{o}dinger-Poisson system with critical Sobolev exponent, J. Differential Equations 261 (2016) 3061--3106.
\bibitem{Teng1}K. M. Teng, Ground state solutions for the nonlinear fractional Schr\"{o}dinger-Poisson system, Applicable Analysis, (2018) doi.org/10.1080/00036811.2018.1441998.
\bibitem{Teng2}K. M. Teng and R. P. Agarwal, Existence and concentration of positive ground state solutions for nonlinear fractional Schr\"{o}dinger-Poisson system with critical growth, Math. Meth. Appl. Sci. 41 (2018) 8258--8293.
\bibitem{Teng3}K. M. Teng, Concentrating bounded states for fractional Schr\"{o}dinger-Poisson system involving critical Sobolev exponent, arXiv:1906.10802.
\bibitem{Teng4}K. M. Teng, Concentrating bounded states for fractional Schr\"{o}dinger-Poisson system, arXiv:1710.03495.

\bibitem{WTXZ}J. Wang, L. X. Tian, J. X. Xu and F. B. Zhang, Existence and concentration of positive solutions for semilinear Schr\"{o}dinger-Poisson systems in $\mathbb{R}^3$, Calc. Var. Partial Differential Equations, 48 (2013) 243--273.
\bibitem{WTXZ1}J. Wang, L. X. Tian, J. X. Xu and F. B. Zhang, Existence of multiple positive solutions for Schr\"{o}dinger-Poisson systems with critical growth, Z. Angew. Math. Phys. 66 (2015) 2441--2471.
\bibitem{YZZ}Y. Yu, F. Zhao and L. Zhao, The concentration behavior of ground state solutions for a fractional Schr\"{o}dinger-Poisson system,  Calc. Var. Partial Differential Equations (2017) 56: 116.
\bibitem{Z}J. Zhang, The existence and concentration of positive solutions for a nonlinear Schr\"{o}dinger-Poisson system with critical growth, J. Math. Phys. 55 (2014) 031507.
\bibitem{ZDS}J. Zhang, J. M. DO \'{O} and M. Squassina, Fractional Schr\"{o}dinger-Poisson system with a general subcritical or critical nonlinearity, Adv. Nonlinear Stud. 16 (2016) 15--30.
\bibitem{ZCZ}J. Zhang, Z. Chen, W. Zou, Standing waves for nonlinear Schr¡§odinger equations involving critical growth, J. London Math. Soc. 90 (2014) 827--844.
\bibitem{ZLZ}L. G. Zhao H. D. Liu and F. K. Zhao, Existence and concentration of solutions for the Schr\"{o}dinger-Poisson equations with steep well potential, J. Differential Equations. 255 (2013) 1--23.
\bibitem{ZZX}X. Zhang, B. L. Zhang and M. Q. Xiang, Ground states for fractional Schr\"{o}dinger equations involving a critical nonlinearity, Adv. Nonlinear Anal. 5 (2016) 293--314.





\end{thebibliography}
\end{document}